\documentclass[11pt, a4paper]{article}

\usepackage{versions}

\includeversion{ArXiv}
\excludeversion{journal}

\usepackage{palatino}
\usepackage[mathcal]{euler}

\usepackage{amsmath,amsthm,amssymb,amsfonts,latexsym,amscd,enumerate,mathtools}
\usepackage{microtype,fancyvrb}
\RequirePackage[pdftex,dvipsnames,usenames]{xcolor}
\usepackage[colorlinks]{hyperref}
\hypersetup{plainpages=false,
    colorlinks,linkcolor=RoyalBlue,citecolor=RoyalBlue,
    urlcolor=magenta,filecolor=magenta,
    breaklinks,bookmarksopen}
    
\usepackage[capitalise,nameinlink,noabbrev]{cleveref}
\crefname{equation}{}{}
\crefname{enumi}{}{}
\let\eqref\cref

\theoremstyle{plain}
    \newtheorem{theorem}{Theorem}[section]
    \newtheorem{lemma}[theorem]{Lemma}
    \newtheorem{corollary}[theorem]{Corollary}
    \newtheorem{proposition}[theorem]{Proposition}

\theoremstyle{definition}
    \newtheorem{definition}[theorem]{Definition}
    
    \newtheorem{example}[theorem]{Example}

    \newtheorem{remark}[theorem]{Remark}

\theoremstyle{remark}

\usepackage{needspace}
\makeatletter
    \let\OLD@begintheorm\@begintheorem
    \def\@begintheorem{\needspace{2\baselineskip}\OLD@begintheorm}
\makeatother

\numberwithin{equation}{section}

\newcommand{\Bb}[1]{% define symbol #1#1
  \typeout{**** Command #1#1 denotes mathbb #1 for vector spaces}%
  \expandafter\DeclareRobustCommand\csname #1#1\endcsname%
  {\ensuremath{\mathbb #1}% record when used
  \global\csname UsedBb#1true\endcsname}%
  \expandafter\newif\csname ifUsedBb#1\endcsname
  \csname UsedBb#1false\endcsname
  \AtEndDocument{\csname ifUsedBb#1\endcsname\else
  \typeout{**** Bb symbol #1#1 not used.}
  \fi}
  }
\Bb M \Bb E \Bb D

    \newcommand{\R}{\mathbb{R}}
    \newcommand{\C}{\mathbb{C}} 
    \newcommand{\N}{\mathbb{N}}
    \newcommand{\Z}{\mathbb{Z}}

\DeclareMathOperator{\dom}{dom}
\DeclareMathOperator{\id}{id}

\newcommand{\Cal}[1]{% define symbol c#1
  \typeout{**** Command c#1 denotes calligraphic #1}%
  \expandafter\DeclareRobustCommand\csname c#1\endcsname%
  {\ensuremath{\mathcal #1}% record when used
  \global\csname UsedCal#1true\endcsname}%
  \expandafter\newif\csname ifUsedCal#1\endcsname
  \csname UsedCal#1false\endcsname
  \AtEndDocument{\csname ifUsedCal#1\endcsname\else
  \typeout{**** Cal symbol c#1 not used.}
  \fi}
  }
\Cal M
\Cal B
\Cal K
\Cal P
\Cal E
\Cal F
\Cal C
\Cal A
\Cal S
\Cal O

\allowdisplaybreaks

\DeclareMathOperator{\Span}{span}

\DeclareMathOperator{\Pol}{Pol}

\DeclareMathOperator{\sgn}{sgn}

\newcommand{\Real}{\Re}

\newcommand{\beq}[1]{\begin{equation} \label{#1}}
\newcommand{\eeq}{\end{equation}}

\newcommand{\ODE}{\textsc{ode}}\let\ode\ODE
\newcommand{\PDE}{\textsc{pde}}\let\pde\PDE

\begin{document}   

\begin{ArXiv}
\title{Normal forms and invariant manifolds for nonlinear, non-autonomous PDEs, viewed as ODEs in infinite dimensions}  
\end{ArXiv}

\begin{journal}
\title[Normal forms and invariant manifolds]{Normal forms and invariant manifolds for nonlinear, non-autonomous PDEs, viewed as ODEs in infinite dimensions}  
\end{journal}

\date{\today}

\begin{ArXiv}
\author{Peter Hochs\thanks{School of Mathematical Sciences, University of Adelaide, \protect\url{peter.hochs@adelaide.edu.au}}
\and
A.J. Roberts\thanks{School of Mathematical Sciences, University of Adelaide, \protect\url{mailto:anthony.roberts@adelaide.edu.au},
\protect\url{http://orcid.org/0000-0001-8930-1552}
}}  
\end{ArXiv}

\begin{journal}
\author{Peter Hochs}
\address{School of Mathematical Sciences, University of Adelaide}
\email{peter.hochs@adelaide.edu.au}
\author{A.J. Roberts}
\address{School of Mathematical Sciences, University of Adelaide}
\email{\protect\url{mailto:anthony.roberts@adelaide.edu.au}}
\urladdr{\protect\url{http://orcid.org/0000-0001-8930-1552}}
\end{journal}

\maketitle

\begin{abstract}
We prove that a general class of nonlinear, non-autonomous %\PDE{}s and 
\ODE s in Fr\'echet spaces are close to \ODE s in a specific normal form, where closeness means that  solutions of the normal form \ODE\ satisfy the original \ODE\ up to a residual that vanishes up to any desired order.  
%
%can be brought into an approximate normal form via a coordinate transformation. This normal form is approximate in the sense that transformed solutions of the transformed equation satisfy the original equation modulo a residual that vanishes up to a specified order.  
%
In this normal form, the centre, stable and unstable coordinates of the \ODE\ are clearly separated, which allows us to define invariant manifolds of such equations in a robust way.
In particular, our method empowers us to study approximate centre manifolds, given by solutions of \ODE s that are central up to a desired, possibly nonzero precision. 
The main motivation is the case where the Fr\'echet space in question is a suitable function space, and the maps involved in an \ODE\ in this space are defined in terms of derivatives of the functions, so that the infinite-dimensional \ODE\ is a finite-dimensional \PDE. 
We show that our methods apply to a relevant class of nonlinear, non-autonomous \PDE s in this way.
\end{abstract}

\setcounter{tocdepth}{2}
\tableofcontents

\section{Introduction}

\subsection{Background and motivation}

Various invariant manifolds are central to many areas of dynamical systems, including using centre manifolds to construct and justify  reduced low-dimensional models of high-dimensional dynamics \cite[e.g.]{Roberts2014a}.
Many dynamical systems involve \pde{}s in infinite dimensional state spaces of functions, % of \pde{}s \cite{??}, 
and some applications require infinite dimensional centre manifolds \cite[e.g.]{Bunder2018a, Roberts2013a}.
In general we also want to cater for non-autonomous systems, with an aim to subsequently generalise to stochastic\slash rough dynamics \cite[e.g.]{Friz2014}.
Further, encompassing unstable dynamics with both centre and stable is necessary for application to Saint Venant-like, cylindrical, problems \cite[e.g.]{Haragus95b, Mielke88b, Mielke91a}, and to deriving boundary conditions for approximate \pde{}s \cite[e.g.]{Roberts92c}.
Consequently, here we address the general challenge of constructing and justifying various infinite dimensional invariant manifolds for non-autonomous dynamical systems which have stable, unstable and centre modes.
A crucial novel feature of the approach is that we further develop a backward theory recently initiated for finite dimensional systems~\cite{Roberts2018a}: analogous backward theory has been very useful in other domains~\cite[e.g.]{Grcar2011}.

Applications of the extant forward theory in such a general setting is often confounded by impractical preconditions.
Non-autonomous invariant manifold theories typically require bounded operators, and Lipschitz and/or uniformly bounded nonlinearities, \cite[e.g.]{Aulbach96, Aulbach99, Aulbach2000, Chicone97, Haragus2011}.
The extant boundedness requirement \cite[Hypothesis~2.1(i) and 3.8(i), e.g.]{Haragus2011} arises from the general necessity of both forward and backward time convolutions with the semigroup (e.g., \(e^{At}\) for systems that linearise to \(\dot x=Ax\)), convolutions that must be continuous in extant forward theory, but cannot be continuous with unbounded operators.
Despite many interesting specific scenarios having rigorous invariant manifolds  established via strongly continuous semigroup operators and by mollifying nonlinearity \cite[e.g.]{Carr81, Vanderbauwhede89}, extant \emph{non-autonomous} forward theory fails to rigorously apply in many practical~cases.

Our main motivation for studying \ODE s in infinite-dimensional vector spaces is their possible application to analysing invariant manifolds of \PDE s in finite space-time dimensions. In that setting, the infinite-dimensional vector space in question is a space of functions, and the maps occurring in the \ODE\ are differential operators. The linear part of such an \ODE\ is a linear partial differential operator, which typically is unbounded in applications. 
%This is the key reason why we include unbounded operators. 
Such an operator can be viewed as a bounded operator between \emph{different} Banach spaces, with norms adapted to make the operator bounded. (For example, the operator $\frac{d}{dx}$ is unbounded on $L^2(\R)$, but becomes bounded if we take its domain to be a first-order Sobolev space.) Centre manifold theory in this setting was developed by several authors \cite[e.g.]{Haragus2011, Mielke88, Vanderbauwhede92}, and applied to \PDE s. 

However, to achieve our goal of developing the desired backward theory, and robustly constructing invariant manifolds via coordinate transformations to approximate normal forms, we need to go beyond this setting.
This essentially boils down to the fact that a bounded operator on a single Banach space can be iterated to yield new bounded operators, whereas this is not possible for a bounded operator between different Banach spaces. The necessity of iterating operators in our constructions leads us outside the setting of Banach spaces, to \emph{graded Fr\'echet spaces}: intersections of infinite sequences of Banach spaces connected by bounded inclusion maps. These include spaces relevant to the study of \PDE s, such as spaces of smooth functions.

%Here we avoid requiring boundedness via a proposed backward theory: for example, \cref{sec??} only requires ??.

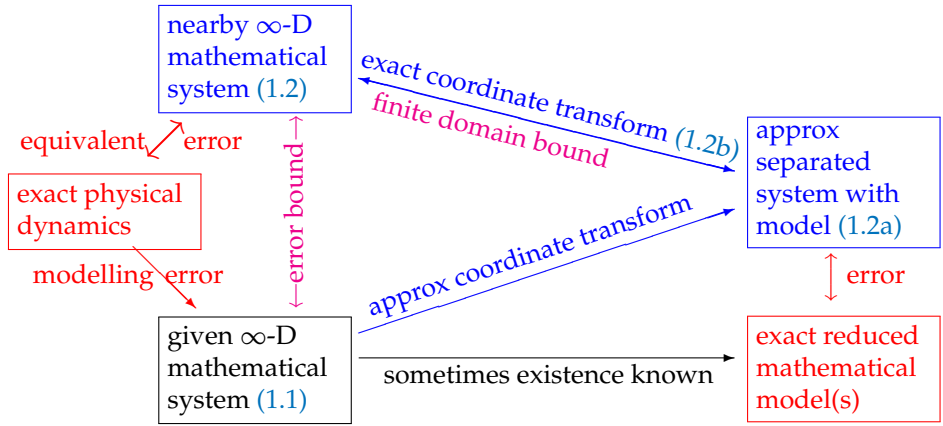
\begin{figure}
\centering\small
\setlength{\unitlength}{1ex}
\newcommand{\boxd}[1]{\fbox{\parbox[b]{14ex}{\raggedright #1}}}
\begin{picture}(75,34)
%\put(0,0){\framebox(75,34){}}
\color{red}
\put(0,15){\boxd{exact physical dynamics}}
\put(10,14){\vector(1,-1){5}\put(-13,-3){modelling \,error}}
\color{black}
\put(12,1){\boxd{given $\infty$-D mathematical system~\eqref{eq ODE intro}}}
\put(28,5){\vector(1,0){30}}
%\put(33,3){proof manifolds exist?}
\put(30,3){sometimes existence known}
\color{red}
\put(59,1){\boxd{exact reduced mathematical model(s)}}
\color{black}
\put(65,11){\color{red}{$\Big\updownarrow$} error}
\color{blue}
\put(59,15){\boxd{approx separated system with model~\eqref{eq normal form intro}}}
\put(28,7){\vector(3,1){30}}
\put(28.5,8.4){\rotatebox{18.44}{approx coordinate transform}} 
%\put(30,5.7){\rotatebox{18.44}{multinomial approximation}}
%%%%%%%%%%%%
\put(1,21.8){\color{red}equivalent{\Large$\nearrow$\llap{$\swarrow$}}error}
\put(28,27.5){\vector(4,-1){30}}
\put(58,20){\vector(-4,1){30}}
\put(28,28.5){\rotatebox{-14.04}{exact coordinate transform~\eqref{eq:xiIntro}}}%%%%%%%%%%%
\put(12,26){\boxd{nearby $\infty$-D mathematical system~\eqref{eqs:nfIntro}}}
\color{magenta}
\put(29,25){\rotatebox{-14.04}{finite domain bound}}
\put(22,8.9){\rotatebox{90}{$\leftarrow${error bound}$\to$}}
\end{picture}
\caption{\label{fig:scheme} 
schematic diagram: blue, new theory and practice established by this article; magenta, for future research; black, mostly established extant theory and practice; red, practically unattainable (in general).}
\end{figure}

\subsection{Results}

The first step in the proposed backward theory is to establish an approximate conjugacy between a given system and a `nearby' system for which we know its invariant manifolds, by its construction.
\cref{fig:scheme} illustrates what this article achieves.
Planned future research will then provide novel finite domain and error bounds as illustrated in \cref{fig:scheme}.
That is, instead of proving that there exists a reduced dimensional manifold for a specified system, which is then approximately constructed, our main results, \cref{thm normal form,cor special case}, establish that there is a system which is both `close' to the specified system, and also has a reduced dimensional manifold which we know exactly.
In essence, we invoke an (extended) normal form coordinate transform---related to Hartman--Grobman theory \cite[e.g.]{Aulbach96, Aulbach99, Aulbach2000}---and use it from a new point of view.

An intuitive formulation of our main result on the existence of such a normal form is the following \cref{thm normal form intro}. \cref{thm normal form} is a precise formulation, and \cref{cor special case} is a special case that applies to nonlinear, non-autonomous \PDE s. (In that special case, the space $V$ is a space of functions, and the maps that occur are partial differential operators.)

\begin{theorem}[Normal form theorem, intuitive formulation] \label{thm normal form intro}
Let $I$ be an interval in~\(t\), %  containing \(t=0\), 
and $V$ a possibly infinite-dimensional vector space.
Consider a nonlinear, non-autonomous \ODE
\beq{eq ODE intro}
\dot x(t) = Ax(t) + f(t, x(t))
\eeq
for $x\colon I \to V$, where $A$ is a linear operator on~$V$ (independent of~\(t\)), and $f \colon I \times V \to V$ and its derivative vanish on $I \times \{0\}$. 
For each  $p\geq 2$, there are both an \ODE 
\begin{subequations}\label{eqs:nfIntro}%
\beq{eq normal form intro}
\dot X(t) = AX(t) + F_p(t, x(t))
\eeq
for $X\colon I \to V$, and a time-dependent coordinate transformation 
\begin{equation}
\xi_p\colon I \times V \to V
\label{eq:xiIntro}
\end{equation} 
\end{subequations}
such that
\begin{itemize}
\item if $X(t)$ satisfies \eqref{eq normal form intro}, then setting $x(t) = \xi_p(t, X(t))$ defines a solution of \eqref{eq ODE intro} up to a residual term that vanishes to order $p$;
\item the map $F_p\colon I \times V \to V$ is of order $2$ in its second entry~\(X\);%, and preserves the centre, stable and unstable spaces of $A$;
\item the component of a solution to \eqref{eq normal form intro}  in the stable subspace for $A$ decays exponentially quickly to zero as  \(t\)~increases in \(I\). Its component in the unstable subspace for $A$ decays exponentially quickly to zero as  \(t\)~decreases in $I$. If the solution starts out in either the centre-stable or the centre-unstable subspace for $A$, then its component in the central subspace for $A$ is bounded by a constant for all $t \in I$, or at worst by a specified, small exponential growth rate.
%
%lies inside a suitable domain, for which $X(0)$ lies in the centre-stable subspace of $A$, decays exponentially quickly, as \(t\)~increases and \(t\geq0,t\in I\), to the centre subspace; and
%such a solution for which $X(0)$ lies in the centre-unstable subspace of $A$, decays exponentially quickly, as \(t\)~decreases and \(t\leq0,t\in I\), to the centre subspace. 
%converges exponentially to zero as $t \to \infty$; such a solution for which $X(0)$ lies in the unstable subspace of $A$ converges exponentially to zero as $t \to -\infty$.
\end{itemize}
\end{theorem}
The three most important things to make more precise in this intuitive formulation is what `order $p$' means, the related question what topology on $V$ is used, and what kinds of maps $A$, $f$, $\xi_p$ and $F_p$ are.

The last point in \cref{thm normal form intro} means that the centre, stable and unstable manifolds (in this case, linear subspaces) of \eqref{eq normal form intro} are exactly the centre, stable and unstable spaces of~$A$, respectively. (And similarly for the centre-stable and centre-unstable subspaces.) The  centre, stable, unstable, centre-stable and centre-unstable subspaces for the dynamics in $x$ described by \cref{eq normal form intro} and $x(t) = \xi_p(t, X(t))$ (which becomes an \ODE\ in $x$ if $\xi_p(t, \relbar)$ is invertible for all $t$) are then obtained from these spaces via an application of the coordinate transform $\xi_p$. In this way, we show that any (non-autonomous) system of the form \cref{eq ODE intro} is arbitrarily close to a  system with robustly defined invariant manifolds.

This definition of these key invariant manifolds  is a crucial reformation of the backward theory proposed.
Classic definitions of un\slash stable and centre manifolds require the existence of limits as time goes to~\(\pm\infty\) \cite[e.g.]{Aulbach2006, Barreira2007, Haragus2011, Potzsche2006}.
This consequently requires solutions of the dynamical system to be  well-behaved for all time, which requires constraints that in applications are often not found, or are hard to establish.
For example, in stochastic systems very rare events will eventually happen over the infinite time requiring global Lipschitz and boundedness that are oppressive in applications.
By modifying definitions we establish results for finite times, which are useful in many applications, and for a wider range of non-autonomous systems.

%\Todo: Background, motivation, literature overview.

%\subsection{Results}

\subsection{Ingredients of the proof}

The key ingredients of the proofs of our main results are \emph{nested sequences of Banach spaces}, whose intersections are \emph{graded Fr\'echet spaces}; \emph{compact polynomial maps} between Banach spaces and graded Fr\'echet spaces;  and \emph{compactly differentiable maps} between such spaces.

%\Todo: Say that in a few places, we may  develop parts of the theory slightly more generally than we strictly need to prove our theorems, but that these generalisations all seem natural and possibly applicable to other problems.

\subsubsection{Sequences of Banach spaces}

It is important to specify what is meant by a \emph{nearby} infinite-dimensional mathematical system in \cref{fig:scheme}. Intuitively, we mean by this that solutions of the nearby system \eqref{eqs:nfIntro} are solutions of the original system \eqref{eq ODE intro} up to any desired order in the magnitude of such a solution, as in \cref{thm normal form intro}. To be more precise about what that this magnitude is, we need to specify norms or seminorms on spaces containing these solutions. Working with a single Banach space (i.e., a single norm) is too restrictive for applications. This is because  in applications to {\PDE}s, the maps $A$ and~$f$ in \cref{thm normal form intro} are generally not continuous maps from a single Banach space to itself. This could be remedied by allowing maps between two different Banach spaces, but that would not let us iteratively apply maps involving $A$ and~$f$, which we do in the proof of \cref{thm normal form intro}.

%a nonlinear operator that occurs in a system of interest is not a continuous operator on a single Banach space in general. 

A type of space that is both general enough to apply to various nonlinear \PDE s, while being close enough to Banach spaces to allow us to define a meaningful notion of a solution of an equation up to a given order, is what are often called {graded Fr\'echet spaces}.
These are intersections of sequences of Banach spaces, each with a bounded inclusion map into the next. The notion of an operator of a given order  on a graded Fr\'echet space is then defined in terms of the norms on these Banach spaces, see \cref{def On}.

For several convergence questions, it would be useful if the Banach spaces that occur in the definition of a graded Fr\'echet space are Hilbert spaces. Then we can use orthogonality, for example. However, for applications to nonlinear \PDE s, it is not enough to use Hilbert spaces. For example, a nonlinear term $u \mapsto u^2$ is a well-defined (and differentiable) map from the Banach space $L^4(\R)$ to the Hilbert space $L^2(\R)$. To be able to use Hilbert space techniques in such settings, we use the notion of nested sequences of Banach space that are \emph{comparable} to nested sequences of Hilbert spaces (see \cref{def comparable}). This effectively means that a graded Fr\'echet space that naturally occurs as an intersection of Banach spaces can equivalently be presented as an intersection of Hilbert spaces. Proving such a property in situations relevant to \PDE{}s involves the relevant Sobolev embedding theorems.

%Such Fr\'echet spaces both include the ones relevant to the study of \PDE s, such as spaces of smooth functions, and are also close enough to Banach spaces to allow us to apply the relevant Banach space techniques.

\subsubsection{Compact polynomial maps}

We construct coordinate transformations on graded Fr\'echet spaces to bring \ODE s in such spaces into normal forms that allow us to define invariant manifolds directly and robustly. These transformations are polynomial maps, which we construct by adding (infinitely many) monomial terms with the right properties together. At the level of Banach spaces, a polynomial map can be naturally defined as a finite sum of  restrictions to the diagonal of  bounded multilinear maps. For example, Taylor polynomials of differentiable maps between normed vector spaces are polynomials of this type.

But not all such polynomial maps (for example, the identity map) can be approximated by sums of monomials. This leads us to define \emph{compact} polynomial maps, which can be approximated in this way in settings relevant to us. The notion of a compact polynomial map that we use seems natural, but we have not been able to find it elsewhere in the literature. Different notions of compact polynomial maps were developed and studied by Gonzalo, Jaramillo and Pe{\l}czy\'{n}sky  \cite{GJ, Pelczynski}.

\subsubsection{Compactly differentiable maps}

Our construction of the required coordinate transforms involves Taylor polynomials of differentiable maps between Banach spaces,  and between graded Fr\'echet spaces. This construction is possible if those polynomials are compact in the sense just mentioned. That is the case for \emph{compactly} differentiable maps, which we define for this purpose. We will see in \cref{sec ex f} that, in applications to \PDE s, the relevant differentiable maps are indeed compactly differentiable. This follows from various Sobolev embedding theorems. 

\subsection{Outline of this paper}

The main results of this paper, on normal forms and invariant manifolds of nonlinear, non-autonomous \ODE s in Fr\'echet spaces, and of nonlinear, non-autonomous \PDE s in finite-dimensional spaces, are stated in \cref{secPrelim}. We illustrate our results by applying them to  an example \PDE\ in \cref{sec ex}.

In the rest of the paper, we prove our main results. We start by reviewing  standard material on differentiable maps and polynomials on normed spaces in \cref{sec der}. In
\cref{sec cpt der,sec Frechet}, we develop technical tools we need for our proofs. Then in
\cref{sec co xform,sec transf choice}, we use these tools to prove the main \cref{thm normal form,thm special case}. We prove some properties of the normal form equation, which allow us to identify its invariant manifolds, in \cref{sec dynamics}.

A key ingredient in the proof of a version Taylor's theorem for compactly differentiable maps between Fr\'echet spaces, mentioned above, is the fact
that a compact operator from a Banach space with the approximation property into another Banach space can be approximated by finite-rank operators in a suitable way. 
\begin{ArXiv}
This is reviewed in 
\cref{secCompFinite}.
\end{ArXiv}

\subsection{Notation and conventions}
\label{secNota}

We write~$\N$ for the set of positive integers, and~$\N_0$ for the set of nonnegative integers. We write~$\N_0^{\infty}$ for the set of sequences in~$\N_0$ with finitely many nonzero entries, interpreted as multi-indices. For $q \in \N_0^{\infty}$, or in~$\N_0^n$, we denote the (finite) sum of its elements by~$|q|$.

We denote spaces of bounded linear operators by the letter~$\cB$, and spaces of compact linear operators by the letter~$\cK$.

When we mention a normed vector space~$V$, the implicitly given norm is denoted by~$\|\cdot \|_V$. Similarly, if~$V$ is an inner product space, then the inner product is denoted by~$(\relbar, \relbar)_V$.
Inner products on complex vector spaces are assumed to be linear in their second entries, and antilinear in their first entries. For maps $f,g\colon V \to \R$, when we write $f(v) = O(g(v))$, we implicitly mean that $f(v) = O(g(v))$ as $v \to 0$ in $V$.
% :
% unless otherwise specified, for all~\(p,x,X\) ``\(O(\|x\|^p_X)\)'' implicitly means ``as \(x\to0\) in~\(X\)''.

If~$V$ is a normed vector space, 
and $I$ is an open interval, and $f\colon I \to V$ and $f_j\colon I \to V$, for $j \in \N$, are maps, then we say that~$f_j$ converges to~$f$ if~$f_j(t)$ converges to~$f(t)$ in~$V$ uniformly in~$t$ in compact subsets of~$I$. If~$f$ and~$f_j$ are smooth, then we say that~$f_j$ converges to~$f$ differentiably in~$t$ if~$f_j^{(n)}$ converges to~$f^{(n)}$ for every $n \in \N_0$, in this sense. 

For maps $f,g\colon I \times V \to V$ and $h\colon  V \to V$, the maps $f\circ g\colon I \times V \to V$ and $f\circ h\colon I\times V \to V$ are defined by
%\ajr{Used := instead of~=}
\[
(f\circ g)(t,v):= f(t, g(t,v)), \quad
(f\circ h)(t,v):= f(t, h(v)),
\]
for $t \in I$ and $v \in V$.\label{not compos}

If~$\Omega$ is an open subset of~$\R^d$ and $m \in \N$, then the Sobolev space of functions from~$\Omega$ to~$\R^m$ with weak derivatives up to order~$k$ in~$L^p$ is denoted by~$W^{k,p}(\Omega; \R^m)$. The norm on this space is 
\[
\|u\|_{W^{k,p}} := \sum_{\alpha \in \N_0^n; |\alpha|\leq k} \left\|\frac{\partial^{\alpha} u}{\partial x^{\alpha}}\right\|_{L^p}.
\]
If $m = 1$, we write $W^{k,p}(\Omega) = W^{k,p}(\Omega; \R)$.

\section{Preliminaries and results}
\label{secPrelim}

Our main result, \cref{thm normal form}, asserts that a broad class of nonlinear \PDE{}s and \ODE s in infinite-dimensional vector spaces may be effectively approximated by normal form systems via well-chosen, time-dependent, coordinate transformations. In this normal form, the centre, stable and unstable components of the \PDE\ and \ODE\ are clearly separated, which allows us to define centre manifolds for this class of equations in a robust way (\cref{def centre mfd}).

We first state our result on normal forms, and the definition of centre manifolds, for \ODE s in a class of abstract vector spaces (\cref{sec result}). 
Our main reason for developing this theory is to apply it to the study of \PDE s, for which the vector spaces used are spaces of functions, and the relevant maps between them are defined in terms of derivatives of functions. 
We discuss a relevant class of examples of such function spaces and maps in \cref{sec special}.

%\subsection{Derivatives and polynomials}
%
%In \cref{sec der}, we review some material on differentiable and polynomial maps between normed vector spaces. Here we briefly mention some 

\subsection{Nested sequences of Banach spaces}

The normal form we obtain in \cref{thm normal form} is approximate in the sense that functions satisfying an equation transformed into that form satisfy the original equation up to a residual term. 
An important point in \cref{thm normal form} is that this residual vanishes up to a specified order. 
To make it precise what this vanishing up to a certain order means, we introduce the type of topological vector spaces we consider in this subsection. 
More details about these spaces and their properties are given in \cref{sec Frechet}. A concrete class of examples of these spaces relevant to the study of \PDE s is given in \cref{sec special}.
 
 \begin{definition} \label{def nested seq}
By a \emph{nested sequence of Banach spaces}, we mean a sequence
$\{V_k\}_{k=1}^{\infty}$  of  Banach spaces such that 
 \begin{itemize} 
 \item  for every~$k$, $V_{k+1} \subset V_k$, where the inclusion map is bounded, and
 %  (\Todo: is boundedness really necessary here? Safe to assume in any case.);
% \item for every~$k$, the spaces~$V_k$ is equipped with a
% complete norm~$\|\cdot \|_{V_k}$  such that the inclusion map $V_{k+1}\hookrightarrow V_k$ is bounded. 
 \item the intersection $V_{\infty}:=\bigcap_{l=1}^{\infty}V_l$ is dense in~$V_k$ for every~$k\in\N$.
  \end{itemize}
We then consider~$V_{\infty}$ as a Fr\'echet space\footnote{Much of what we write about Fr\'echet spaces of this form holds for more general projective limits of Banach spaces connected by bounded operators. But we do not need that degree of generality.} with the seminorms (now actual norms) that are the restrictions of the norms on the spaces~$V_k$. 

A \emph{compactly} nested sequence of Banach spaces is such a sequence such that for every $l \in \N$, there is a $k \geq l$ %\ajr{Clarify \(k\lessgtr l\)? or not constrained?} 
such that the inclusion $V_k \subset V_l$ is compact.
\end{definition}

A  Fr\'echet space $V_{\infty}$ as in this \cref{def nested seq} is often called a \emph{graded} Fr\'echet space.
 
Let $\{V_k\}_{k=1}^{\infty}$ be a nested sequence of Banach spaces. 
\begin{definition} \label{def BVinfty}
The space~$\cB(V_{\infty})$ of \emph{bounded operators on~$V_{\infty}$} consists of the linear maps $A \colon V_{\infty} \to V_{\infty}$ such that for every $l \in \N$, there is a $k \in \N$
%\ajr{Clarify \(k\lessgtr l\)? or not constrained?} 
such that the linear map~$A$ extends continuously to a map in~$\cB(V_k, V_l)$. 
\end{definition}
 \begin{remark}
 In \cref{def BVinfty}, if $k \leq l$, then the composition 
 \[
 V_l \hookrightarrow V_k \xrightarrow{A} V_l
 \]
 is a bounded operator on $V_l$. So we may always take $k \geq l$ in this context, but this does not need to be assumed a priori. Similar remarks apply in analogous situations, such as \cref{def On,def Vinfty diffble} below.
 \end{remark}
 
 \begin{definition} \label{def On}
 A map $f\colon V_{\infty} \to V_{\infty}$ is \emph{of order~$n$}, written as $f = \cO(n)$, if for every $l \in \N$, there is a $k \in \N$ %\ajr{Clarify \(k\lessgtr l\)?} 
such that $\| f(v)\|_{V_l} = O(\|v\|_{V_k}^n)$ as~$v\to 0$ in~$V_k$.
 
If~$I$ is an open interval, a map  $f\colon I \times V_{\infty} \to V_{\infty}$ is \emph{of order~$n$}, written as $f = \cO(n)$, if for every $l \in \N$, there is a $k \in \N$ such that $\| f(t, v)\|_{V_l} = O(\|v\|_{V_k}^n)$ as~$v\to 0$ in~$V_k$, uniformly in~$t$ in compact subsets of~$I$.
 \end{definition}

\begin{definition} \label{def Vinfty diffble}
An~$n$ times \emph{differentiable map from~$V_{\infty}$ to itself} is a map $f\colon V_{\infty} \to V_{\infty}$ such that for every $l \in \N$, there is a $k \in \N$ %\ajr{Clarify \(k\lessgtr l\)?} 
such that~$f$ extends to an~$n$ times differentiable map from~$V_k$ to~$V_l$. If a map is~$n$ times differentiable for every~$n \in \N$, then it is \emph{infinitely differentiable}.
\end{definition}
Basic material on differentiable maps between normed vector spaces is reviewed in \cref{sec der}.

\begin{definition}\label{def comparable}
Two nested sequences  $\{V_k\}_{k=1}^{\infty}$ and $\{W_k\}_{k=1}^{\infty}$  of Banach spaces are \emph{comparable} if 
 %they are both contained in some larger vector space~$V$, and $V_{\infty} = W_{\infty}$ as sets, and
for every $k \in \N$, there are $l_1, l_2, l_3, l_4 \in \N$ such that we have bounded inclusions \(V_{l_1} \subset W_k \subset V_{l_2}\) and
\(W_{l_3} \subset V_k \subset W_{l_4}\).
%\[
%\begin{split}
%V_{l_1} &\subset W_k \subset V_{l_2};\\
%W_{l_3} &\subset V_k \subset W_{l_4}.\\
%%V_{k} &\subset W_{l_3};\\
%%W_k &\subset V_{l_4}.\\
%\end{split}
%\]
\end{definition}
 In the setting of this definition, $V_{\infty} = W_{\infty}$.
%\ajr{May be lemma? or more simply replace "Note that" by "Consequently"?}

\subsection{Setup and goal} \label{sec setup}

Let $\{V_k\}_{k=1}^{\infty}$ be a compactly nested sequence of Banach spaces, such that~$V_1$ is a Hilbert space. Let $A \in \cB(V_{\infty})$. 
Let $I \subset \R$ be an open interval in~\(t\) containing \(t=0\), and let 
%\[
\(f\colon I \times V_{\infty} \to V_{\infty}\)
%\]
be infinitely  differentiable with respect to~$V_{\infty}$ and~$I$. Suppose that $f = \cO(2)$, and that for every $l \in \N$, there is a $k \in \N$ such that $f\colon V_k \to V_l$ is differentiable, and  
\beq{eq est der f}
\|f'_{V_{\infty}}(t,v)\|_{\cB(V_k, V_l)} = O(\|v\|_{V_k}), 
\eeq
uniformly in~$t$ in compact subsets of~$I$.

Suppose that $\{e_j\}_{j=1}^{\infty} \subset V_{\infty}$ is a set of eigenvectors of~$A$ which is a Hilbert basis of~$V_1$ (an orthonormal set that spans a dense subspace of $V_1$). 
%\ajr{Clarify ``Hilbert basis''?}
We assume that the sequence~$\{V_k\}_{k=1}^{\infty}$ is {comparable} to a nested sequence of separable Hilbert spaces  in which the vectors~$e_j$ are orthogonal. 
However, we will see in \cref{rem comparable unnecessary}  that we may equivalently make the seemingly stronger but more concrete assumption that the spaces~$V_k$ themselves are separable Hilbert spaces. 
The assumption that~$\{V_k\}_{k=1}^{\infty}$ is {comparable} to a nested sequence of separable Hilbert spaces is easier to check in practice than the condition that every space~$V_k$ can be chosen to be a separable Hilbert space itself. 
For example, \cref{sec special} discusses a class of relevant cases where the spaces~$V_k$ are not Hilbert spaces for $k \geq 2$. 
In this sense, the notion of comparable sequences of Banach spaces is a tool that makes it easier to check the conditions of \cref{thm normal form}.

%And in \cref{sec special}, we show that this condition is satisfied in relevant cases. 

We study smooth maps $x\colon I \to V_{\infty}$ satisfying the non-autonomous dynamical system differential equation
\beq{eq ODE}
\dot x(t) = Ax(t) + f(t, x(t))\quad\text{for all }t \in I\,.
\eeq
Since the nonlinearity~\(f\) satisfies~\eqref{eq est der f}, \(x=0\) is an equilibrium of the system~\eqref{eq ODE}.
We provide a novel backward approach to establish invariant manifolds in a finite domain about the equilibrium \(x=0\).
For these invariant manifolds to be useful in applications, the time interval~\(I\) will be long enough for transient dynamics to decay to insignificance in the context of the application.  
The proofs of our main results simplify considerably if the time interval~\(I\) is short, or bounded. But
%Consequently, some of our theorems look near-trivial when the time interval~\(I\) is short: 
we emphasise that we only aim this theory to support the many applications where the time interval~\(I\) is long enough, or unbounded, so that the theorems are non-trivially useful in the application.

\subsection{Dynamics in a normal form}

We define invariant manifolds, or sets, for dynamical systems in a particular normal form, and show that this definition captures the essence of such manifolds. In \cref{sec result}, we show that a very general class of {\ODE}s of the form \eqref{eq ODE} can be brought into this normal form, modulo residuals that vanish to a desired order.

\paragraph{Spectral gap in an exponential trichotomy}
Let $\alpha, \beta, \gamma,\tilde \mu$ be such that $0\leq\alpha<\tilde\mu < \min(\beta,\gamma)$, and no eigenvalues of~$A$ have real parts in the intervals~$(-\beta, -\alpha)$ and~$(\alpha, \gamma)$ (depending upon the circumstances, \(\beta\) or~\(\gamma\) could be~\(\infty\), and/or \(\alpha\)~may be zero).
For every $j \in \N$, let~$\alpha_j$ be the eigenvalue of~$A$ corresponding to~$e_j$.
With respect to the parameters~$\alpha$, $\beta$ and~$\gamma$, we define
 the sets of indices of \emph{central}, \emph{stable} and \emph{unstable} eigenvalues and eigenvectors, respectively, as
\begin{align*}
J_c &:= \{j \in \N : |\Real(\alpha_j)| \leq \alpha\};\\
J_s &:= \{j \in \N : \Real(\alpha_j) \leq -\beta\};\\
J_u &:= \{j \in \N : \Real(\alpha_j) \geq \gamma\}.
\end{align*}
%\ajr{What about ``for every \(i\in\{c,s,u\}\) let \(V_i\) be \ldots~\(e_j\) for \(j\in J_i\)''?}%
%
%If $$
%
For $a = c,s,u$, let $V_a$ be  the closure in~$V_{1}$ of the span of the eigenvectors~$e_j$, for $j \in J_a$.
%
%Let~$V_c$, $V_s$ and~$V_u$ be the closures in~$V_{1}$ of the spans of the eigenvectors~$e_j$, for $j \in J_c$, $j \in J_s$ and $j \in J_u$, respectively.
For any map~$g$ into~$V_{1}$ and $a \in \{c,s,u\}$, we write~$g_a$ for its components in~$V_a$. The sets~\(V_c,V_s,V_u\) are respectively called the \emph{centre\slash stable\slash unstable subspaces}.
Further, we define the \emph{centre-stable subspace} \(V_{cs}:=V_c\oplus V_s\), and the \emph{centre-unstable subspace} \(V_{cu}:=V_c\oplus V_u\).

%We write~$\N_0^{\infty}$ for the set of sequences in~$\N_0$ with finitely many nonzero terms. For such a sequence~$q$, we set $|q| := \sum_{j=1}^{\infty} q_j$. (Note that this sum is finite.) 
For $v \in V_{\infty}$ and a multi-index $q \in \N_0^{\infty}$, we set
\[
v^q := \prod_{j=1}^{\infty} ({e_j},v)_{V_1}^{q_j}.
\]
For $q \in \N_0^{\infty}$, write $q = q^c + q^s + q^u$, for  $q^c, q^s, q^u \in \N_0^{\infty}$ such that $q^c_j = 0$ if $j \not \in J_c$, $q^s_j = 0$ if $j \not \in J_s$ and $q^u_j = 0$ if $j \not \in J_u$.

%\begin{AJR}
%\paragraph{Spectral gap in slow-fast systems}
%\ajr{If possible I want to cater for this scenario as well since an awful lot of applications require it, and extant forward theory says almost nothing rigorous.}
%Suppose instead that no eigenvalues~\(\alpha_j\) of~$A$ have modulus~\(|\alpha_j|\) in the interval~$(\alpha,\beta)$.
%Define the \emph{slow} and \emph{fast} index sets as \(J_0:=\{j \in \N: |\alpha_j| \leq \alpha\}\) and \(J_f:=\{j \in \N: |\alpha_j| \geq \beta\}\), respectively, \(V_0\) and~\(V_f\) be the respective closures in~\(V_1\), and so on.
%\end{AJR}

\paragraph{Normal form dynamics}

\begin{definition}
A smooth map $F\colon I \times V_{\infty} \to V_{\infty}$ \emph{separates invariant subspaces} if the components of $F$ in~$V_c$, $V_s$ and~$V_u$ are of the forms
\newcommand{\subpar}[2][5em]{\parbox{#1}{\scriptsize\raggedright #2}}
\begin{subequations} \label{eq Fcsu}
\begin{align}
F_c (t,v)&= \sum_{
\subpar{$q \in \N_0^{\infty}$: $|q| \leq p$ 
and $q^s = q^u=0$
or   $q^s \not=0\not= q^u$}}
F^q(t) v^q, \label{eq sep invar c} \\
F_s (t,v)&= \sum_{
\subpar{$q \in \N_0^{\infty}$: $|q| \leq p$
and $q^s \not=0$}} 
F^q(t) v^q, \label{eq sep invar s} \\
F_u (t,v)&= \sum_{\subpar{$q \in \N_0^{\infty}$: $|q| \leq p$
and $q^u \not=0$}} 
F^q(t) v^q,\label{eq sep invar u} 
\end{align}
\end{subequations}
for all $t \in I$ and $v \in V_{\infty}$,
for smooth maps $F^q \colon I \to V_{\infty}$, where the series converge in $\Pol(V_{\infty})$, differentiably in $t$. 
\end{definition}

Consider  a polynomial map $F\colon I \times V_{\infty} \to V_{\infty}$
that separates invariant subspaces, and the \ODE
\beq{eq ODE XYZ}
\dot X(t) = AX(t) + F(t, X(t)),
\eeq
in smooth maps $X\colon I \to V_{\infty}$. Because $F$ separates invariant subspaces, this \ODE\ has very explicit invariant manifolds, by \cref{lem Vj invar} and \cref{prop dynamics} below.

\begin{lemma} \label{lem Vj invar}
Suppose that  $X\colon I \to V_{\infty}$ satisfies~\eqref{eq ODE XYZ}, where $F$ separates invariant subspaces. Let $a \in\{ c,s,u\}$. If there exists a $t \in I$ such that $X(t) \in V_a$, then $X(t) \in V_a$ for all $t \in I$.
\end{lemma}
\begin{proposition} \label{prop dynamics}
There is a neighbourhood $D_{\tilde \mu}$ of $I \times \{0\}$ in $I \times V_{\infty}$,  with the following property.
Let $X\colon I \to V_{\infty}$ be a solution of \eqref{eq ODE XYZ}, for some open interval $I$ containing $0$, and where $F$ sepearates invariant subspaces. Write $X = X_c + X_s + X_u$, with $X_a  \in V_a$ for $a = c,s,u$.
%Let $X_s$ be its component in $V_s$,  and $X_u$  its components in $V_u$. 
\begin{itemize}
\item  If $(t, X(t)) \in D_{\tilde \mu}$ for all $ t\in I$ with $t\geq0$, then for all such $t$,
 $\|X_s(t)\|_{V_1} \leq \|X_s(0)\|_{V_1} e^{-(\beta - \tilde \mu)t}$. 
\item  
If $(t, X(t)) \in D_{\tilde \mu}$ for all  $t \in I$ with $t\leq0$, then for all such $t$, $\|X_u(t)\|_{V_1} \leq \|X_u(0)\|_{V_1} e^{(\gamma -\tilde \mu)t}$.
\item  Suppose that  $X_s(0) = 0$ or $X_u(0)=0$. 
%If $(t, X(t)) \in D_{\tilde \mu}$ for all  $t \in I$ with $t>0$, then for all such $t$, $\|X_u(t)\|_{V_1} \geq \|X_s(0)\|_{V_1} e^{(\gamma -\tilde \mu)t}$. 
If $(t, X(t)) \in D_{\tilde \mu}$ for all  $t \in I$, then for all  $t \in I$, $\|X_c(t)\|_{V_1} \leq \|X_c(0)\|_{V_1} e^{(\alpha + \tilde \mu) |t|}$.
\end{itemize}
\end{proposition}
Since $\beta - \tilde \mu$ and $\gamma - \tilde \mu$ are positive, this proposition in particular states that stable solutions decrease to zero exponentially quickly as $t$~increases in~\(I\), while unstable solutions decrease to zero exponentially quickly as $t$~decreases in~\(I\). The numbers~$\alpha$ and~$\tilde \mu$ represent bounds on what one takes to be relatively small real parts of eigenvalues of~$A$ (classically, these numbers are zero), so that the third point in \cref{prop dynamics} intuitively states that central solutions, at worst, only grow relatively slowly as $|t|$~increases.

\cref{lem Vj invar,prop dynamics} are proved in \cref{sec dynamics}. The specific form of the set $D_{\tilde \mu}$ is also specified there, see \eqref{eq Dmu}.

\subsection{Invariant manifolds}

\cref{lem Vj invar,prop dynamics} show that, for every $a=c,s,u$, the set
\[
D_{\tilde \mu} \cap (I\times V_a)
\]
is a centre, stable or unstable submanifold of $I \times V_{\infty}$ for \eqref{eq ODE XYZ}, respectively. Furthermore, for $a = cs$ and $a = cu$, we obtain centre-stable and centre-unstable manifolds, respectively. (Here we use the cases of the third point in \cref{prop dynamics} where $X_u(0)=0$ and $X_s(0)=0$, respectively.)
This motivates
 \cref{def centre mfd} of invariant subspaces of dynamical systems of a certain form. 
  To state it precisely, we incorporate existence of solutions of \eqref{eq ODE XYZ}.

For $v \in V_{\infty}$, we write $a_v$ for the infimum of the set of all $a > 0$ such that there is a solution $X: (-a, 0] \to V_{\infty}$ of \eqref{eq ODE XYZ}, with $X(0)=v$. Similarly, $b_v$ is the supremum of the set of all $b > 0$ such that there is a solution $X: [0,b) \to V_{\infty}$ of \eqref{eq ODE XYZ}, with $X(0)=v$.
 If such  $a$ and $b$ exist, we set $I_v := (-a_v, b_v)$. (In particular, $I_v = \R$ if such a solution exists for all $a,b > 0$.) 
  If such an $a$  exists but no $b$, we set $I_v := (-a_v, 0)$, and if such a $b$  exists but no $a$, we set $I_v := (0, b_v)$.   If there are no such $a,b > 0$, we set $I_v := \emptyset$.
\begin{definition} \label{def centre mfd}
Let $\xi\colon I \times V_{\infty} \to V_{\infty}$ be a smooth map, and let $F\colon I \times V_{\infty} \to V_{\infty}$ be a polynomial map 
that separates invariant subspaces.
Consider the dynamical system for smooth maps $x\colon I \to V_{\infty}$ determined by
\beq{eq def xyz}
x(t) = \xi(t, X(t)),
\eeq
for $t \in I$,
for a smooth map $X\colon I \to V_{\infty}$ satisfying \cref{eq ODE XYZ}.
Let $D_{\tilde \mu}$ be as in \cref{prop dynamics}.
For every $a = c,s,u, cs, cu$, set
\[
E_{a} :=  \bigl\{(t, \xi(t,v)): t \in I_v, v \in  V_a, (t,v) \in D_{\tilde \mu} \bigr\} \subset \R \times V_{\infty}.
\]
The set $E_{c}$ is a \emph{centre subset} of the dynamical system in $x$; the set $E_{s}$ is a \emph{stable subset} of the system; and the set $E_{u}$ is an \emph{unstable subset} of the system. The set $E_{cs}$ is a \emph{centre-stable subset}, and $E_{cu,p}$ is a \emph{centre-unstable subset} of the system.
Such spaces are \emph{invariant} or \emph{integral subsets} of the dynamical system in $x$.
\end{definition}

\begin{remark}
If the map $\xi_t := \xi(t, \relbar)$ in \cref{def centre mfd} is invertible for all $t \in I$ (on a suitable domain), then the dynamical system in $x$ in that definition is equivalent to the \ODE
\[
\frac{d}{dt} (\xi_t^{-1} \circ x)(t) = A (\xi_t^{-1} \circ x)(t) + F(t, (\xi_t^{-1} \circ x) (t)).
\]
\end{remark}

\begin{remark}
In general, existence and uniqueness of solutions of \eqref{eq ODE XYZ} is not guaranteed, hence the careful definition of $I_v$. Existence and uniqueness of solutions is an assumption in previous definitions \cite[Theorem~2.9, e.g.]{Haragus2011}; see Hypothesis~2.7 in that reference.
There are existence and uniqueness results if $f$ satisfies a local Lipschitz condition, but that is not the case in many applications to \PDE s. Under additional assumptions, Vanderbauwhede \& Iooss \cite[proof of Theorem~3]{Vanderbauwhede92} showed such a local Lipschitz condition holds.
\end{remark}

\begin{remark}
Invariant subsets or submanifolds are not unique in general; here this non-uniqueness is due to various possibilities for~\(I_v\), \(D_{\tilde \mu}\) and~\(\xi_p\), and is reflected in the use of the indefinite article in \cref{def centre mfd}.
\end{remark}
\begin{example}
For one example of the non-uniqueness engendered via~\(\xi\), consider the classic example system of \(\dot x=-x^2\) and \(\dot y=-y\) in the role of~\eqref{eq ODE} (and let the step function \(H(x):=1\) when \(x>0\), and \(H(x):=0\) when \(x\leq0\)).  
This \ode\ system may be given, for every~\(C\), as the coordinate transformation,~\cref{eq def xyz}, \(x=X\) and \(y=Y+CH(X)e^{-1/X}\) together with the system,~\eqref{eq ODE XYZ}, \(\dot X=-X^2\) and \(\dot Y=-Y\) (by design, here symbolically identical to the original \(xy\)-system).
\cref{lem Vj invar} identifies \(Y=0\) as the centre subspace of this \(XY\)-system.
\cref{def centre mfd} then gives the classic non-uniqueness that, for every~\(C\), \(y=CH(x)e^{-1/x}\) are centre manifolds for the \(xy\)-system.
\end{example}

\begin{remark} \label{rem space mfd}
%The derivative at $0 \in V_{\infty}$ of the coordinate transformation~$\xi_p$
%%\ajr{Surely derivative evaluated ``at \(0\in V_1\)''}
% is the identity map, and hence invertible.
%If a suitable generalisation of the inverse function theorem applies to $\xi_p$, such as a version of the Nash--Moser theorem, then it follows
In the setting of \cref{def centre mfd}, if $\xi$ is a local diffeomorphism in the Fr\'echet manifold sense, then the subsets $E_{j}$ in \cref{def centre mfd} are Fr\'echet manifolds. This would justify the more specific terminology \emph{invariant submanifolds} rather than just {invariant subsets}. 
%\Todo: use geometric series to show $\xi_p$ is invertible on a domain?
\end{remark}

\subsection{Main result: an approximate normal form}\label{sec result}

%\subsection{Centre, stable and unstable components}\label{sec csu}

Our main result, \cref{thm normal form}, states that for an \ODE\ of the form \cref{eq ODE}, there is a dynamical system in the 
 normal form  used to define invariant manifolds in \cref{def centre mfd}, such that solutions of the normal form system satisfy \cref{eq ODE} up to a residual term that vanishes to any desired order. In this sense, \cref{eq ODE} is arbitrarily close to a dynamical system with clearly and robustly defined invariant manifolds.

%of maps $X\colon I \to V_{\infty}$ satisfying
%a so-called normal form differential equation
%\beq{eq ODE XYZ 2}
%\dot X(t) = AX(t) + F_p(t, X(t)),
%\eeq
%for a smooth map $F_p\colon I \times V_{\infty} \to V_{\infty}$,
%when~$x$ and~$X$ are related by a
%%\ajr{diffeomorphic?  Instead of \(\xi_p\), use \(\xi_p\)?}
%time-dependent coordinate transform, near identity at \(x=0\) (\cref{defNearId}),
%\beq{eq def xyz}
%x(t) = \xi_p(t, X(t))
%\eeq
%for $t \in I$.
%Crucially, the system~\eqref{eq ODE XYZ} is called a \emph{normal form} because the linear operator~\(A\) and the map~$F_p$ has a specific form in which central, stable and unstable subspaces can be clearly identified. 
%\Todo: also: in def of~$\xi_p$ have nicely converging operations (convolutions in time).

\begin{definition} \label{def pol growth}
A function $f\colon I \to \R$ \emph{grows at most polynomially} if there are $C,r>0$ such that for all $t \in I$, $|f(t)| \leq C(1+|t|^r)$. (This condition holds for all bounded functions if~$I$ is bounded.)

%\ajr{Clarify: grows in~\(t\)? or in~\(v\)?  as what?}
An infinitely differentiable map $\varphi\colon I \times V_{\infty} \to V_{\infty}$ \emph{has polynomial growth} if for every $v \in V_{\infty}$, every $k \in \N$, and every $l \in \N_0$, the function 
\[
\|\varphi(\relbar, v)^{(l)}\|_{V_k}\colon I \to [0, \infty)
\]
grows at most polynomially. 
\end{definition}

We use the term \emph{$\tilde \mu$-regular integral} for  an integral of the form 
%\ajr{Clarify the role of this paragraph?  what about backward?}
\[
\int_a^{\infty} e^{-\mu t} f(t) \, dt,
\]
where $\Re(\mu) > \tilde \mu$ and~$f$ grows at most polynomially. The larger~$\tilde \mu$, the better the convergence properties of~$\tilde \mu$-regular integrals.

\begin{theorem} \label{thm normal form}
Let $p \in \N$ be such that $p \geq 2$, 
$\beta - (p+1)\alpha > \tilde \mu$ and $\gamma - (p+1)\alpha > \tilde \mu$. Suppose that~$f$ has polynomial growth.
%
% and every $k \in \N$, there is a $C_{v, k} >0$ such that for all $t \in I$,
%\beq{eq f t 1}
%\|f(t,v)\|_{V_k} \leq C_{v, k} e^{\tilde \mu |t|}.
%\eeq
Then there are three infinitely differentiable maps
\(F_p, \xi_p, R_p\colon I \times V_{\infty} \to V_{\infty}\),
such that 
\begin{itemize}
\item
$F_p = \cO(2)$ and $F_p$ separates invariant manifolds;
\item  $R_p = \cO(p)$, 
\end{itemize}
 and if a smooth map $x\colon I \to V_{\infty}$ is given by
\beq{eq def xyz p}
x(t) = \xi_p(t,X(t))
\eeq
for all $t \in I$, for a smooth map $X \colon I \to V_{\infty}$ satisfying
\beq{eq ODE XYZ p}
\dot X(t) = AX(t) + F_p(t, X(t))
\eeq
for all $t \in I$,
%~$X$ and~$x$ are as in the transformed normal form~\eqref{eq ODE XYZ}+\eqref{eq def xyz}, 
then for all $t \in I$,
%\begin{subequations}\label{eq Fcsu}%
\begin{equation}
\dot x(t) = Ax(t) + f(t, x(t)) + R_p(t, X(t)).
%\quad \text{with }R_p = \cO(p).
\label{eq:FcsuDx}
\end{equation}
%We may choose the coordinate transform~\(\xi_p\) and the nonlinearity~$F_p$ such that
%$F_p$ is a polynomial map that separates invariant subspaces.
%
% the components of $F_p$ in~$V_c$, $V_s$ and~$V_u$ satisfy
%\newcommand{\subpar}[2][5em]{\parbox{#1}{\scriptsize\raggedright #2}}
%%\beq{eq Fcsu}
%\begin{align}
%(F_p)_c (t,v)&= \sum_{
%\subpar{$q \in \N_0^{\infty}$: $|q| \leq p$ 
%and $q^s = q^u=0$
%or   $q^s \not=0\not= q^u$}}
%F^q(t) v^q, \\
%(F_p)_s (t,v)&= \sum_{
%\subpar{$q \in \N_0^{\infty}$: $|q| \leq p$
%and $q^s \not=0$}} 
%F^q(t) v^q, \\
%(F_p)_u (t,v)&= \sum_{\subpar{$q \in \N_0^{\infty}$: $|q| \leq p$
%and $q^u \not=0$}} 
%F^q(t) v^q,
%\end{align}
%\end{subequations}
%%\eeq
%%
%%\beq{eq Fcsu}
%%\begin{split}
%%(F_p)_c (t,v)&= \sum_{
%%\substack{
%%q \in \N_0^{\infty}; |q| \leq p \\
%%\text{and $q^s = q^u=0$} \\
%%\text{or   $q^s \not=0$ and $q^u \not= 0$}
%%}
%%}
%%F^q(t) v^q; \\
%%(F_p)_s (t,v)&= \sum_{
%%\substack{q \in \N_0^{\infty}; |q| \leq p \\ \text{ and $q^s \not=0$}}
%%} F^q(t) v^q; \\
%%(F_p)_u (t,v)&= \sum_{\substack{q \in \N_0^{\infty}; |q| \leq p \\ \text{ and $q^u \not=0$}}} F^q(t) v^q,
%%\end{split}
%%\eeq
%for all $t \in I$ and $v \in V_{\infty}$,
%for smooth maps $F^q \colon I \to V_{\infty}$.
%
%\ajr{Clarify?}
Finally, there is a construction of the map~$\xi_p$ in which all integrals over~$I$ that occur are $\tilde \mu$-regular. 
\end{theorem}
We prove this theorem in \cref{sec der,sec cpt der,sec Frechet,sec co xform,sec transf choice}; see in particular \cref{sec proof main}.

%\begin{remark}
\cref{thm normal form 2} shows that the maps~$F_p$ and~$\xi_p$ can be chosen to be polynomials of a certain type. The conclusion that the integrals occurring are $\tilde \mu$-regular is more than just convenient: this is clear in the classical case where $\alpha = 0$ (see \cref{rem alpha zero}).
%\end{remark}

%\Todo: also include in the theorem that we only need to convolve with $e^{\mu t}$ if $|\Re(\mu)|>\tilde \mu$.

%\subsection{Invariant subspaces}
%
%Our main application of \cref{thm normal form} is that it empowers us to  define invariant\slash integral subsets  in a robust way.  (We use the more general term subsets rather than manifolds, as the manifold structure of these sets is not necessarily obvious; see \cref{rem space mfd}.) This is based on the 
%the fact that these subsets for the normal form \eqref{eq ODE XYZ} can be read off directly, by \cref{lem Vj invar,prop dynamics} below.

\begin{remark} \label{rem alpha zero}
In cases where the centre eigenvalue bound $\alpha$ equals zero,  we can always choose~$\tilde \mu$ so small that the conditions on $\alpha, \beta, \gamma, \tilde \mu$ and~$p$ in \cref{thm normal form} are satisfied. 
In these cases, the residual~$R_p$ can be made to vanish to arbitrarily large order~\(p\). 
Furthermore, the integrals that occur in the construction of $\xi_p$ (see \cref{def conv emu}) are $\tilde \mu$-regular for some $\tilde \mu>0$ precisely if they converge. Hence $\tilde \mu$-regularity for some $\tilde \mu>0$ is a necessary condition for the construction to make sense.

Choosing the centre eigenvalue bound~$\alpha$ positive, which imposes a positive lower bound on~$\tilde \mu$, restricts the vanishing order~$p$ of~$R_p$, but also makes the construction of the coordinate transform~$\xi_p$ more robust, in the sense that the integrals over~$I$ in its construction are~$\tilde \mu$-regular. 
Many researchers choose to phrase problems as singular perturbations \cite[e.g.]{Bykov2013, Pavliotis07, Verhulst05}.
In such cases the bounds on the hyperbolic rates \(\beta,\gamma \propto\frac1\varepsilon \to\infty\) as the perturbation parameter \(\varepsilon\to0\).
Consequently, choosing \(\tilde\mu,p\propto1/\sqrt\varepsilon\) (say) then the residual~$R_p$ again can be made to vanish to arbitrarily large order for small enough~\(\varepsilon\).

However, in applications we generally require an invariant manifold in some chosen domain of interest that resolve phenomena on chosen time scales of interest.
Such subjective choices, informed by the governing equations, generally dictate the chosen bound~\(\alpha\) separated by a big enough gap from the bounds~\(\beta,\gamma\) so that the centre manifold evolution, constructed to a valid order~\(p\), provides a useful model over the chosen domain for the desired phenomena. 

%
%This is reflected in the decay behaviour we require of the functions 
%we convolve by to construct $\xi_p$, as in \cref{def conv emu}; see \eqref{eq hat psi hat F Jq}. 
%In practice, the parameter~$\alpha$ should be thought of as very small compared to~$\beta$ and~$\gamma$, of the order of a measurement or numerical error in the eigenvalues of~$A$ with real parts close to zero. Then the vanishing order~$p$ of the residual~$R_p$ can be made very large. In fact, for any desired vanishing order $p$, one can choose a positive centre eigenvalue bound $\alpha$ so that the conditions on $p$, $\alpha$, $\beta$ and $\gamma$ in \cref{thm normal form} hold.

%Finally, and importantly for applications, taking positive~$\alpha$ allows us to adapt the definition of an approximate centre manifold to the problem considered. For example, one may wish to consider eigenfunctions~$e_j$ to be central when~$|\Re(\alpha_j)|$ is within measurement or numerical precision of zero.   
\end{remark}

\begin{remark}
The derivative at $0 \in V_{\infty}$ of the coordinate transformation~$\xi_p$
%\ajr{Surely derivative evaluated ``at \(0\in V_1\)''}
 is the identity map, and hence invertible.
If a suitable generalisation of the inverse function theorem applies to $\xi_p$, such as a version of the Nash--Moser theorem, then it follows that $\xi_p$ is a local diffeomorphism at zero.
Then it would be justified to call the invariant subsets of \cref{def centre mfd} invariant submanifolds in this setting (at least in a neighbourhood of zero), see \cref{rem space mfd}.
%
% in the Fr\'echet manifold sense. That would imply that the subspaces $E_{j,p}$ in \cref{def centre mfd} are Fr\'echet manifolds, which would justify the more specific terminology \emph{invariant submanifolds} rather than just {invariant subsets}. 
\end{remark}

\begin{remark} \label{rem direct constr}
In the proof of \cref{thm normal form}, explicit constructions of the maps $F_p$ and $\xi_p$ are given. In practice, however, it can be easier to determine these maps in more direct ways. This is illustrated in an example in \cref{sec ex}. \cref{thm normal form} implies that one can always find these maps. We prove this by giving a construction that always leads to an answer, even though more direct constructions may exist in specific situations.

Similarly, the domain $D_{\mu}$ in \cref{prop dynamics}, defined in \cref{eq Dmu}, is guaranteed to have the properties in \cref{prop dynamics}. In practice, these properties often hold on much larger domains. 
\end{remark}

\subsection{A general class of PDEs in bounded domains} \label{sec special}

Because \cref{thm normal form} applies to abstract Banach spaces~$V_k$, it gives one the flexibility to choose these spaces such that, for specific \pde\ applications, 
\begin{enumerate}
\item the residual~$R_p$ is of order~$p$ with respect to norms relevant to the problem, and
\item the spaces~$V_k$ incorporate the relevant boundary conditions.
\end{enumerate}
This subsection explores a class of nonlinear \PDE s to which \cref{thm normal form}, and hence \cref{def centre mfd}, apply.

Let $d \in \N$ be the dimension of the domain of the \pde{}s to be considered. 
Let~$\Omega$ be a bounded, open subset of~$\R^d$,  or of a $d$-dimensional manifold, with~$C^1$ boundary. 
Let $m \in \N$, % be the number of specified \pde{}s, 
%$l \in \N_0$\ajr{What is \(l\)? It only appears, and that differently?, in \cref{cor special case}}, 
and let $1\leq p < \infty$.
%Consider the Sobolev space $W^{l,p}(\Omega, \R^m)$. %If $m = 1$, we denote this space by~$W^{l,p}(\Omega)$. 
For $k \in \N$, let~$V_k$ be the Sobolev space~$W^{k-1,k+1}(\Omega, \R^m)$.

%\Todo: need $W^{k,p}_0$-type Sobolev spaces here? Are the same for smooth bdry?

%\ajr{What is the import of subscript~\(c\)?  Is it centre related?  is there \(s\) for stable?}
Let $A\colon C_c^{\infty}(\Omega; \R^m)\to C_c^{\infty}(\Omega; \R^m)$ be a linear partial differential operator.
(Here the subscript $c$ denotes compactly supported functions.)
Let $s \in \N$, with $s \geq 2$, be the `polynomial' order of the nonlinearities in the \pde{}s. Let
\[
D_1, \ldots, D_s\colon C_c^{\infty}(\Omega; \R^{m}) \to  C_c^{\infty}(\Omega; \R)  
\]
be linear partial differential operators.  
For index-vector $q \in \N_0^s$ and $u \in C^{\infty}_c(\Omega, \R^m)$, we set
\beq{eq Duq}
(Du)^q := (D_1u)^{q_1} \cdots (D_s u)^{q_s}. 
\eeq

Let~$\alpha$, $\beta$, $\gamma$ and~$\tilde \mu$ be as in  \cref{sec result}. 
Fix smooth functions\footnote{The real line may be replaced by a smaller open interval.} $a_q^j\colon \R \to \C$, for $q \in \N_0^{s}$, with $|q| \leq s$, such that these functions and all their derivatives grow %\ajr{Is growth in \(u\)? or \(t\)? or both?} 
at most polynomially.
Define $f\colon \R \times C_c^{\infty}(\Omega; \R^{m}) \to C_c^{\infty}(\Omega; \R^{m})$ by $f(t, u) := (f_1(t, u), \ldots, f_m(t, u))$, where for each~$j$,
%\beq{eq f cpt poly Sob}
\begin{equation*}
f_j(t, u) = \sum_{q \in \N_0^s, |q| \leq s} a_q^j(t)(Du)^q% \cdots D_n u,
\end{equation*}%\eeq
%\ajr{All time? or \(I\)?}
for $t \in \R$ and $u \in C_c^{\infty}(\Omega; \R^{m})$.
Suppose that $f = \cO(2)$.

We write 
\[
W^{\infty}(\Omega; \R^m) := \bigcap_{k=1}^{\infty} W^{k-1, k+1}(\Omega; \R^m).
\]
Then $C^{\infty}_c(\Omega; \R^m) \subset W^{\infty}(\Omega; \R^m) \subset C^{\infty}(\Omega; \R^m)$.
The maps~$A$ and~$f$ extend continuously to~$W^{\infty}(\Omega; \R^m)$.  Suppose that that the  eigenfunctions~$\{e_j\}_{j=1}^{\infty}$ of this extension of~$A$ form a Hilbert basis of~$L^2(\Omega, \R^m)$.

\begin{theorem} \label{thm special case}
The spaces~$V_k$ and the maps~$A$ and~$f$ satisfy the hypotheses of \cref{thm normal form}.
\end{theorem}
We prove \cref{thm special case} in
\cref{sec ex f}. 
Together with \cref{thm normal form}, it has the following immediate consequence.

\begin{corollary}\label{cor special case}
Let $p \in \N$ be such that $p \geq 2$. 
Suppose that~$\alpha$, $\beta$, $\gamma$ and~$\tilde \mu$ satisfy
$\beta - (p+1)\alpha > \tilde \mu$ and $\gamma - (p+1)\alpha > \tilde \mu$ (as in \cref{thm normal form}). 
Then there are infinitely differentiable maps
\[
F_p, \xi_p, R_p\colon \R \times W^{\infty}(\Omega; \R^m) \to W^{\infty}(\Omega; \R^m),
\]
where $F_p$ is a polynomial map that separates invariant subspaces,
such that if~$X$ and~$x$ are as in \eqref{eq ODE XYZ p} and \eqref{eq def xyz p}, then
\[
\dot x(t) = Ax(t) + f(t, x(t)) + R_p(t, X(t))
\]
for all $t \in I$.
Further, for every $l \in \N$, there is a $k \in \N$ %\ajr{is \(k\lessgtr l\)?} 
such that for all $u \in W^{\infty}(\Omega; \R^m)$, 
\[
\|R_p(t,v) \|_{W^{l-1, l+1}}  = O( \|v\|^p_{W^{k-1, k+1}} )
\]
as $v \to 0$ in $W^{k-1,k+1}(\Omega, \R^m)$.
%\ajr{Aim to generalise here or in a separate corollary, the slow-fast case.  However, the \(\mu\)-regular integrals are a problem.}
%We may choose~$F_p$ such that its components in~$V_c$, $V_s$ and~$V_u$ satisfy \eqref{eq Fcsu}, 
%for all $t \in I$ and $v \in W^{\infty}(\Omega; \R^m)$,
%for smooth maps $F^q \colon \R \to C_c^{\infty}(\Omega; \R^{m})$. 
There is a construction of the map~$\xi_p$ in which all integrals over~$I$ that occur are $\tilde \mu$-regular. 
\end{corollary}
This corollary shows that any \pde{} of the form  \eqref{eq ODE}, with~$A$ and~$f$ as in this subsection,  is equivalent up to a residual of order $p$ to a \pde\ with clear invariant manifolds, as in Definition \ref{def centre mfd}.

%empowers us to construct approximate invariant manifolds of the \PDE s.

\begin{example}
Suppose that $\Omega = S^1$, the circle. This amounts to imposing periodic boundary conditions. Take $m = 1$, and let $A\colon C^{\infty}(S^1) \to C^{\infty}(S^1)$ by any linear partial differential operator with constant coefficients. Its eigenfunctions,  $e_j(\theta) = e^{ij\theta}$ for $j \in \Z$ and $\theta \in \R/2\pi\Z$, are orthogonal in the Sobolev spaces~$W^{k,2}(S^1)$. For a map~$f$ as in \cref{thm special case}, that is, a polynomial expression in derivatives of functions, whose polynomial coefficients increase at most polynomially, \cref{thm special case} implies that the conditions of \cref{thm normal form} are satisfied in this case, so \cref{cor special case} applies. This generalises directly to cases where~$\Omega$ is a higher-dimensional torus; that is, to problems in~$\R^d$ with periodic boundary conditions. 
Here we used %\ajr{Where do "we use"? this paragraph? subsequent sections? or?} 
the case where the domain~$\Omega$ is a manifold, rather than an open subset of~$\R^d$.
\end{example}

Most of the rest of this paper is devoted to proofs of  \cref{thm normal form,thm special case}, and developing the tools used in these proofs.  
In \cref{sec dynamics}, we prove \cref{lem Vj invar,prop dynamics}.
In \cref{sec ex} we illustrate \cref{cor special case} by working out an example.

\section{Derivatives and polynomials} \label{sec der}

In this section we review standard material on derivatives of maps between normed vector spaces. We also briefly discuss polynomial maps between normed vector spaces.
Throughout this section, $(V, \|\cdot \|_V)$ and $(W, \|\cdot\|_W)$ are normed vector spaces, possibly infinite-dimensional. Let $U\subset V$ be an open subset, and let $f\colon U \to W$ be a map. We fix an element $u \in U$.

\subsection{First order derivatives}

This subsection and the next contain some standard definitions and facts about derivatives of maps between normed vector spaces. Details and proofs can be found in various textbooks \cite[e.g.]{Zorich}.

For a map $\varphi\colon V \supset \dom(\varphi) \to W$, we use the notation $\varphi(h) = o(h)$ for the statement
\[
\lim_{h \to 0} \frac{\|\varphi(h)\|_W}{\|h\|_V}=0,
\]
were $h$ runs over $ \dom(\varphi) \setminus \{0\}$.

\begin{definition}
The map $f:U\to W$ is \emph{differentiable} at $u$, if there is an operator $f'(u) \in \cB(V,W)$ such that
%\ajr{Clarify the space $V$?}
\[
f(u+h) = f(u) + f'(u)h + o(h).
\]
Then $f'(u)$ is the \emph{derivative} of $f$ at $u$. 
If $f$ is differentiable at every point in~$U$, then we say that $f$ is \emph{differentiable}. In that case, the \emph{derivative} of $f$ is the map
\beq{eq def der}
f'\colon U \to \cB(V,W)
\eeq
mapping $u \in U$ to $f'(u)$.
\end{definition}
The derivative of a map at a point is unique, if it exists.

\begin{lemma}[Chain rule] \label{lem chain rule}
Let $(X, \|\cdot \|_X)$ be a third normed vector space. Let $A \subset W$ be an open subset containing $f(U)$. If $g\colon A \to X$ is differentiable at~$f(u)$ and $f$ is differentiable at~$u$, then $g\circ f$ is differentiable at~$u$, and
\[
(g\circ f)'(u) = g'(f(u)) \circ f'(u).
\]
\end{lemma}
\begin{ArXiv}
\begin{proof}
For all $h \in V$ such that $u+h \in U$, differentiability of $f$ at $u$ and of $g$ at $f(u)$ imply that
\[
\begin{split}
(g\circ f)(u+h) &= g(f(u)+f'(u)h+o(h)) \\
&= g(f(u)) + g'(f(u))(f'(u)h+o(h)) + o(f'(u)h+o(h)).
\end{split}
\]
Since $g'(f(u))$ and $f'(u)$ are bounded operators, the second term on the right-hand side equals $g'(f(u))f'(u)h+o(h)$, while the last term is $o(h)$.
\end{proof}
\end{ArXiv}

\begin{definition}\label{defNearId}
The map $f$ is a \emph{near-identity} at $u$ if the map \cref{eq def der} is continuous in a neighbourhood of $u$, and 
\[
f'(u)h = h +o(h).
\]
%If $f$ is a near-identity at every point in $U$, then we say that $f$ is a \emph{near-identity}.
\end{definition}

%\begin{definition} \label{def partial der}
%Suppose that $V = V_1 \times V_2$ for two normed vector spaces $V_1$ and $V_2$. 
%Write $u = (u_1, u_2)$, with $u_j \in V_j$. Then $f$ is \emph{differentiable in the $V_1$-direction} at $u$ if there is an operator $f'_{V_1}(u) \in \cB(V_1, W)$ such that, for $h_1 \in V_1$,
%\[
%f(u_1+h_1, u_2) = f'_{V_1}(u_1, u_2)h_1 + o(h_1).
%\]
%Then $f'_{V_1}(u)$ is the \emph{derivative of $f$ at $u$ in the $V_1$-direction}. The {derivative of $f$ at $u$ in the $V_2$-direction} is defined analogously.
%\ajr{Similarly for \(V_2\).??}
%\end{definition}

\subsection{Higher order derivatives}

Fix a positive integer $n \in \N$. 
%\ajr{For every? any?}
We write $\cB^n(V,W)$ for the space of multilinear maps
$
\lambda\colon V^n \to W
$
for which the norm
\beq{eq norm Bn}
\|\lambda\| := \sup_{
\substack{v_1, \ldots, v_n \in V\\ \|v_1\|_V = \cdots = \|v_n\|_V=1}
} \|\lambda(v_1, \ldots, v_n)\|_W
\eeq
is finite. There is a natural isometric isomorphism
\beq{eq iso Bn}
\cB(V, \cB(V, \ldots, \cB(V,W) \cdots )) \xrightarrow{\cong} \cB^n(V,W)
\eeq
mapping an operator $T$ in the left-hand side to the operator $\lambda \in \cB^n(V,W)$ given by
\[
\lambda(v_1, \ldots, v_n) = T(v_1)(v_2)\cdots(v_n),
\]
for $v_1, \ldots, v_n \in V$.

Suppose $f:U\to W$ is differentiable. 
%\ajr{and $V$ is ?}
%Then its derivative defines a map
%\[
%f'\colon U \to \cB(V,W).
%\]
The map $f$ is \emph{twice differentiable} at $u$ if the map \cref{eq def der} is differentiable at~$u$. Then we write
\[
f^{(2)}(u) :=
(f')'(u) \in \cB(V, \cB(V,W)) \cong \cB^2(V,W).
\]
Inductively, for $n \geq 2$,  $f$ is defined to be $n$ times differentiable at $u$ if it is $n-1$ times differentiable, and the map
\[
f^{(n-1)}\colon U \to \cB^{n-1}(V,W)
\]
is differentiable at $u$. We then set
\[
f^{(n)}(u) := (f^{(n-1)})'(u) \quad \in \cB^n(V,W).
\]
In this case, we write
\beq{eq fn}
f^{(n)}(u)h^n := f^{(n)}(u)(h, h, \ldots, h).
\eeq
As before, we say that $f$ is \emph{$n$ times differentiable} if it is $n$ times differentiable at every point in $U$.
And \emph{infinitely differentiable} means $n$ times differentiable for every $n\in\N$.

\begin{theorem}[Taylor's theorem] \label{thm Taylor}
Suppose $f$ is $n+1$ times differentiable. 
%Let $h \in V$, 
%
%and suppose that the closed line segment from $u$ to $u+h$ is contained in $U$. 
Suppose that $\|f^{(n+1)}(\xi)\| \leq M$ for all $\xi$ in a closed ball around $u$ contained in $U$. Then for every $h$ in this ball,
%\ajr{Should \(1/(n+1)\) be \(1/(n+1)!\)?}
\[
\Bigl\| f(u+h) - \sum_{j=0}^n \frac{1}{j!}  f^{(j)}(u)h^j\Bigr\|_W \leq \frac{M}{(n+1)!}\|h\|_V^{n+1}.
\]
\end{theorem}

\begin{ArXiv}

\subsection{Example: Burgers' equation}  \label{sec ex burgers}

An example of a map to which we would like to apply the material in this section and the next is the nonlinear term $u_x u$ in Burgers' equation
\beq{eq Burgers}
u_t = u_{xx} - u_x u.
\eeq

%\ajr{Unless I am mistaken, this \(I\)-interval should be in~$x$, so we should use a different symbol to avoid confusion with the \(I\)-interval in time.  Commonly people use $\Omega$ for `spatial' domain, as we do elsewhere.}
Let $\Omega \subset \R$ be a bounded, open interval in~\(x\). For every $k \in \N_0$, consider the $k$th $L^2$-Sobolev space $W^{k,2}(\Omega)$, with %, which we here take to be the completion of $C^{\infty}_c(\Omega)$ in 
the inner product
\[
(u_1, u_2)_{W^{k,2}} := \sum_{j=0}^k (u_1^{(j)}, u_2^{(j)})_{L^2}.
\]
Consider the map
$f\colon W^{1,2}(\Omega) \to L^1(\Omega)$
given by 
\[
f(u) = u'u.
\]
First of all, for $u \in C^{\infty}_c(\Omega)$, the Cauchy--Schwartz inequality for $L^2(\Omega)$ (or H\"older's inequality) implies that
\beq{eq burgers CS}
\|u'u\|_{L^1} \leq \| u' \|_{L^2} \| u \|_{L^2}  \leq \|u\|_{W^{1,2}}^2.
\eeq
So $f$ indeed maps  $W^{1,2}(\Omega)$ into $L^1(\Omega)$.

We claim that $f$ is infinitely differentiable. %, but not compactly differentiable.
 Indeed, for  $u,h \in W^{1,2}(\Omega)$,
\[
f(u+h) = f(u) + h'u+u'h + h'h.
\]
And by \eqref{eq burgers CS}, 
$
\|h'h\|_{L^1} \leq \|h\|_{W^{1,2}}^2, 
$
so 
\[
f'(u)h = h'u+u'h.
\]
The map $f'(u)\colon W^{1,2}(\Omega) \to L^1(\Omega)$ is bounded, because, analogously to \eqref{eq burgers CS}, 
\beq{eq Burgers f' bdd}
\|h'u+u'h\|_{L^1} \leq 2 \|u\|_{W^{1,2}}  \|h\|_{W^{1,2}}.
\eeq
%This map is not compact, however. This means that we can't directly apply \cref{cor Taylor cpt} to $f$.

If $u,h_1, h_2 \in W^{1,2}(\Omega)$, then
\[
f'(u+h_2)(h_1) = f'(u)h_1 + h_2'h_1 + h_1'h_2.
\]
So $f^{(2)}(u)$ is the  operator in $\cB^2(W^{1,2}(\Omega), L^1(\Omega))$ given by
\[
f^{(2)}(u)(h_1, h_2) = h_2'h_1 + h_1'h_2.
\]
The term $o(h_2)$ in the definition of the derivative is zero in this case, and that $f^{(2)}(u)$ does not depend on $u$. This implies that for every $n \geq 3$,  $f^{(n)}(u) = 0$. So $f$ is indeed infinitely differentiable.

\end{ArXiv}

\subsection{Bounded polynomial maps}

An operator in $\cB^n(V,W)$ is said to be \emph{symmetric} if it is invariant under permutations of its arguments. Let $S\cB^n(V,W)$ be the subspace of symmetric operators in $\cB^n(V,W)$. An example of such a symmetric operator is the $n$th derivative of a map.
\begin{lemma} \label{lem symm}
If $f$ is $n$ times differentiable at $u$, then $f^{(n)}(u)$ is symmetric.
\end{lemma}

We denote the permutation group of $\{1,\ldots, n\}$ by
$\Sigma_n$.
\begin{ArXiv}
\begin{lemma} \label{lem SBn closed}
The subspace $S\cB^n(V,W) \subset \cB^n(V,W)$ is closed.
\end{lemma}
\begin{proof}
If $T \in \overline{S\cB^n(V,W)} \setminus  S\cB^n(V,W)$, and $v_1, \ldots, v_n \in V$ and $\sigma \in \Sigma_n$ are such that
$
T(v_1, \ldots, v_n) \not= T(v_{\sigma(1)}, \ldots, v_{\sigma(n)}),
$
set
\begin{multline*}
\varepsilon := \bigl\| T\bigl(v_1/\|v_1\|_V, \ldots, v_n/\|v_n\|_V\bigr) - T\bigl(v_{\sigma(1)}/\|v_{\sigma(1)}\|_V, \ldots, v_{\sigma(n)}/\|v_{\sigma(n)}\|_V\bigr) \bigr\|_W
\\ > 0.
\end{multline*}
Let $\tilde T \in S\cB^n(V,W)$ be such that $\|\tilde T - T\| < \varepsilon/2$, for the norm \eqref{eq norm Bn}. Then symmetry of $S$ and the triangle inequality imply that
\begin{multline*}
\bigl\| T\bigl(v_1/\|v_1\|_V, \ldots, v_n/\|v_n\|_V\bigr) - T\bigl(v_{\sigma(1)}/\|v_{\sigma(1)}\|_V, \ldots, v_{\sigma(n)}/\|v_{\sigma(n)}\|_V\bigr) \bigr\|_W 
\\
< \bigl\| T\bigl(v_1/\|v_1\|_V, \ldots, v_n/\|v_n\|_V\bigr) -  \tilde T\bigl(v_1/\|v_1\|_V, \ldots, v_n/\|v_n\|_V\bigr) \bigr\|_W 
\\
+\bigl\| \tilde T\bigl(v_{\sigma(1)}/\|v_{\sigma(1)}\|_V, \ldots, v_{\sigma(n)}/\|v_{\sigma(n)}\|_V\bigr) - T\bigl(v_{\sigma(1)}/\|v_{\sigma(1)}\|_V, \ldots, v_{\sigma(n)}/\|v_{\sigma(n)}\|_V\bigr)\bigr\|_W 
\\
< \varepsilon,
\end{multline*}
a contradiction.
\end{proof}
By this lemma, $S\cB^n(V,W)$ is a Banach space if $V$ and $W$ are.
\end{ArXiv}

Let 
$
S\colon \cB^n(V,W) \to S\cB^n(V,W)
$
be the symmetrisation operator: for every $\lambda \in \cB^n(V,W)$ and $v_1, \ldots, v_n \in V$,
\[
(S\lambda)(v_1, \ldots, v_n) = \frac{1}{n!}\sum_{\sigma \in \Sigma_n} \lambda(v_{\sigma(1)}, \ldots, v_{\sigma(n)}).
\]
\begin{ArXiv}
 (An alternative proof of \cref{lem SBn closed} is to show that $S$ is continuous, and to note that $S\cB^n(V,W)$ is the zero level set of $S$ minus the identity.)
\end{ArXiv}
\begin{journal}
The operator  $S$ is continuous, and that $S\cB^n(V,W)$ is the zero level set of $S$ minus the identity, and hence closed in $\cB^n(V,W)$. So $S\cB^n(V,W)$ is a Banach space if $V$ and $W$ are.
\end{journal}

An element $\lambda \in \cB^n(V,W)$  defines a map
$
p_{\lambda}\colon V \to W
$
by
\beq{eq def monomial}
p_{\lambda}(v) = \lambda(v, \ldots, v).
\eeq
We have $p_{S\lambda} = p_{\lambda}$, and the map $\lambda \mapsto p_{\lambda}$ is injective on $S\cB^n(V,W)$.

\begin{definition}
A \emph{bounded homogeneous polynomial map} of degree $n$ from~$V$ to $W$ is a map of the form $p_{\lambda}$ as in \eqref{eq def monomial}. We write $\Pol^n(V,W)$  for the space of such maps.  It inherits a norm  from the space $S\cB^n(V,W)$ via the linear isomorphism $\lambda \mapsto p_{\lambda}$.
 If $\lambda \not=0$, then the \emph{degree} of ~$p_{\lambda}$ is $n$.

A \emph{bounded polynomial map} from $V$ to $W$ is a finite sum of bounded homogeneous polynomial maps. The \emph{degree} of a bounded polynomial map is the degree of its highest-degree homogeneous term.

%We write $\Pol^n(V,W)$  for the space of bounded polynomial maps of degree $n$ from~$V$ to $W$.
We write  $\Pol(V,W)$ for the space of all bounded polynomial maps from $V$ to $W$. This is the algebraic direct sum of the spaces $\Pol^n(V,W)$.
\end{definition}
\begin{ArXiv}
By \cref{lem SBn closed}, $\Pol^n(V,W)$ is a Banach space if $V$ and $W$ are. 
\end{ArXiv}
\begin{journal}
Because $S\cB^n(V,W)$ is a Banach space if $V$ and $W$ are, so is $\Pol^n(V,W)$. 
\end{journal}
We could define $\Pol^0(V,W)$ as the space of constant maps into $W$, but we only consider homogeneous polynomials of order at least one.

%By \cref{lem symm}, 
If $f$ is $n$ times differentiable at $u$, then we have the map
\[
h \mapsto
f^{(n)}(u)h^n \in \Pol^n(V,W).
\]

\cref{lem pol diffble}--\ref{lem compos bdd pol} below are basic facts showing that bounded polynomials and their orders and compositions behave as one would expect.
\begin{journal}
Their proofs are short and straightforward.
\end{journal}

\begin{lemma} \label{lem pol diffble}
Every bounded polynomial map is infinitely differentiable.
\end{lemma}
\begin{ArXiv}
\begin{proof}
Let $\lambda \in S\cB^n(V,W)$, for some $n \geq 2$. Then for all $u,h \in V$,
\[
p_{\lambda}(u+h) = p_{\lambda}(u) + n\lambda(h, u, \ldots, u) + O(\|h\|^2).
\]
Hence $p_{\lambda}$ is differentiable, and 
\[
p_{\lambda}'(u) = n\lambda(u, \ldots, u),
\]
where on the right-hand side, the operator~$\lambda$ is applied to $n-1$ copies of $u$, to give an element of $\cB(V,W)$. Hence $p_{\lambda}'$ is a bounded polynomial map in $\Pol^{n-1}(V, \cB(V,W))$. This proves the claim by induction.
\end{proof}
\end{ArXiv}

\begin{lemma}\label{lem order pol}
If $p \in \Pol^n(V,W)$, then there is a constant $C>0$ such that for all $v \in V$, 
\[
\|p(v)\|_W \leq C \|v\|_V^n.
\]
\end{lemma}
\begin{ArXiv}
\begin{proof}
Let $\lambda \in \cB^n(V,W)$.
By boundedness and multilinearity of $\lambda$, we have for all nonzero $v \in V$,
\[
\|p_{\lambda}(v)\|_W = \|v\|_V^n \cdot \|\lambda(v/\|v\|_V, \ldots, v/\|v\|_V)\|_W \leq \|\lambda \|  \cdot \|v\|_V^n
\]
%
%%%
%We use induction on $n$. For $n=1$, the claim is true by definition of bounded operators. Suppose that the claim is true for a given $n$. Let $\lambda \in \cB^{n+1}(V,W)$.
%%
%%Let $p \in \Pol^{n+1}(V,W)$. Then there is a $\lambda \in \cB^{n+1}(V,W)$ such that $p = p_{\lambda}$. 
%Now $\cB^{n+1}(V,W) =  \cB^{n}(V, \cB(V,W))$, and this identification gives a polynomial $\tilde p_{\lambda} \in \Pol^n(V, \cB(V,W))$, given by
%\[
%\tilde p_{\lambda}(v) = \lambda(v, \ldots, v),
%\]
%for $v\in V$,
%where we insert $n$ copies of $v$  into $\lambda$ on the right-hand side to obtain an operator in $\cB(V,W)$. By the induction hypothesis, there is a $C>0$ such that for all $v \in V$,
%\[
%\|\tilde p_{\lambda}(v)\|_{\cB(V,W)} \leq C \|v\|_{V}^n.
%\]
%Hence
%\[
%\|p_{\lambda}(v)\|_W = \| \bigl(\tilde p_{\lambda}(v)\bigr)(v)\|_W \leq \|\tilde p_{\lambda}(v)\|_{\cB(V,W)} \|v\|_V
%\leq C \|v\|_{V}^{n+1}.
%\]
\end{proof}
\end{ArXiv}

\begin{lemma} \label{lem order pol lower}
If $p$ is a polynomial map from $V$ to $W$ of order lower than $n$, and
\[
\|p(v)\|_W = O(\|v\|_V^{n}),
\]
as $v\to 0$ in $V$,
then $p=0$.
\end{lemma}
\begin{ArXiv}
\begin{proof}
Let $m,n \in \N$.
Let $\lambda \in S\cB^m(V,W)$, and suppose that $\|p_{\lambda}(v)\|_W = O(\|v\|_V^{n})$, as $v\to 0$ in $V$. Then there is a $C>0$ such that for all $v \in V$ with unit norm and $s>0$ small enough,
%\ajr{Since we use \(t\) for time, better to use something else here?}
\[
s^m \|p_{\lambda}(v)\|_W = \|p_{\lambda}(sv)\|_W \leq Cs^n.
\]
If $m<n$, this implies that $\|p_{\lambda}(v)\|_W =0$.
\end{proof}
\end{ArXiv}

\begin{lemma}\label{lem compos bdd pol}
If $p_1 \in \Pol^m(U,V)$ and $p_2 \in \Pol^n(V,W)$, then $p_2 \circ p_1 \in \Pol^{mn}(U,W)$.
\end{lemma}
\begin{ArXiv}
\begin{proof}
For $\lambda_1 \in \cB^m(U,V)$ and $\lambda_2 \in \cB^n(V,W)$,  define $\lambda_2 \circ \lambda_1\colon U^{mn} \to W$ by
\begin{multline*}
\lambda_2 \circ \lambda_1(u_{11}, \ldots, u_{1m}; \ldots; u_{n1}, \ldots, u_{nm}) 
:=\\
 \lambda_2\bigl(\lambda_1(u_{11}, \ldots, u_{1m}), \ldots, \lambda_2(\lambda_1(u_{n1}, \ldots, u_{nm})\bigr),
\end{multline*}
for $u_{jk} \in U$. Then one checks directly that $\lambda_2 \circ \lambda_1 \in \cB^{mn}(U,V)$. This implies the claim about polynomials.
\end{proof}
\end{ArXiv}

\subsection{Standard monomials} \label{sec std mon}

%Now suppose that $V$ and $W$ are Hilbert spaces. Fix (orthonormal) Hilbert bases $\{e_j\}_{j=1}^{\infty}$ of $V$ and $\{f_k\}_{k=1}^{\infty}$ of $W$. 

%Now let $\{e_j\}_{j=1}^{\infty}$ be a countable subset of $V$. What follows is most natural if $V$ is a Hilbert space and  $\{e_j\}_{j=1}^{\infty}$ is a Hilbert basis, but it applies more generally.

Let $V^* := \cB(V,\C)$ be the continuous dual of $V$. We denote the pairing between $V^*$ and $V$ by $\langle \relbar, \relbar \rangle$.
%Suppose the norm on $V$ comes from an inner product $(\relbar, \relbar)_V$.
For every $j \in \N$, let $e^j \in V^*$ be given.
What follows is most natural if $V$ is a Hilbert space and  $e^j$ is given by taking inner products with an element $e_j$ of a Hilbert basis, but it applies more generally.

%Let $\N_0^{\infty}$ be the set of sequences in $\N_0$ with finitely many nonzero terms. 
%For $q \in \N_0^{\infty}$, let $|q|$ be the sum of its terms. 
Consider a multi-index $q \in \N_0^{\infty}$.
If $|q|=n$, and $m$ is the largest number for which $q_m \not=0$, then we define the element
\beq{eq def eq}
e^q := \underbrace{e^1 \otimes \cdots \otimes e^1}_{\text{$q_1$ factors}} \otimes \cdots \otimes 
 \underbrace{e^m \otimes \cdots \otimes e^m}_{\text{$q_m$ factors}} 
\in \cB^n(V, \C).
\eeq
In other words, for all $v_1, \ldots, v_n \in V$,
\begin{align*}&
e^q(v_1, \ldots, v_n) =
\\& 
\langle e^1, v_1 \rangle \cdots \langle e^1, v_{q_1}\rangle \langle e^2, v_{q_1 + 1}\rangle
 \cdots \langle e^2, v_{q_1 + q_2}\rangle \cdots
\langle e^m, v_{q_1 + \cdots + q_{m-1} +1}\rangle \cdots  \langle e^m, v_n\rangle. 
\end{align*}
%The operators $e^q$ of this form are finite-rank in each entry, so in particular compact. We therefore have
%We have
%\[
%Se^q \in S\cB^n(V, \C).
%\]
We write $p^q := p_{e^q}$ for the corresponding homogeneous polynomial. One could call this the standard $q$-monomial with respect to the set $\{e^j\}_{j=1}^{\infty}$. (If $V = \C^k$ and the elements $e^j$ are the standard coordinates, then the monomial functions in the usual sense are precisely the scalar multiples of the maps $p^q$.)

 For $v \in V$, we write
\beq{eq def vq}
v^q := p^q(v) = \prod_{j=1}^{\infty} \langle e^j, v\rangle ^{q_j}.
\eeq
 This product is finite (since $q$ has finitely many nonzero terms) and depends on the set $\{e^j\}$.
%
%\Todo: I don't think we need this lemma, so it can be removed.
%\begin{lemma} \label{lem pol conv}
%Suppose that the elements $e_j \in V$ are unit vectors.
%Fix $n \in \N_0$. Suppose that 
%for each $q \in \N_0^{\infty}$ with $|q|=n$, we have an element $w_q \in W$, such that the sum
%\[
%\sum_{|q|=n} \|w_q\|_W
%\]
%converges. Then the series
%\[
%\sum_{|q| = n}w_q p^q
%\]
%converges in $\Pol^n(V,W)$.
%\end{lemma}
%\begin{proof}
%If $v_1, \ldots, v_n \in V$ are unit vectors, then 
%\[
%|e^q(v_1, \ldots, v_n)| \leq 1.
%\]
%Hence
%\[
%\Bigl\| \sum_{|q| = n}w_q e^q (v_1, \ldots, v_n) \Bigr\|_W \leq  \sum_{|q| = n} \|w_q\|_W,
%\]
%which converges by assumption. So the series
%$
%\sum_{|q| = n}w_q e^q
%$
%converges in $\cB^n(V,W)$. Hence
%$
%\sum_{|q| = n}w_q S e^q
%$
%converges in $S\cB^n(V,W)$. 
%\end{proof}
The following lemma follows from the definition of the derivative.
\begin{lemma} \label{lem der poly}
The derivative of $p^q$ in~\eqref{eq def vq} is given by
\begin{equation*}
(p^q)'(u)(h) = \sum_{j=1}^{\infty} q_j\langle e^j, u\rangle^{q_j-1}
\langle e_j, h\rangle
\Bigl(\prod_{k\not= j} \langle e^{k}, u\rangle^{q_{k}}\Bigr),
\end{equation*}
for all $u,h \in V$.
\end{lemma}
%\begin{proof}
%This follows from the definition of the derivative.
%\end{proof}

\section{Compact derivatives and polynomials} \label{sec cpt der}

It is a nontrivial question in what sense differentiable maps between normed vector spaces can be approximated by polynomial maps \cite[e.g.]{AP, DAlessandro, DH, DHJ, Nemirovskii, NS}. 
In this section we discuss an approach to this problem that is suitable for our purposes. This discussion includes the further problem of approximating a polynomial by sums of the standard monomials of \cref{sec std mon}. The polynomials for which this is possible are the \emph{compact polynomials} introduced in \cref{sec cpt pol}. 

\cref{sec cpt der def} introduces \emph{compactly differentiable} maps. 
We combine these with Taylor's theorem to express the lowest order parts of such maps in terms of standard monomials. We discuss a class of examples of compactly differentiable maps relevant to the study of {\PDE}s.

\subsection{Compact multilinear maps}

Let $V$ and $W$ be Banach spaces.
Let $\cK^n(V,W) \subset \cB^n(V,W)$ be the image of the space
\[
\cK(V, \cK(V, \ldots, \cK(V,W) \cdots ))
\]
under the isomorphism \eqref{eq iso Bn}.
\begin{journal}
Using induction on $n$, one can show that  $\cK^n(V,W)$ is closed in $\cB^n(V,W)$, and hence a Banach space.
\end{journal}
\begin{ArXiv}
\begin{lemma}\label{lem Kn closed}
For every \(n\in\N\) the space $\cK^n(V,W)$ is closed in $\cB^n(V,W)$.
\end{lemma}
\begin{proof}
We use induction on $n$. For $n=1$ the claim is standard. Suppose the claim holds for $n$. Then 
\[
\cK^{n+1}(V,W) = \cK(V, \cK^n(V,W)),
\]
which is a closed subspace of $\cB(V, \cK^n(V,W))$. And that space
 is closed in $\cB^{n+1}(V,W)$ since $\cK^n(V,W)$ is closed in $ \cB^n(V,W)$ by the induction hypothesis.
\end{proof}
By \cref{lem Kn closed}, $\cK^n(V,W)$ is a Banach space.
\end{ArXiv}

Let 
 $\{e^j\}_{j=1}^{\infty} \subset V^*$ and $\{f_k\}_{k=1}^{\infty} \subset W$ be countable subsets whose spans are  dense. (So $V^*$ and $W$ are separable.) 
 %For each $j \in \N$, let $e^j \in V^*$ be given by taking inner products with $e_j$.
For any $\alpha \in \N^{n}$, consider the multilinear map
 \beq{eq def ealpha}
 e^{\alpha} := e^{\alpha_1} \otimes \cdots \otimes e^{\alpha_n}\colon V \times \cdots \times V \to \C.
 \eeq
 A Banach space has the approximation property if every compact operator on the space is a norm-limit of finite-rank operators. This is always true for Hilbert spaces, but we need to consider more general Banach spaces for applications. 
 \begin{journal}
The following result is standard in the case where $V$ and $W$ are Hilbert spaces.
%%For each $j \in \N$, let $e^j \in V^*$ be given by taking inner producs with $e_j$. 
\begin{proposition}\label{prop approx cpt}
If $V^*$ has the approximation property, then the space $\Span\{ e^j \otimes f_k : j,k \in \N\}$ is dense in $\cK(V,W)$.
\end{proposition}
\begin{proof}
 Since $V^*$ has the approximation property, the space of finite-rank operators (linear operators with finite-dimensional images) is dense in $\cK(V,W)$. See for example Proposition 4.12(b) in the book by Ryan~\cite{Ryan}. The space $\Span\{ e^j \otimes f_k : j,k \in \N\}$ is dense  in the space of finite-rank operators, so the claim follows.
\end{proof}
 \end{journal}
 %
% In \cref{prop Sobolev AP}, it is shown that Sobolev spaces on the circle have the approximation property, for example.
\begin{lemma} \label{lem Kn fin dense}
If $V^*$ has the approximation property, then
for every $n \in \N$, 
the span of $\{e^{\alpha} \otimes f_k: \alpha \in \N^n, k \in \N\}$ is dense in $\cK^n(V,W)$.
\end{lemma}
\begin{proof}
We prove this by induction on $n$. If $n=1$, then the claim is precisely \cref{prop approx cpt}\begin{ArXiv} in the appendix\end{ArXiv}. Now suppose that the claim holds for a given $n$. By definition,
\[
\cK^{n+1}(V,W) = \cK(V, \cK^{n}(V,W)).
\]
By the induction hypothesis, the set $\{e^{\alpha} \otimes f_k : \alpha \in \N^n, k \in \N\}$ has dense span in $\cK^n(V,W)$. Therefore, \cref{prop approx cpt}, with $W$ replaced by $\cK^{n}(V,W)$, implies that the set
\[
\{e^j \otimes e^{\alpha} \otimes f_k : j,k \in \N, \alpha \in \N^n\}
\]
has dense span in $\cK^{n+1}(V,W)$. This is precisely the claim for $n+1$.
\end{proof}

A Schauder basis of a Banach space $V$ is a subset $\{e_j\}_{j=1}^{\infty} \subset V$ such that for each $v \in V$, there are unique complex numbers $\{v^j\}_{j \in \N}$ such that
\[
\Bigl\|v - \sum_{j=1}^n v^j e_j\Bigr\|_V \to 0 \quad \text{as }n \to \infty. 
\]
A space with a Schauder basis has the approximation property.
\begin{lemma} \label{lem Kn Schauder}
If $\{e^j\}_{j=1}^{\infty}$ is a Schauder basis of $V^*$ and $\{f_k\}_{k=1}^{\infty}$ is a Schauder basis of $W$, then 
for every $n \in N$, 
the set $\{e^{\alpha} \otimes f_k : \alpha \in \N^n, k \in \N\}$ is a Schauder basis of $\cK^n(V,W)$.
\end{lemma}
\begin{proof}
\cref{lem Kn fin dense} implies that $\{e^{\alpha} \otimes f_k : \alpha \in \N^n, k \in \N\}$ has dense span. So it remains to show that if $a_{\alpha}^k \in \C$ are such that
\[
\sum_{\alpha \in \N^n} \sum_{k=1}^{\infty} a_{\alpha}^k e^{\alpha} \otimes f_k = 0,
\]
then $a_{\alpha}^k = 0$ for all $\alpha$ and $k$. Since $\{f_k\}_{k=1}^{\infty}$ is a Schauder basis of $W$, this reduces to the case where $W = \C$. 
\begin{journal}
In that case, one can prove the claim  by induction on $n$, using the fact that $\{e^j\}_{j=1}^{\infty}$ is a Schauder basis of $V^*$.
\end{journal}
\begin{ArXiv}
We prove the claim in that case, by induction on $n$.

If $n = 1$, then the claim follows since $\{e^j\}_{j=1}^{\infty}$ is a Schauder basis of $V^*$. So suppose that the claim holds for a given $n$, and let $a_{\alpha}^k \in \C$ be such that
\[
\sum_{\alpha \in \N^{n+1}}a_{\alpha}^k e^{\alpha}  = 0.
\]
Then for all $v_1, \ldots, v_n \in V$,
\[
\sum_{j=1}^{\infty} \Bigl[\sum_{\alpha \in \N^{n}} a_{(j, \alpha)}^k e^{\alpha} (v_1, \ldots, v_n)\Bigr] e^j= 0.
\]
Since, $\{e^j\}_{j=1}^{\infty}$ is a Schauder basis of $V^*$, this implies that for every $j \in \N$,
\[
\sum_{\alpha \in \N^{n}} a_{(j, \alpha)}^k e^{\alpha}  (v_1, \ldots, v_n)= 0.
\]
Because the sum 
\[
\sum_{\alpha \in \N^{n}} a_{(j, \alpha)}^k e^{\alpha}
\]
converges in $\cB^n(V, \C)$, we find that the sum converges to zero in this space. By the induction hypothesis, this implies that $a_{j, \alpha} = 0$ for every $j \in \N$ and $\alpha \in \N^n$.
\end{ArXiv}
\end{proof}

\begin{remark}
In the induction step in the proof of the special case of \cref{lem Kn fin dense} where $W$ is a Hillbert space, we still need the general version of \cref{prop approx cpt}, where $W$ is a Banach space. This is because $\cK^n(V,W)$ is only a Banach space, even if $W$ is a Hilbert space.
\end{remark}

The subspace $S\cK^n(V,W)$ of symmetric operators in $\cK^n(V,W)$ is closed in $\cB^n(V,W)$, since it is the intersection of the closed subspaces  $S\cB^n(V,W)$ and $\cK^n(V,W)$ 
\begin{ArXiv}
(\cref{lem SBn closed,lem Kn closed}). 
\end{ArXiv}
Hence $S\cK^n(V,W)$ is a Banach space with respect to the norm \eqref{eq norm Bn}.

A Schauder basis $\{e_j\}_{j=1}^{\infty}$ of a Banach space $V$ is unconditional if there is a constant $C>0$ such that for all $a^j, \varepsilon_j \in \C$ with $|\varepsilon_j|=1$, and all $n \in \N$,
\[
\Bigl\| \sum_{j=1}^{n}\varepsilon_j a^j e_j \Bigr\|_V \leq C \Bigl\| \sum_{j=1}^{n}a^j e_j \Bigr\|_V.
\]
In that case, convergence of $\sum_{j=1}^{n}a^j e_j$ implies convergence of $\sum_{j \in A}a^j e_j$, for every $A \subset \N$.

\begin{lemma}\label{lem uncond Kn}
Suppose that $V$ and $W$ are Hilbert spaces, and that $\{e_j\}_{j=1}^{\infty}$ and $\{f_k\}_{k=1}^{\infty}$ are orthogonal sets in $V$ and $W$ respectively, with dense spans. Let $e^j \in V^*$ be defined by taking inner products with $e_j$. Then  $\{e^{\alpha} \otimes f_k : \alpha \in \N^n, k \in \N\}$ is an unconditional Schauder basis of $\cK^n(V,W)$.
\end{lemma}
\begin{proof}
The set $\{e^{\alpha} \otimes f_k : \alpha \in \N^n, k \in \N\}$ is a Schauder basis of $\cK^n(V,W)$ 
by  \cref{lem Kn Schauder}. It remains to show that it is unconditional.
By rescaling the vectors $e_j$ and $f_k$, we  reduce the proof to the case where $\{e_j\}_{j=1}^{\infty}$ and $\{f_k\}_{k=1}^{\infty}$ are Hilbert bases. In that case,  for all finite subsets $A \subset \N_0^{n} \times \N$ and all $a_{\alpha}^k \in \C$,
\[
\Bigl\| \sum_{(\alpha, k) \in A}  a_{\alpha}^k e^{\alpha}\otimes f_k\Bigr\|_{\cB^n(V,W)}^2 = 
\sup_{\alpha \in \N_0^m} \sum_{k \in \N;\, (\alpha, k) \in A} |a_{\alpha}^k|^2.
\]
\end{proof}

\begin{lemma}\label{lem ex Kn}
Let $U$, $V$ and $W$ be normed vector spaces, and \(n\in\N\). Let $\lambda \in \cB^n(V,W)$, and $a_1, \ldots, a_n \in \cK(U,V)$. Define $\nu\colon U \times \cdots \times U \to W$ by
\[
\nu(u_1, \ldots, u_n) = \lambda(a_1u_1, \ldots, a_n u_n),
\]
for all $u_1, \ldots, u_n \in U$. Then $\nu \in \cK^n(U, W)$.
\end{lemma}
\begin{proof}
We use induction on $n$. For $n = 1$, $v \in \cK^1(U,W)$ because the composition of a compact operator and a bounded operator is compact. Suppose that the claim holds for a given $n$. Let  $\lambda \in \cB^{n+1}(V,W)$, and $a_1, \ldots, a_{n+1} \in \cK(U,V)$. For a fixed $u \in U$, define $\nu_u \in \cB^n(U,W)$
%
%\colon U \times \cdots \times U \to W$ ($n$ factors $U$) 
by
\[
\nu_u(u_1, \ldots, u_n) = \lambda(a_1 u_1, \ldots,  a_n u_n, a_{n+1} u),
\]
for $u_1, \ldots, u_n \in U$. For a fixed $v \in V$, define $\lambda_{v} \in \cB^n(V,W)$
%
%\colon V \times \cdots \times V \to W$ ($n$ factors $V$) 
by
\[
\lambda_v(v_1, \ldots, v_n) = \lambda(v_1, \ldots,  v_n, v),
\]
for $v_1, \ldots, v_n \in V$. Then for all $u_1, \ldots, u_n \in U$,
\[
 \nu_u(u_1, \ldots, u_n) = \lambda_{a_{n+1}u}(a_1 u_1, \ldots, a_n u_n).
 \]
 So by the induction hypothesis, $\nu_u \in \cK^n(U,W)$. In this way, we obtain the map
 \[
 \tilde \nu\colon U \to \cK^n(U,W),
 \]
 mapping $u \in U$ to $\nu_u$. It remains to show that $\tilde \nu$ is a compact operator.
 
 Define $\tilde \lambda \in \cB(V,  \cK^n(U,W))$ by
 \[
 \tilde \lambda(v)\colon (u_1, \ldots, u_n) \mapsto \lambda(a_1 u_1, \ldots, a_n u_n, v),
 \]
 for $v \in V$ and $u_1, \ldots, u_n \in U$. (This map takes values in $\cK^n(U,W)$ by the induction hypothesis.) 
% Since $\lambda \in \cB^{n+1}(V,W)$ and $a_1, \ldots, a_n$ are compact, hence in particular bounded, the map $\tilde \lambda$ is bounded. 
 Since $a_{n+1}$ is compact and $\tilde \lambda$ is bounded, we find that   $\tilde \nu = \tilde \lambda \circ a_{n+1}$ is a compact operator.
\end{proof}

\subsection{Compact polynomial maps} \label{sec cpt pol}

%The space $\cB^n(V,W)$ contains the closed subspace $\cK^n(V,W)$ of operators that are compact in each entry. %Let $S\cK^n(V,W)$ be the closed subspace of symmetric operators in $\cK^n(V,W)$. 

%Let 
%\[
%S\colon \cK^n(V,W) \to S\cK^n(V,W)
%\]
%be the symmetrisation operator: for all $\lambda \in \cK^n(V,W)$ and $v_1, \ldots, v_n$,
%\[
%(S\lambda)(v_1, \ldots, v_n) = \frac{1}{n!}\sum_{\sigma \in \Sigma_n} \lambda(v_{\sigma(1)}, \ldots, v_{\sigma(n)}),
%\]
%where $\Sigma_n$ is the permutation group of $\{1,\ldots, n\}$.
%
%An element $\lambda \in S\cK^n(V,W)$ (indeed, any element of $\cB^n(V,W)$) defines a map
%\[
%p_{\lambda}\colon V \to W
%\]
%by
%\beq{eq def monomial}
%p_{\lambda}(v) = \lambda(v, \ldots, v).
%\eeq
\begin{definition}
A \emph{compact homogeneous polynomial map} of degree $n$ from $V$ to $W$ is a map of the form $p_{\lambda}$ as in \eqref{eq def monomial}, for $\lambda \in \cK^n(V,W)$. We write $\cK\Pol^n(V,W)$  for the space of such maps.  This space inherits a norm from the space $\Pol^n(V,W)$ it is contained in.

 If $\lambda \not=0$, then the \emph{degree} of  $p_{\lambda}$ is $n$.
A \emph{compact polynomial map} from $V$ to $W$ is a finite sum of compact homogeneous polynomial maps. The \emph{degree} of a compact polynomial map is the degree of its highest-degree homogeneous term.
We write $\cK\Pol(V,W)$ for the space of all compact polynomial maps between these spaces. 
%They inherit norms from the spaces $\Pol^n(V,W)$ and $\Pol(V,W)$ they are contained in.
%
%inherit topologies from the spaces $S\cK^n(V,W)$ via the linear isomorphisms $\lambda \mapsto p_{\lambda}$.
\end{definition}
The isometric isomorphism $S\cB^n(V,W) \cong \Pol^n(V,W)$ restricts to an isometric isomorphism $S\cK^n(V,W) \cong \cK\Pol^n(V,W)$. So $\cK\Pol^n(V,W)$ is a closed subspace of the Banach space $\Pol^n(V,W)$, and hence is a Banach space itself.

%, where $S\cK^n(V,W)$ is the closed subspace of symmetric operators in $\cK^n(V,W)$.

%One can also define bounded polynomial maps by replacing compact operators by bounded ones everywhere. But the compact variant will be most useful to us.
%\begin{remark}
%The definition of compact polynomial maps is motivated by \cref{lem basis} below.
%\end{remark}

For every $w \in W$, an operator of the form $e^q \otimes w$, with 
 $e^q$ as in \eqref{eq def eq}, is an element of $\cK^n(V,W)$. Indeed, $e^q \otimes w$ is an iteration of rank-one operators,
So $p^q \otimes w \in \cK\Pol^n(V,W)$.
%
%\subsection{Hilbert spaces}
%
%Now, again suppose that $V$ and $W$ are Hilbert spaces. Fix (orthonormal) Hilbert bases $\{e_j\}_{j=1}^{\infty}$ of $V$ and $\{f_k\}_{k=1}^{\infty}$ of $W$. For $j \in \N$, let $e^j \in V^*$ be given by taking inner products with $e_j$.
%
%Let $\N_0^{\infty}$ be the set of sequences in $\N_0$ with finitely many nonzero terms. 
%For $q \in \N_0^{\infty}$, let $|q|$ be the sum of its terms. If $|q|=n$, and $m$ is the largest number for which $q_m \not=0$, then we define the element
%\[
%e^q := \underbrace{e^1 \otimes \cdots \otimes e^1}_{\text{$q_1$ factors}} \otimes \cdots \otimes 
% \underbrace{e^m \otimes \cdots \otimes e^m}_{\text{$q_m$ factors}} 
%\in \cK^n(V, \C).
%\]
%In other words,
%\begin{multline*}
%e^q(v_1, \ldots, v_n) =\\ (e_1, v_1) \cdots (e_1, v_{q_1}) \cdot (e_2, v_{q_1 + 1})
% \cdots (e_2, v_{q_1 + q_2}) \cdots
%(e_m, v_{q_1 + \cdots + q_{m-1} +1}) \cdots  (e_m, v_n). 
%\end{multline*}
%The operators $e^q$ of this form are finite-rank in each entry, so in particular compact. We therefore have
%\[
%Se^q \in S\cK^n(V, \C).
%\]
%We write $p^q := p_{Se^q}$ for the corresponding homogeneous polynomial. For $v \in V$, we will write
%\[
%v^q := p^q(v) = \prod_{j=1}^{\infty} (e_j, v)^{q_j}.
%\]
%Note that this product is finite, and depends on the basis $\{e_j\}$.
%
%Part (b) of the 

%As before, we write $p^q := p_{e^q}$ for $q \in \N_0^{\infty}$, with $e^q$ as in \eqref{eq def eq}.

The following proposition is the reason why we are interested in {compact} polynomial maps.
\begin{proposition} \label{prop basis}
Suppose that $V$ and $W$ are  Banach spaces, that $V^*$ has a Schauder basis $\{e^j\}_{j=1}^{\infty}$, and that $W$ has a Schauder basis $\{f_k\}_{k=1}^{\infty}$.
%
%. %, and that $V^*$ has the approximation property.
% Let  and $\{f_k\}_{k=1}^{\infty} \subset W$ be countable, linearly independent subsets whose spans are dense. (So $V^*$ and $W$ are separable.) 
 Then
the elements
\beq{eq pq fk}
p^q \otimes f_k \quad \in \cK \Pol^n(V, W),
\eeq
where the multi-index $q$ ranges over the elements of $\N_0^{\infty}$ with $|q|=n$, and $k$ ranges over the positive integers, 
form a Schauder basis of  $\cK \Pol^n(V,W)$.
%\end{enumerate}
\end{proposition}
\begin{proof}
Consider the space
\[
X := \overline{\Span\{e^\alpha \otimes f_k :\alpha \in \N^n, \alpha_1 \leq \cdots \leq \alpha_n,  k \in \N\}}.
\]
%Define the map $\psi\colon X \to Y$ by $\psi(e^{\alpha} \otimes f_k) = e^{q(\alpha)} \otimes f_k$, where
% $q(\alpha) \in \N_0^{\infty}$ is defined by
%\[
%q(\alpha)_j= \# \{m \in \N : \alpha_m = j\}.
%\]
%Then\footnote{Note that $e^{\alpha}$, for $\alpha \in \N^n$, and $e^{q}$, for $q\in \N_0^{\infty}$, are defined differently. Compare \eqref{eq def eq} and \eqref{eq def ealpha}.} $S e^{\alpha} = Se^{q(\alpha)}$. We therefore have a commuting diagram of bounded linear isomorphisms with bounded inverses
%\[
%\xymatrix{
%X \ar[r]^-{\psi} \ar[dr]_-{S} & Y \ar[d]^-{S} \\
%& S\cK^n(V, W).
%}
%\]
\cref{lem Kn Schauder} implies that the set of $e^{\alpha} \otimes f_k$ where $\alpha_1 \leq \cdots \leq \alpha_n$ is a Schauder basis of $X$. And $S\colon X \to S\cK^n(V,W)$ is a bounded linear isomorphism with bounded inverse. Since such isomorphisms map Schauder bases to Schauder bases, we find that the set $Se^{\alpha} \otimes f_k$, for non-decreasing $\alpha$ as above, is a Schauder basis of $S\cK^n(V,W)$.

For  $\alpha \in \N$ with non-decreasing entries, define $q(\alpha) \in \N_0^{\infty}$  by
\[
q(\alpha)_j= \# \{m \in \N : \alpha_m = j\}.
\]
Then $e^{\alpha} = e^{q(\alpha)}$. (Note that $e^{\alpha}$, for $\alpha \in \N^n$, and $e^{q}$, for $q\in \N_0^{\infty}$, are defined differently; compare \eqref{eq def eq} and \eqref{eq def ealpha}.)
Every sequence in $\N_0^{\infty}$ occurs in exactly one way as $q(\alpha)$, for alpha as above, so $Se^{q} \otimes f_k$, where $q \in \N^{\infty}_0$ and $k \in \N$, is a Schauder basis of $S\cK^n(V,W)$.
Since $p^q = p_{Se^q}$, the claim follows.
%
%Bounded linear isomorphisms with bounded inverses map Schauder bases to Schauder bases, so the claim follows.
\end{proof}

A reformulation of \cref{prop basis} is that for every compact polynomial map $p \in \cK\Pol^n(V, W)$, there are unique complex numbers $a^k_q$ such that
\[
p = \sum_{q, k} a^k_q p^q \otimes f_k,
\]
where the sum converges in the norm on $\Pol^n(V,W)$. Conversely, all polynomial maps $p$ of this form are compact.

\begin{lemma}\label{lem uncond KPn}
In the setting of \cref{lem uncond Kn}, the set $\{p^q \otimes f_k : |q| = n, k \in \N\}$ is an unconditional Schauder basis of $\cK \Pol^n(V,W)$.
\end{lemma}
\begin{proof}
The proof is analogous to the proof of \cref{prop basis}, where we now use \cref{lem uncond Kn} instead of \cref{lem Kn Schauder}, and we use the fact that bounded linear isomorphisms with bounded inverses map unconditional Schauder bases to unconditional Schauder bases.
%The map 
%\[
%S\colon \overline{\Span\{e^q \otimes f_k : |q|=n, k \in \N\}} \to S\cK^n(V,W)
%\]
%is a bounded linear isomorphism with bounded inverse. Here we use \cref{lem SBn closed,lem Kn fin dense}. \cref{lem uncond Kn} implies that $\{p^q \otimes f_k : |q| = n, k \in \N\}$ is an unconditional Schauder basis of the space on the left-hand side. So the claim follows, because $p^q = p_{Se^q}$ for all $q$.
\end{proof}

\begin{lemma}\label{lem compos cpt pol}
If $p_1 \in \cK\Pol^m(U,V)$ and $p_2 \in \cK\Pol^n(V,W)$, then $p_2 \circ p_1 \in \cK\Pol^{mn}(U,W)$.
\end{lemma}
\begin{proof}
The proof is similar to the proof of \cref{lem compos bdd pol}, with bounded multilinear maps replaced by compact ones.
\end{proof}
\begin{remark}
Other notions of compact polynomial maps were studied by Gonzalo, Jaramillo and Pe{\l}czy\'{n}sky in \cite{GJ, Pelczynski}.
\end{remark}

\subsection{Compactly differentiable maps} \label{sec cpt der def}

Let $V$ and $W$ be normed vector spaces, let $U \subset V$ be an open subset containing a vector $u$, and let  $f\colon U \to W$ be $n$ times differentiable at $u$.
\begin{definition}
The map $f$ is  $n$ times \emph{compactly differentiable} at $u$ if
\[
f^{(n)}(u) \in \cK^n(V,W).
\]
\end{definition}

%Suppose that $V$ and $W$ are inner product spaces, with given Hilbert bases.
If  $f$ is  $n$ times {compactly differentiable} at $u$, then by \cref{lem symm}, 
\[
f^{(n)}(u) \in S\cK^n(V,W).
\]
Then
the map $h \mapsto f^{(n)}(u)h^n$ of \eqref{eq fn} is the compact polynomial map associated to $f^{(n)}(u)$.
Together with \cref{thm Taylor} and \cref{prop basis}, this leads to the following conclusion.
\begin{corollary} \label{cor Taylor cpt}
Suppose that $V$ and $W$ are  Banach spaces, that $V^*$ has a Schauder basis $\{e^j\}_{j=1}^{\infty}$, and that $W$ has a Schauder basis $\{f_k\}_{k=1}^{\infty}$.
Suppose $f$ is $n+1$ times differentiable, and $k$ times compactly differentiable for every $k \leq n$. 
%Suppose that $\|f^{(n+1)}(\xi)\|_W \leq M$ for all $\xi$ in a closed ball around $u$ contained in $U$. 
%Suppose that $\|f(u+h)\|_W = O(\|h\|^{n})$ as $h \to 0$ in $V$. 
Then  there are unique complex numbers $a^k_q$ such that
%
%Then for all $h$ in this ball,
%
%
%Let $h \in V$, and suppose that the closed line segment from $u$ to $u+h$ is contained in $U$. Suppose that $f(u+h) = O(\|h\|^{n+1})$ as $h \to 0$ in $V$. Then  there are complex numbers $a^k_q$ such that
\beq{eq Taylor cpt 1}
f(u+h) = %\sum_{q, k}
\sum_{q \in \N_0^{\infty};\, |q|\leq n} 
\sum_{k=1}^{\infty}
a^k_q  h^q f_k + O(\|h\|_W^{n+1}),
\eeq
where the part of the sum where $|q|=m$ converges as a function of $h$ in the norm on $\Pol^m(V,W)$, for $m = 0, \ldots, n$.
\end{corollary}
(Note that, in \eqref{eq Taylor cpt 1}, on the left-hand side $f$ is a map from $U$ to $W$, whereas on the right-hand side, $f_k$ is an element of $W$.) 
%Convergence of the sum \eqref{eq Taylor cpt} is in the complete subspace of $\Pol(V,W)$ of polynomial maps of order at most $n$.

%\Todo: reformulate this for $W$ with the approximation property! Or a Hilbert space.

\begin{lemma} \label{lem cpt poly diffble}
A compact polynomial map is infinitely compactly differentiable.
\end{lemma}
\begin{proof}
We show that the derivative of  every homogeneous compact polynomial $p_{\lambda} \in \cK\Pol^n(V,W)$, for $\lambda \in S\cK^n(V,W)$, is a compact polynomial in $\cK\Pol^{n-1}(V,\cK(V,W))$. This implies the claim by induction on $n$.
As in the proof of \cref{lem pol diffble}, 
 $p'_{\lambda}(u)h = n\lambda(h,u,\ldots, u)$ for all $u,h \in V$. In other words, $p'_{\lambda}(u) = p_{\nu}$, with $\nu = n\lambda$, where we view $\lambda$ as an element of $\cK^{n-1}(V, \cK(V,W))$. This shows that $p'_{\lambda}(u) \in \cK\Pol^{n-1}(V, \cK(V,W))$.
\end{proof}

\begin{lemma} \label{lem cpt chain rule}
Let $U$, $V$ and $W$ be normed vector spaces and let $f \colon U \to V$ and $g\colon V \to W$ be  differentiable maps. If either $f$ or $g$ is compactly differentiable, then so is $g \circ f$.
\end{lemma}
\begin{proof}
\cref{lem chain rule} implies that for all $u \in U$,
\[
(g\circ f)'(u) = g'(f(u)) \circ f'(u).
\]
If $f$ is compactly differentiable, then $f'(u) \in \cK(U,V)$. If $g$ is compactly differentiable, then $g'(f(u)) \in \cK(V,W)$. In either case, we find that  $(g\circ f)'(u) \in \cK(U,V)$.
\end{proof}

\begin{lemma} \label{lem cpt diff incl}
Let $U$, $V$ and $W$ be normed vector spaces, let $f\colon V \to W$ be  $n$ times differentiable, and suppose that $U$ is a subspace of $V$, with compact inclusion map $j \colon U \hookrightarrow V$. Then $f \circ j$ is $n$ times compactly differentiable as a map from $U$ to $W$.
\end{lemma}
\begin{proof}
Let $u, h_1, \ldots, h_n \in U$. Then
\[
(f \circ j)^{(n)}(u)(h_1, \ldots, h_n) = f^{(n)}(j(u))(j(h_1), \ldots, j(h_n)).
\] 
So the claim follows from \cref{lem ex Kn}.
\end{proof}

\begin{remark}
It is possible for a map to be $n$ times compactly differentiable, but not $n-1$ times. For example, the first derivative of the identity operator on an infinite-dimensional Banach space $V$ is the identity map itself, and not a compact operator. But its higher-order derivatives are zero, and hence compact. 
\end{remark}

\subsection{A class of compactly differentiable maps} \label{sec ex f}

We end this section by discussing a class of compactly differentiable maps (specifically, compact polynomials) that are relevant to the study of nonlinear \PDE s. These maps are polynomial expressions in derivatives of functions; see \cref{prop ex cpt poly Sobolev} below.

Let $\Omega \subset \R^m$ be a bounded open subset with $C^1$ boundary.
For $k \in \N_0$ and $p>0$, consider the Sobolev space $W^{k,p}(\Omega)$.
% as the completion of $C^{\infty}_c(\Omega)$ in the norm
%\[
%\|u\|_{W^{k,p}} := \sum_{\alpha \in \N_0^m;\, |\alpha|\leq k} \bigl\|\frac{\partial^{\alpha} u}{\partial x^{\alpha}}\bigr\|_{L^p}.
%\]
\begin{lemma}\label{lem mult Sobolev}
Let $n \in \N$ and $p>1$.
Pointwise multiplication of $n$ functions defines a map $\mu \in \cB^n(W^{k, np}(\Omega), W^{k,p}(\Omega))$.
\end{lemma}
\begin{proof}
Let $n \in \N$.
By H\"older's inequality, for all $u_1, \ldots, u_n \in L^{np}(\Omega)$,
\beq{eq Holder 1}
\|u_1 \cdots u_n\|_{L^{p}(\Omega)} \leq \|u_1\|_{L^{np}(\Omega)} \cdots \|u_n\|_{L^{np}(\Omega)}. 
\eeq
For $\alpha \in \N_0^m$, there are combinatorial constants $c^{\alpha}_{\beta}$, for  $\beta = (\beta^{(1)}, \ldots, \beta^{(n)})$, with $\beta^{(1)}, \ldots, \beta^{(n)} \in \N_0^m$ such that $|\beta^{(1)}|+ \cdots + |\beta^{(n)}| \leq |\alpha|$, such that for all 
 $u_1, \ldots, u_n \in C^{\infty}_c(\Omega)$, we have the generalised Leibniz rule
 \[
 \frac{\partial^{\alpha} (u_1 \cdots u_n)}{\partial x^{\alpha}} = \sum_{|\beta^{(1)}|+ \cdots + |\beta^{(n)}| \leq |\alpha|} c^{\alpha}_{\beta} \frac{\partial^{\beta^{(1)}} u_1}{\partial x^{\beta^{(1)}}} \cdots \frac{\partial^{\beta^{(n)}} u_n}{\partial x^{\beta^{(n)}}}. 
 \] 
 Together with \eqref{eq Holder 1}, this implies that
 \[
 \|u_1 \cdots u_n\|_{W^{k, p}} \leq 
 \Bigl( \sum_{|\alpha| \leq k} \,\, \sum_{|\beta^{(1)}|+ \cdots + |\beta^{(n)}| \leq |\alpha|}
  c^{\alpha}_{\beta}\Bigr) \|u_1\|_{W^{k,np}} \cdots  \|u_n\|_{W^{k,np}}. 
 \]
\end{proof}

\begin{proposition} \label{prop ex cpt poly Sobolev}
Let $k, l,m_1, m_2, n \in \N$ and $p>1$, with $k \leq l$.
Let
\[
D_1, \ldots, D_n\colon C^{\infty}_c(\Omega; \R^{m_1}) \to  C^{\infty}_c(\Omega; \R^{m_2})  
\]
be linear partial differential operators of orders smaller than $k$.  
%For $q \in \N_0^n$, we set
%\[
%(Du)^q := (D_1u)^{q_1} \cdots (D_n u)^{q_n}. 
%\]
%
%
%Define $f\colon C_c^{\infty}(\Omega; \R^{m}) \to C_c^{\infty}(\Omega; \R^{m})$ by $f(t, u) = (f_1(t, u), \ldots, f_m(t, u))$, where for each $j$,
%\beq{eq f cpt poly Sob}
%f_j(t, u) = \sum_{q \in \N_0^s, |q| \leq s} a_q^j(t)(Du)^q% \cdots D_n u,
%\eeq
%for $t \in \R$ and $u \in C_c^{\infty}(\Omega; \R^{m})$.
%
%
Fix complex numbers $a_q^j$, for $q \in \N_0^{n}$ with $|q| \leq n$, and $j \in \{1, \ldots, m_2\}$. Define $f\colon C_c^{\infty}(\Omega; \R^{m_1}) \to C_c^{\infty}(\Omega; \R^{m_2})$ by $f(u) = (f_1(u), \ldots, f_{m_2}(u))$, where for each $j$,
\beq{eq f cpt poly Sob}
f_j(u) = \sum_{q \in \N_0^n, |q| \leq n} a_q^j(Du)^q% \cdots D_n u,
\eeq
for $u \in W^{l,np}(\Omega; \R^{m_1})$, and with $(Du)^q$ as in \eqref{eq Duq}, defines a compact polynomial map 
\[
f \in \cK \Pol^n(W^{l,np}(\Omega; \R^{m_1}), W^{l-k, p}(\Omega; \R^{m_2})).
\]
\end{proposition}
\begin{proof}
We first consider the case where $m_2 = 1$, and $a_q = 1$ if $q = (1,\ldots, 1)$, and zero otherwise.
By \cref{lem mult Sobolev}, pointwise multiplication defines a map $\mu \in \cB^n(W^{l-k, np}(\Omega), W^{l-k, p}(\Omega))$. 

By Rellich's lemma, boundedness of $\Omega$ implies that the maps $D_1, \ldots, D_n$ extend to compact operators
\[
D_1, \ldots, D_n\colon W^{l,p}(\Omega)\to W^{l-k,p}(\Omega).
\]
The map 
\[
\nu\colon W^{l, np}(\Omega) \times \cdots \times W^{l, np}(\Omega) \to W^{l-k, p}(\Omega)
\]
defined by
\[
\nu(u_1, \ldots, u_n) = \mu(D_1 u_1, \ldots, D_n u_n)
\]
for $u_1, \ldots, u_n \in W^{l,np}(\Omega)$, is in $\cK^n(W^{l, np}(\Omega), W^{l-k, p}(\Omega))$ by \cref{lem ex Kn}. Hence $p_{\nu}$ is an element of $\cK\Pol^n(W^{l, np}(\Omega), W^{l-k, p}(\Omega))$.

Every component of a general map of the form \eqref{eq f cpt poly Sob} is a finite sum of maps of the form $p_{\nu}$ as above, applied to the components of $f$. Hence it is  in $\cK\Pol^n(W^{l, np}(\Omega; \R^{m_1}), W^{l-k, p}(\Omega; \R^{m_2}))$.
\end{proof}

\begin{example}\label{ex cpt poly burgers}
\begin{ArXiv}
Consider the map $f$ from \cref{sec ex burgers}, mapping $u$ to $u' u$. We now view $f$ as a map from $W^{l,2p}(\Omega)$  to $W^{l-k, p}(\Omega)$,  for $k \geq 1$,  $l \geq k$ and $p>1$. 
\end{ArXiv}
\begin{journal}
Let $k \geq 1$,  $l \geq k$ and $p>1$ be integers. Let $\Omega \subset \R$ be a bounded open interval.
Consider the map $f$ from $W^{l,2p}(\Omega)$  to $W^{l-k, p}(\Omega)$, mapping $u \in W^{l,2p}(\Omega)$ to $u' u$. 
\end{journal}
Taking $m_1 = m_2 = 1$, $n = 2$,  $D_1 u = u'$ and $D_2 u = u$, for $u \in W^{l,2p}(\Omega)$, in \cref{prop ex cpt poly Sobolev}, we find that $f$ is a compact polynomial in $\cK \Pol^2(W^{l,2p}(\Omega), W^{l-k,p}(\Omega))$ for every $k>1$. Hence, by \cref{lem cpt poly diffble}, $f$ is in particular infinitely compactly differentiable.
For $k=1$, the map $f$ is only a bounded polynomial in $\Pol^2(W^{l,2p}(\Omega), W^{l-1, p}(\Omega))$.
\end{example}

\begin{remark}\label{rem ext ex cpt poly}
\cref{prop ex cpt poly Sobolev} extends directly 
to relatively compact open subsets $\Omega$ of manifolds. The latter extension is relevant, for example, if one uses periodic boundary conditions, so that one works with
 with functions on a torus.
\end{remark}

\section{Polynomial and differentiable maps on graded Fr\'echet spaces} \label{sec Frechet}

Apart from polynomial and differentiable maps between normed vector spaces, we also use such maps between graded Fr\'echet spaces, defined in terms of nested sequences of Banach spaces, as in \cref{def nested seq}. In this section, we discuss some further properties of such spaces, and in particular what it means for two sequences of Banach spaces defining such a space to be comparable.

\subsection{Properties of nested sequences of Banach spaces} \label{sec seq sub}

%Let $V$ be a vector space. 
%\begin{definition} \label{def nested seq}
%By a \emph{nested sequence of Banach spaces}, we will mean a sequence
%$\{V_k\}_{k=1}^{\infty}$  of  Banach spaces such that 
% \begin{itemize} 
% \item  for every $k$, $V_{k+1} \subset V_k$, and the inclusion map is bounded  (\Todo: is boundedness really necessary here? Safe to assume in any case.);
%% \item for every $k$, the spaces $V_k$ is equipped with a
%% complete norm $\|\cdot \|_{V_k}$  such that the inclusion map $V_{k+1}\hookrightarrow V_k$ is bounded. 
% \item the intersection $V_{\infty}:=\bigcap_{k=1}^{\infty}V_k$ is dense in $V_k$ for every $k$.
%  \end{itemize}
% We then consider $V_{\infty}$ as a Fr\'echet space with the topology induced by the norms $\|\cdot \|_{V_k}$.
%
%A \emph{compactly} nested sequence of Banach spaces is such a sequence such that for every $l \in \N$, there is a $k \in \N$ such that the inclusion $V_k \subset V_l$ is compact.
% \end{definition}
 
 Let $\{V_k\}_{k=1}^{\infty}$ be a nested sequence of Banach spaces, as in \cref{def nested seq}. Their intersection $V_{\infty}$ is a graded Fr\'echet space.
 \begin{definition}
%\ajr{Clarify what is being defined here: is it \(\cB^n(V_{\infty})\)?  Also, is \(k\lessgtr l\)?}
The space $\cB^n(V_{\infty})$ consists of the multilinear maps $\lambda \colon V_{\infty} \times \cdots \times V_{\infty} \to V_{\infty}$ (with $n$ factors $V_{\infty}$) such that for every $l \in \N$, there is a $k \in \N$ such that $\lambda$ extends continuously to a map in $\cB^n(V_k, V_l)$. 
 \end{definition}
Note that $\cB(V_{\infty}) = \cB^1(V_{\infty})$. 
 
% \begin{definition} \label{def On}
% A map $f\colon V_{\infty} \to V_{\infty}$ is \emph{of order $n$}, written as $f = \cO(n)$, if for every $l \in \N$, there is a $k \in \N$ such that $\| f(v)\|_{V_l} = O(\|v\|_{V_k}^n)$ as $v\to 0$ n $V_k$.
% 
% If $I$ is an open interval, a map  $f\colon I \times V_{\infty} \to V_{\infty}$ is \emph{of order $n$} written as $f = \cO(n)$, if for every $l \in \N$, there is a $k \in \N$ such that $\| f(t, v)\|_{V_l} = O(\|v\|_{V_k}^n)$ as $v\to 0$ n $V_k$, uniformly in $t$ in compact subsets of $I$.
% \end{definition}

% 
% \subsection{Polynomial and differentiable maps on nested sequences of Banach spaces}
% 
%As before,  let $\{V_k\}_{k=1}^{\infty}$ be a nested sequence of Banach spaces. 
 
\begin{definition} 
%\ajr{is \(k\lessgtr l\)?}
We write $\Pol^n(V_{\infty})$ for the space of all maps $p\colon V_{\infty} \to V_{\infty}$ such that for every $l \in \N$, there is a $k \in \N$ such that $p$ extends continuously to a polynomial in $\Pol^n(V_k, V_l)$. The space $\cK\Pol^n(V_{\infty})$ is defined analogously.
\end{definition}
A feature of the spaces  $\Pol^n(V_{\infty})$ and  $\cK\Pol^n(V_{\infty})$ that is useful to us, is that they admit natural compositions.
If $p_1 \in \Pol^m(V_{\infty})$ and $p_2 \in \Pol^n(V_{\infty})$, then 
%
%
%it follows from the definition that $p_1(V_{\infty}) \subset V_{\infty}$ and  $p_2(V_{\infty}) \subset V_{\infty}$. Hence the composition $p_1\circ p_2$ is well-defined. By 
%
\cref{lem compos bdd pol} implies that $p_2 \circ p_1$ lies in $\Pol^{mn}(V_{\infty})$. Similarly, \cref{lem compos cpt pol} implies that  $p_2 \circ p_1 \in \cK\Pol^{mn}(V_{\infty})$ if $p_1 \in \cK\Pol^m(V_{\infty})$ and $p_2 \in \cK\Pol^n(V_{\infty})$.

\begin{definition}
%\ajr{is \(k\lessgtr l\)?}
An $n$ times \emph{(compactly) differentiable map from $V_{\infty}$ to itself} is a map $f\colon V_{\infty} \to V_{\infty}$ such that for every $l \in \N$, there is a $k \in \N$ such that $f$ extends to an $n$ times (compactly) differentiable map from $V_k$ to $V_l$.
\end{definition}
\cref{lem chain rule,lem cpt chain rule} imply that the spaces of differentiable and compactly differentiable maps from $V_{\infty}$ to itself are closed under composition.
\cref{lem cpt diff incl} implies that if $\{V_k\}_{k=1}^{\infty}$ is a compactly nested sequence of Banach spaces, then all $n$ times differentiable maps from $V_{\infty}$ to itself are $n$ times compactly differentiable.

It follows directly from \cref{def On} that if $f, g\colon V_{\infty} \to V_{\infty}$ are of orders $m$ and $n$, respectively, then $g\circ f$ is of order $m+n$. We also have the following lemma.

 \begin{lemma} \label{lem der Om+n}
Consider two maps $f, g\colon V_{\infty} \to V_{\infty}$, where $f$ is differentiable and $g$ is of order $m$. Suppose that for every $l \in \N$, there is a $k \in \N$ such that $f\colon V_k\to V_l$ is differentiable, and $\|f'(v)\|_{\cB(V_k, V_l)} = O(\|v\|^n_{V_k})$ as $v \to 0$ in $V_k$. Then the map $f' \circ g\colon V_{\infty} \to V_{\infty}$, mapping $v \in V_{\infty}$ to $f'(v)(g(v))$ is of order $n+m$.
 \end{lemma}
 \begin{proof}
 Let $l \in \N$. Let $k \in \N$ be such that $f\colon V_k \to V_l$ is differentiable, and its derivative satisfies the estimate in the lemma. Let $k' \in \N$ be such that $\|g(v)\|_{V_k} = O(\|v\|^m_{V{k'}})$ as $v \to 0$ in $V_{k'}$. Then $\| (f' \circ g)(v)\|_{V_l} = O(\|v\|_{V_{k'}}^{m+n})$ as $v \to 0$ in $V_{k'}$.
 \end{proof}

An example of a situation where the condition on $f$ in \cref{lem der Om+n} is satisfied is the following.
\begin{lemma} \label{lem der poly On-1}
Let $p \in \Pol^n(V_{\infty})$, for $n \geq 2$. Then for every $l \in \N$, there is a $k \in \N$ such that $\|p'(v)\|_{\cB(V_k, V_l)} = O(\|v\|^{n-1}_{V_k})$ as $v \to 0$ in $V_k$.
\end{lemma}
\begin{proof}
Let $p \in \Pol^n(V_{\infty})$, and let $k,l \in \N$ be such that $p \in \Pol^n(V_k, V_l)$.
As in the proof of \cref{lem pol diffble}, $p' \in \Pol^{n-1}(V_k, \cB(V_k, V_l))$. So the claim follows from \cref{lem order pol}.
\end{proof}

\begin{remark}
Everything in this subsection generalises directly to polynomial and (compactly) differentiable maps between two different Fr\'echet spaces that are given as intersections of nested sequences of Banach spaces. We will not need this generalisation, however.
\end{remark}

\subsection{Comparable sequences of Banach spaces}\label{sec pol seq}

In the rest of this section, we discuss some relevant properties and examples of comparable sequences of Banach spaces (\cref{def comparable}). Particularly relevant to \cref{thm normal form} are sequences of Banach spaces comparable to sequences of separable Hilbert spaces, which we discuss in \cref{sec comp Ban Hilb}. We will see  relevant examples in \cref{sec ex Sob Ck}.

Suppose that  $\{V_k\}_{k=1}^{\infty}$ and $\{W_k\}_{k=1}^{\infty}$  are comparable nested sequences of Banach spaces. Then  $V_{\infty} = W_{\infty}$ as sets.
\begin{lemma}
The two spaces $V_{\infty}$ and $W_{\infty}$ are equal as Fre\'chet spaces.
\end{lemma}
\begin{proof}
Let $(v_j)_{j=1}^{\infty}$ be a sequence in $V_{\infty}$ such that for every $k \in \N$, $\lim_{j\to \infty} \|v_j\|_{V_k} = 0$. Let $k \in \N$, and choose $l \in \N$ such that we have a bounded inclusion $V_l \subset W_k$. Then there is a constant $C>0$ such that for every $j$,  $\|v_j\|_{W_k} \leq C \|v_j\|_{V_l}$, which goes to zero as $j \to \infty$.
\end{proof}

The following fact follows directly from the definitions, and the fact that the classes of maps in question are closed under composition with bounded linear maps.
\begin{lemma} \label{lem comparable same maps}
If $f\colon V_{\infty} \to V_{\infty}$ is a (compact) polynomial map or a (compactly) differentiable map, then it also defines a map of the same type on $W_{\infty}$.
\end{lemma}
This lemma in particular states that $\Pol^n(V_{\infty}) = \Pol^n(W_{\infty})$ as vector spaces. We will use the fact that this equality includes natural topologies on these spaces (\cref{cor PolV PolW} below) to prove \cref{cor comp conv}.

\begin{lemma}\label{lem Bn Vk Wk}
For all $l \in N$, there is an  $l' \in \N$ such that for every $k' \in \N$ 
%\ajr{Any other constraints on \(l,l',k,k'\)?}
with $k' \geq l'$, there is a $k \in \N$ such that  we have a bounded inclusion map $\cB^n(V_{k'}, V_{l'}) \subset \cB^n(W_{k}, W_{l})$
\end{lemma}
\begin{proof}
Let $l \in \N$. Choose $l' \in \N$ and $C_1>0$ such that
for every $v \in V_{\infty}$,  $\|v\|_{W_l} \leq C_1 \|v\|_{V_{l'}}$.
Let $k' \geq l'$.
Choose $k \in \N$ and $C_2>0$ such 
 that
for every $v \in V_{\infty}$,  $\|v\|_{V_{k'}} \leq C_2\|v\|_{{W_k}}$.
 Then for all $\lambda \in \cB^n(V_{k'}, V_{l'})$,
\begin{multline*}
\sup_{\|w_1\|_{W_k},  \ldots, \|w_n\|_{W_k} \leq 1} \|\lambda(w_1, \ldots, w_n)\|_{W_l} 
\leq C_1 C_2^n \sup_{\|v_1\|_{V_{k'}},  \ldots, \|v_n\|_{V_{k'}} \leq 1} \|\lambda(v_1, \ldots, v_n)\|_{V_{l'}}
\\
\leq C_2 C_2^n \|\lambda\|_{\cB^n(V_{k'}, V_{l'})}\,.
\end{multline*}
\end{proof}

For a sequence $(p_j)_{j=1}^{\infty}$ in $\Pol^n(V_{\infty})$, we define $p_j \to 0$ in 
$\Pol^n(V_{\infty})$ to mean that for every $l \in \N$, there is a $k\in \N$ such that $p_j \to 0$ in $\Pol^n(V_k, V_l)$. (This includes the requirement that $p_j \in \Pol^n(V_k, V_l)$ for every $j$.) 
%\cref{lem Bn Vk Wk} has the following immediate consequence.
\begin{corollary}\label{cor PolV PolW}
We have
$\Pol^n(V_{\infty}) = \Pol^n(W_{\infty})$, including topologies.
\end{corollary}
\begin{proof}
Let $(p_j)_{j=1}^{\infty}$ be a sequence in $\Pol^n(V_{\infty})$ converging to zero. 
Let $l \in \N$. Choose $l'$ as in \cref{lem Bn Vk Wk}.
% Choose $l' \in \N$ and $C_1>0$ such that we have a bounded inclusion $V_{l'} \subset W_k$.
 Choose $k' \in \N$ such that $p_j \to 0$ in $\Pol^n(V_{k'}, V_{l'})$. Choose $k \in \N$ as in \cref{lem Bn Vk Wk}.

For each $j$, write $p_j = p_{\lambda_j}$, for $\lambda_j \in S\cB^n(V_{k'}, V_{l'})$. By \cref{lem Bn Vk Wk}, there is a $C>0$ such that for every $j$,
\[
\|\lambda_j\|_{\cB(W_k, W_l)} \leq C \|\lambda_j\|_{\cB(V_{k'}, V_{l'})},
\] 
which goes to zero as $j\to \infty$.
Hence $p_j \to 0$ in $\Pol^n(W_{\infty})$.
\end{proof}
%\Todo: version for compact polynomials? Need that? If the right spaces have the approximation property, then this follows from the bounded poly case because compact polys are precisely convergent series of standard monomials.

\subsection{Sequences of Banach spaces comparable to sequences of Hilbert spaces} \label{sec comp Ban Hilb}

As before, we suppose that  $\{V_k\}_{k=1}^{\infty}$ and $\{W_k\}_{k=1}^{\infty}$  are comparable nested sequences of Banach spaces. Now we make the additional assumption that the spaces $W_k$ are separable Hilbert spaces. Suppose that $\{e_j\}_{j=1}^{\infty}$ is a subset of $W_{\infty}$ that is orthogonal in all spaces $W_k$, with dense span. Then taking inner products with $e_j$ defines bounded functionals, all denoted by  $e^j$, on all spaces  $W_k$, and hence in $V_k$ for $k$ large enough. 
%\begin{lemma} \label{lem cpt pol Hilbert}
%% Let $\{e^m\}_{m=1}^{\infty}$ be a  subset of $\bigcap_{k=1}^{\infty} V_{k}^*$ that is  Schauder basis of $V_k^*$ for every $k$. Then 
%% 
% For all $l \in \N$, there is a $k \in \N$ such that for every Banach space $U$ with  a Schauder basis $\{f_m\}_{m=1}^{\infty}$, and all $p \in \cK\Pol^n(V_k, U)$, there are unique $a_q^m \in \C$ such that
% \beq{eq p sum}
% p = \sum_{q \in \N_0^{\infty}; |q|\leq n} \sum_{m=1}^{\infty} a_q^m p^q \otimes f_m,
% \eeq
% where $p^q$ is defined with respect to the set  $\{e^j\}_{j=1}^{\infty}$, and the sum converges in $ \cK\Pol^n(V_l, U)$.
%\end{lemma}
%\begin{proof}
%Let $l \in \N$ be given, and 
%choose $k, k' \in \N$ such that we have bounded inclusions $V_l \subset W_{k'} \subset V_k$.  If $p \in \cK\Pol^n(V_k, U)$, then $p \in \cK\Pol^n(W_{k'}, U)$. 
%%Furthermore, $e^j \in W_{k'}^*$ for every $j$. Since $W_{k'}^*$ has the approximation property, 
%\cref{prop basis} implies that \eqref{eq p sum} holds, where the sum converges in $\cK\Pol^n(W_{k'}, U)$, and hence in $p \in \cK\Pol^n(V_l, U)$.
%\end{proof}

\begin{corollary} \label{cor Taylor cpt Vinfty}
%Let $\{V_k\}_{k=1}^{\infty}$ be a nested sequence of Banach spaces, comparable to a nested sequence of separable Hilbert spaces. Let $(f_r)_{r=1}^{\infty}$ be a linearly independent subset of $V_{\infty}$ with dense span in $V_k$ for every $k$. Let $(e^s)_{s=1}^{\infty}$ be a linearly independent subset of $\bigcap_{k=1}^{\infty}V_k^*$ with dense span in $V_k^*$ for every $k$.
%
Suppose $f$ is an $n+1$ times differentiable map from $V_{\infty}$ to itself, and that $f$ is $m$ times compactly differentiable for every $m \leq n$. 
%Suppose that $\|f^{(n+1)}(\xi)\|_{V_l} \leq M_$ for every $\xi$ in a closed ball around $u$ contained in $U$. 
%Suppose that $f = \cO(n)$.
%
% for every $l \in \N$, there s a $k \in \N$ such that $\|f(u+h)\|_{V_l} = O(\|h\|_{V_k}^{n})$ as $h \to 0$ in $V$. 
Then  there are unique complex numbers $a^k_q$ such that
\beq{eq Taylor cpt}
f(u+h) = %\sum_{q, k}
\sum_{q \in \N_0^{\infty};\, |q|\leq n} 
\sum_{j=1}^{\infty}
a^j_q  h^q e_j + \rho(h),
\eeq
where $\rho\colon V_{\infty} \to V_{\infty}$ is of order $n+1$, and the part of the sum where $|q|=m$ converges as a function of $h$ in  $\Pol^m(V_{\infty})$, for $m = 0, \ldots, n$.
\end{corollary}
\begin{proof}
Let $l \in \N$. Choose $k, k', l' \in \N$ be such that we have bounded inclusions $V_{l'} \subset W_l$ and $W_k \subset V_{k'}$, and $f\colon V_{k'} \to V_{l'}$ is $n+1$ times differentiable, and $m$ times compactly differentiable for every $m \leq n$.  Then the same is true for $f\colon W_{k} \to W_{l}$. So \cref{cor Taylor cpt} implies that  \eqref{eq Taylor cpt} holds, for unique $a_q^j$, where the sum converges in $\Pol(W_{\infty})$, and hence in $\Pol(V_{\infty})$ by \cref{cor PolV PolW}, and $\rho$ is of order $n+1$ as a map from $W_{\infty}$ to itself, and hence as a map  from $V_{\infty}$ to itself.
%
%%%% Old %%%
%
%Let $l \in \N$. Choose $k \in \N$ as in \cref{lem cpt pol Hilbert}. Choose $k' \in \N$ such that $f\colon V_{k'} \to V_k$ is compactly differentiable. Then $f^{(n)}(u) \in \cK\Pol^n(V_{k'}, V_k)$. Hence the claim follows from \cref{cor Taylor cpt} and \cref{lem order pol,lem order pol lower,lem cpt pol Hilbert}.
\end{proof}

\begin{remark}
A useful feature of \cref{cor Taylor cpt Vinfty} is that it is not assumed that the spaces $V_k$ have the approximation property. The point is that the separable Hilbert spaces $W_k$ do have this property.
\end{remark}

The following corollary is an important way in which we use  comparable sequences of Banach spaces. It is used in the proof of \cref{lem conv hat psi hat F}.
%\ajr{Would it be useful notation throughout to define (in \S1.5) \(\N_k:=\{j\in\Z:j\geq k\}\)? and/or (including the case \(n=\infty\)) \(\N_k^{n,p}:=\{q\in \N_k^n:|q|=p\}\)?}
\begin{corollary}\label{cor comp conv}
%Let $\{V_k\}_{k=1}^{\infty}$ be a nested sequence of Banach spaces, comparable to a nested sequence $\{W_k\}_{k=1}^{\infty}$ of separable Hilbert spaces. Suppose that $\{e_j\}_{j=1}^{\infty}$ is a subset of $W_{\infty}$ that is orthogonal in all spaces $W_k$, with dense span. Suppose that taking inner products with $e_j$ defines bounded functionals, both denoted by $e^j$, on $V_1$ and $W_1$ (and hence on all spaces $V_k$ and $W_k$).
Let $a_q^j \in \C$ be given such that
\beq{eq sum all}
\sum_{q \in \N_0^{\infty};\, |q|=n} \sum_{j=1}^{\infty} a_q^j p^q \otimes e_j
\eeq
converges in $\Pol^n(V_{\infty})$. Then for all subsets $A \subset \{q \in \N_0^{\infty} : |q|=n\} \times \N$, the series
\beq{eq sum A}
\sum_{(q,j) \in A} a_q^j p^q \otimes e_j
\eeq
converges in $\Pol^n(V_{\infty})$.
\end{corollary}
\begin{proof}
By \cref{cor PolV PolW}, the series \eqref{eq sum all} converges in  $\Pol^n(W_{\infty})$, and it is enough to show that \eqref{eq sum A} converges in  $\Pol^n(W_{\infty})$. And that follows from \cref{lem uncond KPn}.
\end{proof}
\begin{remark}
In \cref{cor comp conv}, the two series converge to elements of $\cK\Pol^n(V_{\infty})$.
\end{remark}

%\begin{corollary}
%We have
%$\Pol^n(V_{\infty}) = \Pol^n(W_{\infty})$, including topologies.
%\end{corollary}

\subsection{Example: Sobolev spaces and $C^k$-spaces} \label{sec ex Sob Ck}

Let $\Omega$ be a bounded open subset of $\R^d$ or of an $d$-dimensional Riemannian manifold, and suppose that the boundary of $\Omega$ is $C^1$. 
Set
\begin{equation*}
V_k := W^{k-1,2}(\Omega) \quad\text{and}\quad
W_k := W^{k-1, k+1}(\Omega).
\end{equation*}

\begin{lemma} \label{lem Sobolev comparable}
The  above sequences  $\{V_k\}_{k=1}^{\infty}$ and $\{W_k\}_{k=1}^{\infty}$ of Banach spaces are comparable.
\end{lemma}
\begin{proof}
Set $r:= \lceil d/2\rceil$.
By a Sobolev embedding theorem, we have bounded inclusions
\begin{equation*}
W^{l+ \frac{d}{p}(1-p/2) , p}(\Omega) \subset W^{l,2}(\Omega)
\quad\text{and}\quad
W^{l+r,2}(\Omega) \subset W^{l+r-\frac{d}{2}(1-2/q), q}(\Omega),
\end{equation*}
for all $1\leq p<2<q<\infty$ and every $l \in \N_0$. Now for all such $p,q$ and $l$,  
\begin{equation*}
l +\frac{d}{p}(1-p/2) <  l+r
\quad\text{and}\quad
 l+r-\frac{d}{2}(1-2/q) > l \geq 1.
\end{equation*}So we have bounded inclusions
\[
W^{ l+r, p}(\Omega)  \subset W^{l,2}(\Omega)
\quad\text{and}\quad
W^{l+r,2}(\Omega) \subset W^{l, q}(\Omega).
\]
Furthermore, since $\Omega$ has finite volume, we have bounded inclusions $W^{l, p'}(\Omega) \subset W^{ l, p''}(\Omega)$ for every $l$ and all $p'\geq p''$.

The above arguments imply that we have bounded inclusions
\[
W_{k+r} \subset V_k \subset W_1
\quad\text{and}\quad
V_{k+r}\subset W_k \subset V_k.
\]
\end{proof}
\begin{remark}
In this example, the spaces $V_k$ are separable Hilbert spaces. 
\end{remark}

For another example of comparable sequences of Banach spaces,  fix $p>n$.
For $k \in \N$, set 
\[
V_k := W^{k, p}(\Omega)
\quad\text{and}\quad
W_k := C^{k}(\overline{\Omega}).
\]
The space $C^k(\overline{\Omega})$ is complete in the norm given by the maximum of the sup-norms of the partial derivatives of functions up to order $k$.
\begin{lemma}
The sequences  $\{V_k\}_{k=1}^{\infty}$ and $\{W_k\}_{k=1}^{\infty}$ of Banach spaces are comparable.
\end{lemma}
\begin{proof}
We  have a bounded inclusion
$W_k \subset V_{k}$ for every $k$.
So it remains to show that for every $k \in \N$, there are $l_1, l_2 \in \N$ such that we have bounded inclusions 
\beq{eq Sob Ck incl}
\begin{split}
V_{l_1} &\subset W_k;\\
 V_k &\subset W_{l_2}.
\end{split}
\eeq

By a Sobolev embedding theorem, we have a bounded inclusion
\[
W^{k,p}(\Omega) \subset C^l(\overline{\Omega})
\]
for all $k,l \in \N$ such that
\[
l+\frac{n}{p}<k \leq l+1+\frac{n}{p}.
\]
For $k \in \N$, set $l_1 := k+1+\lceil n/p\rceil $ and $l_2 := \max\{k-1-\lfloor n/p \rfloor, 1\}$. Then this Sobolev embedding theorem yields the desired inclusions \eqref{eq Sob Ck incl}.
\end{proof}
\begin{remark}
If $n=1$, then we may take $p=2$, so that the spaces $V_k$ are Hilbert spaces. 
\end{remark}

%\begin{remark}
%All function spaces in this subsection may be taken to be spaces of either scalar- or vector-valued functions.
%\end{remark}

\cref{prop ex cpt poly Sobolev} and \cref{lem Sobolev comparable} together imply \cref{thm special case}. The extension of \cref{prop ex cpt poly Sobolev} to coeficients $a_q$ depending on a real (time) parameter $t$ 
%\ajr{Be more informative and call \(t\) the ``time'', rather than uninformative ``parameter''.  or am I mistaken?}
in a smooth way, and the extension of \cref{lem Sobolev comparable} to vector-valued functions, are straightforward.

\section{A coordinate transform}
\label{sec co xform}

\subsection{A residual}

%Let $\{V_k\}_{k=1}^{\infty}$ be a compactly nested sequence of Banach spaces (see \cref{def nested seq}), such that $V_1$ is a Hilbert space. Let $A \in \cB(V_{\infty})$. Let $I \subset \R$ be an open interval, and let 
%\[
%f\colon I \times V_{\infty} \to V_{\infty}
%\]
%be infinitely compactly differentiable with respect to $V_{\infty}$ and differentiable with respect to $I$. Suppose that $f = \cO(2)$, and that for every $l \in \N$, there is a $k \in \N$ such that $f\colon V_k \to V_l$ is differentiable, and  
%\beq{eq est der f}
%\|f'_{V_{\infty}}(t,v)\|_{\cB(V_k, V_l)} = O(\|v\|_{V_k}), 
%\eeq
%uniformly in $t$ in compact subsets of $I$.
%
%Our goal is to study smooth maps $x\colon I \to V_1$ satisfying the evolution differential equation
%\beq{eq ODE}
%\dot x(t) = Ax(t) + f(t, x(t))
%\eeq
%for all $t \in I$.

Recall the setting of \cref{sec setup}.
In this section and the next, based upon the details of some given dynamical system~\eqref{eq ODE} we construct both a coordinate transformation~\eqref{eq def xyz p} and a corresponding `normal form' system~\eqref{eq ODE XYZ p}, such that solutions $X$ to \eqref{eq ODE XYZ p}, transformed by~\eqref{eq def xyz p}, satisfy the given dynamical system~\eqref{eq ODE} up to  residuals of a specified order~$p$. See \cref{thm normal form}. We do this inductively, by showing how to construct such a transformed system to satisfy~\eqref{eq ODE} with residual of order $p+1$ from a version with residual of order~$p$.

In \cref{sec transf choice}, we construct a more specific choice of the general coordinate transform constructed in this section, in order to establish exact invariant manifolds, and study their properties, for constructed systems arbitrarily close to the given system~\eqref{eq ODE}.

%Suppose that $\{e_j\}_{j=1}^{\infty} \subset V_{\infty}$ is a sequence of eigenvectors of $A$ which is a Hilbert basis of $V_1$.
%Suppose that the sequence $\{V_k\}_{k=1}^{\infty}$ is comparable to a nested sequence  $\{W_k\}_{k=1}^{\infty}$ of separable Hilbert spaces  in which the vectors $e_j$ are orthogonal.

\begin{remark} \label{rem comparable unnecessary}
In  \cref{sec setup}, we assumed that the sequence $\{V_k\}_{k=1}^{\infty}$ is comparable to a nested sequence  $\{W_k\}_{k=1}^{\infty}$ of separable Hilbert spaces  in which the vectors $e_j$ are orthogonal.
\cref{lem comparable same maps} implies that we may equivalently assume that $\{V_k\}_{k=1}^{\infty}$  itself is a nested sequence of separable Hilbert spaces, because all maps from $V_{\infty}$ to itself we use transfer to maps from $W_{\infty}$ to itself of the same type (e.g. compact polynomial and compactly differentiable maps).
%\ajr{Clarify the previous sentence?}
 However, the formulation where $\{V_k\}_{k=1}^{\infty}$ is only comparable to a nested sequence  of separable Hilbert spaces makes it clearer that we have the flexibility to consider maps between Banach spaces. This is natural for example in the context of \cref{prop ex cpt poly Sobolev}.
\end{remark}

Let $p \in \N$, with $p \geq 2$. Let $\xi_p, F_p\colon I \times V_{\infty}\to V_{\infty}$ be such that $\xi_p - \id$ and $F_p$ are compact polynomial maps of order at most $p-1$ in the $V_{\infty}$ component, and infinitely differentiable in $I$. Suppose, furthermore, that $\xi_p$ is a near-identity at zero, and that $F_p = \cO(2)$.

Recall that our goal is to relate the dynamics of maps $x$ satisfying \eqref{eq ODE} to the dynamics of maps $X\colon I \to V_{\infty}$ satisfying \eqref{eq ODE XYZ p} when $x$ and $X$ are related by the coordinate transform $\xi_p$ as in
\eqref{eq def xyz p}.

%\beq{eq ODE XYZ p}
%\dot X(t) = AX(t) + F_p(t, X(t)),
%%\dot Y(t) &= BY(t)+ G_p(t,X(t), Y(t), Z(t) );\\
%%\dot Z(t) &=  CZ(t)+ H_p(t,X(t), Y(t), Z(t) ),\\
%%\end{split}
%\eeq
%when $x$ and $X$ are related by the coordinate transform
%\beq{eq def xyz p}
%x(t) = \xi_p(t, X(t)).
%\eeq

%Consider the space
%\[
%W := \{(t,v) \in I \times V_0; Av+F_p(t,v) \in V_0\}.
%\]
For maps $f,g\colon I \times V_{\infty} \to V_{\infty}$, with $f$ differentiable, we write $f'_{V_{\infty}} \circ g$ for the map from $I \times V_{\infty}$ to $V_{\infty}$ given by
\[
(f'_{V_{\infty}} \circ g)(t,v) = f'_{V_{\infty}}(t,v)(g(t,v)),
\]
for all $t \in I$ and $v \in V_{\infty}$. (Note that this is different from $(f\circ g)'_{V_{\infty}}(t,v) = f'_{V_{\infty}}(t, g(t,v))(g(t,v))$.) If $g$ is a map from $V_{\infty}$ to $V_{\infty}$ to itself, then the composition $f'_{V_{\infty}} \circ g$ is defined analogously. Also recall the notation for compositions of maps to and from $I \times V_{\infty}$ and $V_{\infty}$ under \emph{Notation and conventions} in \cref{secNota}.

Define the maps
$
\Phi_p, R_p\colon I \times V_{\infty} \to V_{\infty}
$
by
\[
\Phi_p:= (\xi_p)'_I + (\xi_p)'_{V_{\infty}} \circ(A + F_p)
\quad\text{and}\quad
R_p := -A\circ \xi_p-f \circ \xi_p+\Phi_p.
\]
The map $R_p$ is the \emph{residual} of the transformed \ODE, in the following sense.
\begin{lemma} \label{lem aqrs ODE}
For all smooth maps $X\colon I \to V_{\infty}$  satisfying \eqref{eq ODE XYZ p}, and with $x\colon I \to V_{\infty}$ determined from~\(X\) by~\eqref{eq def xyz p}, 
\beq{eq ODE xyz}
\dot x(t) = Ax(t) + f(t, x(t)) + R_p(t, X(t)).
\eeq
\end{lemma}
\begin{proof}
For $X$ and $x$ as in the lemma, the chain rule (\cref{lem chain rule}) and \eqref{eq ODE XYZ p} imply that for all $t \in I$,
\[
\dot x(t) = \Phi_p(t, X(t)) = Ax(t) + f(t, x(t)) + R_p(t, X(t)).
\]
\end{proof}

\begin{lemma} \label{lem Rp diffble}
The maps $\Phi_p$ and $R_p$ are infinitely compactly differentiable.
\end{lemma}
\begin{proof}
Because the Banach spaces $V_k$ are compactly nested, it is enough to show that  $\Phi_p$ and $R_p$ are infinitely  differentiable. And that is true by the chain rule, because $f$ is infinitely differentiable, and so are $F_p$ and $\xi_p$, by \cref{lem pol diffble}.
\end{proof}

To recursively construct~\eqref{eq ODE XYZ p} and~\eqref{eq def xyz p}, suppose that $R_p = \cO(p)$. We proceed to show that~\eqref{eq ODE XYZ p} and~\eqref{eq def xyz p} may be refined to make the new residual of~\(\cO(p+1)\). % (see \cref{def On}). 
By \cref{lem Rp diffble}, \cref{cor Taylor cpt Vinfty} (where $u=0$ and $h=v$), and \cref{lem order pol,lem order pol lower}, there are unique, infinitely differentiable maps
\[
a^{q}\colon  I \to V_{\infty}
\]
for all multi-indices $q \in \N_0^{\infty}$ with $|q| = p$, and a map
\[
\rho_p\colon I \times V_{\infty} \to V_{\infty}
\]
such that for all $t \in I$, and $v \in V_{\infty}$, 
\beq{eq aqrs}
R_p(t, v) = -  \sum_{q \in \N_0^{\infty};\, |q|= p} a^{q}(t)v^q + \rho_p(t,v),
\eeq
where the sum converges in $\Pol^p(V_{\infty})$, differentiably in $t$, and
\mbox{$\rho_p = \cO(p+1)$}. 
 The monomial $v^q$ is defined as in \eqref{eq def vq}, with respect to the set of functionals $e^j = (e_j, \relbar)_{V_1}$, for $j \in \N$.

In \cref{sec coord transf}, we construct  maps $\xi_{p+1}, F_{p+1}$ such that the order $p$ term $R_p(t, X(t))$ in \eqref{eq ODE xyz} may be replaced by an order $p+1$ term $R_{p+1}(t, X(t))$, if \eqref{eq ODE XYZ p} and \eqref{eq def xyz p} hold with $p$ replaced by $p+1$. 

\subsection{Construction of the coordinate transform} \label{sec coord transf}

%This lemma will be proved in \cref{sec pf lem aqrs}.

For each $j \in \N$, let $\alpha_j$ be the eigenvalue of $A$ corresponding to $e_j$.
For all $q \in \N_0^{\infty}$ with $|q| = p$, define $\mu^q \in \C$ by %(and this sum has at most $p$ nonzero terms)
%\ajr{Clarify what is \(\mu\) in this equation?}
\beq{eq def nuqrs}
\mu^{q} = \sum_{j=1}^{\infty} q_j \alpha_j  \quad \in \C.
\eeq
(This sum has at most $p$ nonzero terms.)
For such $q \in \N_0^{\infty}$,  let $a^{q}$ be as in \eqref{eq aqrs}. %\cref{lem aqrs}. 
Let
$
\hat \xi^{q}, \hat F^{q} \colon I \to V_{\infty}
$
be smooth maps such that
\beq{eq DE iteration}
\hat F^{q} + (\hat \xi^{q})'  + \mu^{q}\hat \xi^{q} - A\hat \xi^{q} = a^{q}.
\eeq %suppose that this restriced
Suppose that
 the sums
%\beq{eq sums Fq psiq}
\[
\hat F(t,v) =
\sum_{q \in \N_0^{\infty};\, |q| = p} \hat F^{q}(t) v^q
\quad\text{and}\quad
\hat \xi(t,v) =
\sum_{q \in \N_0^{\infty};\, |q| = p} \hat \xi_{q}(t)v^q
\]
%\eeq
converge in $\Pol^p(V_{\infty})$, differentiably in $t$. 
%
%Suppose that these maps restrict to polynomials in $\Pol^p(V_{2^{p+1}}, V_1)$.
%Suppose also that 
%\beq{eq hat F hat psi degs}
%\begin{split}
%\hat F(I \times V_{2^{p+1}}) &\subset V_{\infty};\\
%\hat \xi(I \times V_{2^{p+1}}) &\subset V_{\infty}. 
%\end{split}
%\eeq

 Define a new coordinate transform map $\xi_{p+1}\colon I \times V_{\infty} \to V_{\infty}$ and corresponding map $F_{p+1}\colon I \times V_{\infty} \to V_{\infty}$ that replaces $F_p$ in \eqref{eq ODE XYZ p}, by
\beq{eq def psip+1 Fp+1}
\xi_{p+1}=  \xi_{p}+ \hat \xi
\quad\text{and}\quad
F_{p+1} =  F_{p}+ \hat F.
\eeq
%for $t \in I$ and  $v \in V_0$.
%(\Todo: need this to be smooth in $t$.) Convergence of the second sum in \eqref{eq sums Fq psiq} implies that $\xi_{p+1}(I \times V_0) \subset V_0$, as required.

The following result is the main step in the construction of the coordinate transform we are looking for. 
\begin{proposition} \label{prop coord transf}
If $X$ and $x$ are as in \eqref{eq ODE XYZ p}  and \eqref{eq def xyz p}, with $p$ replaced by $p+1$, 
%and we set
%\[
%(x(t), y(t), z(t)) = \xi_{p+1}(t, X(t), Y(t), Z(t)),
%\]
then \eqref{eq ODE xyz} holds, with the residual
 $R_p$ replaced by a residual $R_{p+1}$ 
satisfying
\[
R_{p+1}= \cO(p+1).
\]
The maps $\xi_{p+1} - \id$ and $F_{p+1}$ are  compact polynomials 
in $\cK \Pol(V_{\infty})$
of order at most $p$, and $\xi_{p+1}$ is a near-identity.
\end{proposition}

\begin{remark} \label{rem components ODE}
The maps $\hat \xi^q$ and $\hat F^q$ can be found explicitly if we decompose~\eqref{eq DE iteration} with respect to the basis $\{e_j\}_{j=1}^{\infty}$. This will be done in  \cref{sec transf choice}.
%
%Taking inner products with $e_j$, we get
%\[
%(\hat F^{q}, e_j)_V + (\hat \xi^{q}, e_j)_V'  + (\mu^{q}- \alpha_j)(\hat \xi^{q}, e_j)_V  = (a^{q}, e_j)_V.
%\]
%Solutions to this underdetermined {\ODE} can be combined into maps $\hat \xi^q$ and $\hat F^q$. 
%
One solution to \eqref{eq DE iteration} is $\hat F^q = a^q$ and $\hat \xi^q = 0$. However, for our purposes, we need the function $F^q$ to be of a specific form. 
The main purpose of this work is to find $\hat F^q$ such that the $e_j$-component of $F^{q}$ is zero for certain combinations of $q$ and~\(j\), in such a way that an exact separation of stable, centre and unstable modes is maintained. See \cref{prop csu}.
\end{remark}

%\begin{remark}
%By construction, the map $\xi_{p+1}$ is polynomial of degree (at most) $p$, and a near-identity. The map $F_{p+1}$ is polynomial of degree at most $p$ as well.
%\end{remark}

%
%\subsection{Proof of \cref{lem aqrs}} \label{sec pf lem aqrs}
%
%Let $X, Y, Z$ be as in \eqref{eq ODE XYZ p}, and define $x(t)$, $y(t)$ and $z(t)$ by \eqref{eq def xyz p}. Applying the chain rule (\cref{lem chain rule}) and \eqref{eq ODE XYZ p}, we find that
%\begin{multline*}
%\dot x(t) = (\xi_p^{U_0})_I(t, X(t), Y(t), Z(t)) + (\xi_p^{U_0})_{U_0}(t, X(t), Y(t), Z(t))\dot X(t) \\
% + (\xi_p^{U_0})_{V_0}(t, X(t), Y(t), Z(t))\dot Y(t) + (\xi_p^{U_0})_{W_0}(t, X(t), Y(t), Z(t))\dot Z(t) \\
% = (\xi_p^{U_0})_I(t, X(t), Y(t), Z(t)) +  (\xi_p^{U_0})_{U_0}(t, X(t), Y(t), Z(t)) \bigl(Ax(t) + F_p(t, X(t), Y(t), Z(t)) \bigr) \\
% + (\xi_p^{U_0})_{V_0}(t, X(t), Y(t), Z(t)) \bigl((B+ G_p(t,X(t), Y(t), Z(t) ))Y(t) \bigr) \\
% + (\xi_p^{U_0})_{W_0}(t, X(t), Y(t), Z(t)) (C+ H_p(t,X(t), Y(t), Z(t) ))Z(t).
%\end{multline*}
%For later use, we will abbreviate  the right-hand side to $\Phi_p(t)$.
%
%By the first line of \eqref{eq ODE xyz}, this equals
%\[
%A\psi^{U_0}_p(t,X(t), Y(t), Z(t)) + f(t, \xi_p(t,X(t), Y(t), Z(t))) + \varphi^U(\xi_p(t,X(t), Y(t), Z(t))).
%\]
%So the left-hand side of \eqref{eq aqrs} equals
%\[
%-\varphi^U(\xi_p(t,X(t), Y(t), Z(t))). 
%\]
%Because $\varphi^U = \cO(p)$, and $\xi_p(t, \relbar, \relbar, \relbar)$ is a near-identity, we have
%\[
%\varphi^U(\xi_p(t,\relbar, \relbar, \relbar))  = \cO(p+1).
%\]

\subsection{Proof of \cref{prop coord transf}} \label{sec pf prop coord transf}

Define the map $\hat \Phi\colon I \times V_{\infty} \to V_{\infty}$ by
\beq{eq def hat Phi}
\hat \Phi = \hat \xi'_I+ \hat \xi'_{V_{\infty}}\circ (A + F_p) 
+ (\xi_p)'_{V_{\infty}}\circ \hat F +
 \hat \xi'_{V_{\infty}}\circ \hat F.
\eeq

\begin{lemma} \label{lem hat Phi p}
We have
\beq{eq diff hat Phi}
\hat \Phi - (\hat \xi'_I + \hat \xi'_{V_{\infty}} \circ A + \hat F) = \cO(p+1).
\eeq
%%
%\| \hat \Phi_p(t, v) - \bigl(\hat \xi'_I(t, v) + \hat \xi'_{V_{\infty}}(t, v)  Av 
%+ \hat F_p(t, v) \bigr) \|_{V_{0}}%+  \hat \xi_{V_0}(t, v)Bv +  \hat \xi_{W_0}(t, v)C w \\
%= O(\|v\|_{2^{p+1}}^{p+1}).
%\]
%uniformly in $t$ on compact subsets of $I$.
\end{lemma}
\begin{proof}
The left-hand side of \eqref{eq diff hat Phi} equals
\[
\hat \xi'_{V_{\infty}} \circ F_p + \hat \xi'_{V_{\infty}} \circ \hat F + ((\xi_p)'_{V_{\infty}} -\id)\circ \hat F.
\]
By \cref{lem der poly On-1}, the derivative  $\hat \xi'_{V_{\infty}}$ satisfies the condition of \cref{lem der Om+n}, with $n=p-1$. Since $F_p = \cO(2)$, \cref{lem der Om+n} implies that $\hat \xi'_{V_{\infty}} \circ F_p = \cO(p+1)$. Similarly, $\hat \xi'_{V_{\infty}} \circ \hat F = \cO(2p-1)$.
Now $\xi_p$ is a polynomial map, and a near-identity. By \cref{lem der poly On-1}, this implies that $(\xi_p)'_{V_{\infty}} -\id$ satisfies the condition of \cref{lem der Om+n}, with $n=1$. So \cref{lem der Om+n} implies that $((\xi_p)'_{V_{\infty}} -\id)\circ \hat F = \cO(p+1)$.
\end{proof}

\begin{lemma} \label{lem hat phi aqrs}
For $q \in \N_0^{\infty}$ such that $|q| = p$, let $a^{q}$ be as in \eqref{eq aqrs}. % \cref{lem aqrs}. 
Then for all $t \in I$ and $v \in V_{\infty}$,
\beq {eq hat phi aqrs}
 \hat \xi_I'(t, v) + \hat \xi'_{V_{\infty}}(t, v)  Au \\
+ \hat F(t, v) %+  \hat \xi_{V_0}(t, v)Bv +  \hat \xi_{W_0}(t, v)C w \\
= A\hat \xi(t, v) + 
 \sum_{q \in \N_0^{\infty};\, |q| = p} a^{q}(t)v^q.
\eeq
\end{lemma}
\begin{proof}
First of all,
$
 \hat \xi'_I(t, v) =  \sum_{q \in \N_0^{\infty};\, |q|= p} (\hat \xi^{q})'(t)v^q.
$
By \cref{lem der poly}, 
\[
\hat \xi'_{V_{\infty}}(t, v)  Av =\\
 \sum_{q \in \N_0^{\infty};\, |q|= p} \hat \xi^{q}(t) \sum_{j=1}^{\infty} q_j(e_j, v)_{V_1}^{q_j-1} (e_j, Av)_{V_1}
 \prod_{j'\not= j} (e_{j'}, v)_{V_1}^{q_{j'}}.
\]
Now, because the vectors $\{e_j\}_{j=1}^{\infty}$ are orthogonal with respect to $(\relbar, \relbar)_{V_1}$, we have $
(e_j, Av)_{V} = \alpha_j (e_j, v)_{V_1}
$ for every $j$. 
So
\[
\hat \xi'_{V_{\infty}}(t, v)  Av =
 \sum_{q\in \N_0^{\infty};\, |q| = p} \hat \xi^{q}(t) \Bigl(\sum_{j=1}^{\infty} q_j \alpha_j \Bigr)v^q.
\]
We find that the left-hand side of \eqref{eq hat phi aqrs} equals
\[
 \sum_{q \in \N_0^{\infty};\, |q|= p} \Bigl(
 (\hat \xi^{q})'(t) + \hat F^{q}(t) + \mu^{q} \hat \xi^{q}(t)
  \Bigr)v^q,
\]
with $\mu^{q}$ as in \eqref{eq def nuqrs}. So the claim follows from \eqref{eq DE iteration}.
\end{proof}

Define the maps
$
\Phi_{p+1}, R_{p+1}\colon I \times V_{\infty} \to V_{\infty}
$
by
\[
\Phi_{p+1} := \Phi_p + \hat \Phi
\quad\text{and}\quad
R_{p+1} := -A\circ \xi_{p+1}-f \circ \xi_{p+1}+\Phi_{p+1}.
\]
\begin{lemma} \label{lem Rp+1}
%The map $R_{p+1}$ is $p+1$ times compactly differentiable as a map from $I \times V_{2^{p+1}}$ to $V_0$ and as a map from $I \times V_{2^{p+2}}$ to $V_1$, and f
The residual $R_{p+1}$ satisfies 
$
R_{p+1} = \cO(p+1).
$
\end{lemma}
\begin{proof}
By \cref{lem hat Phi p,lem hat phi aqrs}, 
\beq{eq Rp+1 Op+1}
\begin{split}
R_{p+1} &= -A \circ \xi_p- f\circ\xi_{p+1} + \Phi_p - \tilde R_p+ \cO(p+1)\\
&= R_p - f \circ \xi_{p+1} + f\circ \xi_p - \tilde R_p+ \cO(p+1),
\end{split}
\eeq
where, for $t \in I$ and $v \in V_{\infty}$,
\[
\tilde R_p(t,v) :=   -\sum_{q \in \N_0^{\infty};\, |q| = p} a^{q}(t)v^q. 
\]
By \eqref{eq aqrs}, % and the definition of $R_p$, %\cref{lem aqrs}, 
the last expression in \eqref{eq Rp+1 Op+1} equals
\[
f \circ\xi_{p+1} - f \circ \xi_{p}  + \cO(p+1).
\]

Let $l \in \N$ be given, and choose $k, k' \in \N$ such that $f\colon I \times V_{k'} \to V_l$ is differentiable, and $\xi_p, \xi_{p+1} \in \Pol(V_{k}, V_{k'})$. 
Using \cref{thm Taylor}, we write
\[
\bigl\| f(t,\xi_{p+1}(t,v)) - f(t,\xi_{p}(t,v))  - f'(t, \xi_p(t,v)) \hat \xi(t,v)\bigr\|_{V_l} = O(\|\hat \xi(t,v)\|^2),
\]
for $t \in I$ and $v \in V_{k'}$. \cref{lem order pol} implies that
$
\|\hat \xi(t,v)\|_{V_{k'}} = O(\|v\|_{V_k})
$
uniformly in $t$ in compact sets. So
\[
\bigl\| f(t,\xi_{p+1}(t,v)) - f(t,\xi_{p}(t,v))\bigr\|_{V_l} 
= \bigl\|f'(t, \xi_p(t,v)) \hat \xi(t,v)\bigr\|_{V_l} + O(\|v\|_{V_k}^{2p}),
\]
uniformly in $t$ in compact sets.
%
%%% Old: %%%
%
%Then for all $t \in I$ and $v \in V_k$,
%\[
%\| f(t,\xi_{p+1}(t,v)) - f(t,\xi_{p}(t,v))  \|_{V_{l}} = \|f'(t, \xi_p(t,v)) \hat \xi(t,v)\|_{V_l} + o(\|\hat \xi(t,v)\|_{V_{k'}}).
%\]
%Since $\hat \xi$ is a polynomial, the term $o(\|\hat \xi(t,v)\|_{V_{k'}})$ is $O(\|\hat \xi(t,v)\|_{V_{k'}}^2)$, which is $O(\|v\|_{V_k}^{p+1})$, uniformly in $t$ in compact sets, by \cref{lem order pol}. 
%
And then, as in the proof of \cref{lem der Om+n}, the assumption \eqref{eq est der f} and \cref{lem der Om+n} imply that $\|f'(t, \xi_p(t,v)) \hat \xi(t,v)\|_{V_l}  = O(\|v\|_{k'}^{p+1})$, uniformly in $t$ in compact sets.
\end{proof}

\begin{proof}[Proof of \cref{prop coord transf}]
The correction terms $\hat \xi$ and $\hat F$ lie in $\cK\Pol^p(V_{\infty})$. Hence $\xi_{p+1}-\id$ and $F_{p+1}$ are compact polynomials, because $\xi_{p}-\id$ and $F_{p}$ are. By \cref{lem order pol}, this also implies that $\xi_{p+1}$ is a near-identity because $\xi_p$ is.
The desired property of $R_{p+1}$ is \cref{lem Rp+1}.
\end{proof}

\section{Centre, stable and unstable coordinates} \label{sec transf choice}

There is considerable flexibility in choosing the maps $\hat \xi^{q}$ and $\hat F^{q}$ in \cref{sec coord transf}. In this section, we discuss how to make specific choices, in terms of the eigenvalues of $A$, so that the normal form \eqref{eq ODE XYZ p} is useful for detecting invariant manifolds.

\subsection{Centre, stable and unstable components of $\hat F^q$} \label{sec comp Fq}

We use the notation from \cref{sec result}. In particular, let $\alpha$, $\beta$, $\gamma$ and $\tilde \mu$ be spectral gap parameters defined there.
Recall the definition of polynomial growth in \cref{def pol growth}.
\begin{proposition} \label{prop csu}
Suppose that $\beta - (p+1)\alpha > \tilde \mu$ and $\gamma - (p+1)\alpha > \tilde \mu$.
Suppose that $R_p$ has polynomial growth.
%
% for all $k \in \N$, there is a $C_k>0$ such that for all
%$q \in \N_0^{\infty}$ such that $|q|=p$ and all $t \in I$,
%\[
%\|a^q(t)\|_{V_k} \leq C_k e^{\tilde \mu |t|}.
%\]
%
%. Suppose that $a^q(t) = \cO(e^{\tilde \mu|t|})$ as $t \to \infty$ and $t \to -\infty$. 
The maps $\hat \xi^q$ and $\hat F^q$ in \cref{sec coord transf} can be chosen such that
\begin{itemize}
\item if either $q^s = 0$ and $q^u \not=0$, or  $q^u = 0$ and $q^s \not=0$, then $\hat F^q_c = 0$;
\item if $q^s = 0$, then $\hat F^q_s = 0$; 
and
\item if $q^u = 0$, then  $\hat F^q_u = 0$.
\end{itemize}
\end{proposition}
\cref{prop csu} is proved in \cref{sec update}, after some preparation done in this subsection.
%\begin{remark}
%The assumption on $\alpha, \beta, \gamma, \tilde \mu$ and $p$ in \cref{prop csu} implies that the order $p$ to which we can reduce the residual $R_p$ using the method described, depends on the spectral gap parameters $\alpha, \beta, \gamma$ and  $\tilde \mu$.
%\end{remark}

\begin{definition}\label{def conv emu}
Let $\mu \in \C$, such that $|\Real(\mu)| > \tilde \mu$. Set $a:= \inf I$ and $b:= \sup I$.
Let $u$ be a continuous %(\Todo: we can probably assume something weaker than continuity.)
 function on $\R$  such that  $u(t) = O(e^{\tilde \mu|t|})$ as $t \to \infty$ if $b = \infty$ and as $t \to -\infty$ if $a = -\infty$. 
Then we define the  function $e^{\mu (\cdot)} \star u$ on~$I$ by
%\ajr{Surely needs to be \(e^{\mu t}\), not \(e^\mu\).  Instead of~*, I usually use \texttt{$\backslash$star},~\(\star\)}
\[
(e^{\mu (\cdot)} \star u)(t) := \begin{cases}
\int_{a}^t e^{\mu(t-\tau)}u(\tau)\, d\tau & \text{if $\Real(\mu) < -\tilde \mu$};\\
\int_{t}^{b} e^{\mu(t-\tau)}u(\tau)\, d\tau & \text{if $\Real(\mu) > \tilde \mu$}.
\end{cases}
\]
\end{definition}
The integrals occurring in this definition are $\tilde \mu$-regular, in the sense defined in \cref{sec result}.

\begin{lemma} \label{lem conv}
In the setting of \cref{def conv emu}, 
%For all continuous functions $u$ on $\R$  such that  $u(t) = O(e^{\tilde \mu|t|})$ as $t \to \infty$ and $t \to -\infty$, we have
\[
(e^{\mu (\cdot)} \star u)' = \mu (e^{\mu (\cdot)} \star u) - \sgn(\Real(\mu))u.
\]
\end{lemma}
\begin{proof}
This is a straightforward computation.
\end{proof}

\begin{lemma} \label{lem conv pol growth}
Let $u\colon I \to \C$ be a smooth function, and suppose that $u$ and all its derivatives grow at most polynomially. Then, for every $\mu$ as in \cref{def conv emu}, $e^{\mu (\cdot)} \star u$ and all its derivatives grow at most polynomially.
%\ajr{OR do you mean as in Lemma 7.3??}
\end{lemma}
\begin{proof}
%\ajr{But what if \(u(t)\) is `white noise' (formal derivative of Wiener process) then we must not consider any of its derivatives??  Is this a problem?}
First of all, because $(e^{\mu (\cdot)} \star u)^{(k)} = e^{\mu(\cdot)}\star(u^{(k)})$,  it is enough to consider the function $u$ itself rather than all its derivatives.

Let $C > 0$ and $n \in \N_0$ be such that for all $t \in I$,  $|u(t)| \leq C(1+ |t|^{n})$. We prove by induction on $n$ that there is a constant $C' > 0$ such that for all $t \in I$,  $|(e^{\mu (\cdot)} \star u)(t)| \leq C'(1+| t|^{n})$. We consider the case where $\Re(\mu) < -\tilde \mu$; the case where $\Re(\mu) > \tilde \mu$ is similar.

If $n = 0$, then for all $t \in I$,
\[
\bigl| (e^{\mu (\cdot)} \star u)(t) \bigr| \leq 2C \int_{-\infty}^{t} e^{\mu(t-\tau)}\, d\tau = \frac{-2C}{\mu}.
\]
Suppose that the claim holds for $n$, and suppose that $|u(t)| \leq C(1+ |t|^{n+1})$ for a constant $C$. Using integration by parts, we find that
\begin{multline*}
\bigl| (e^{\mu (\cdot)} \star u)(t) \bigr| 
\leq C \int_{-\infty}^{t} e^{\mu(t-\tau)}(1+|\tau|^{n+1})\, d\tau 
\\
= \frac{-C}{\mu}
\Bigl(1+|t|^{n+1} - (n+1) \int_{-\infty}^{t} e^{\mu(t-\tau)}\sgn(\tau)|\tau|^{n}\, d\tau\Bigr),
\end{multline*}
which implies the claim by the induction hypothesis.
\end{proof}

Let $q \in \N_0^{\infty}$, with $|q|=p$.
Consider the differentiable maps
\[
a^q_j\colon I \to \C
\quad\text{such that}\quad
a^q(t) = \sum_{j=1}^{\infty}a^q_j(t) e_j,
\]
where the sum converges in $V_1$, uniformly and differentiably in $t$ in compact sets in $I$.

For $q \in \N_0^{\infty}$ with $|q|=p$, let $J^q \subset \N$ be the set of $j \in \N$ such that either
\begin{itemize}
\item $j \in J_c$ and either $q^s = 0$ and $q^u \not=0$, or  $q^u = 0$ and $q^s \not=0$;
\item $j \in J_s$ and $q^s = 0$; or
\item $j \in J_u$ and   $q^u = 0$.
\end{itemize}

For $q \in \N_0^{\infty}$ with $|q|=p$, and $j \in \N$, write 
 $\mu^q_j := \mu^{q}- \alpha_j$, with $\mu^q$ as in~\eqref{eq def nuqrs}.
\begin{lemma} \label{lem Re muqj}
If $\beta - (p+1)\alpha > \tilde \mu$ and $\gamma - (p+1)\alpha > \tilde \mu$, then for every $j \in J^q$,  $|\Re(\mu^q_j)| > \tilde \mu$.
\end{lemma}
\begin{proof}
If  $q^s = 0$ and $q^u \not=0$, and $j \in J_c$, then
\[
\Real(
\mu_j^q) = \sum_{k \in J_c} q_k \Real(\alpha_k) + \sum_{k \in J_u} q_k \Real(\alpha_k) - \Real(\alpha_j) \geq  \gamma-(p+1) \alpha >\tilde \mu.
\]
If  $q^u = 0$ and $q^s \not=0$, and $j \in J_c$, then
\[
\Real(
\mu_j^q) =
\sum_{k \in J_c} q_k \Real(\alpha_k) + \sum_{k \in J_s} q_k \Real(\alpha_k) - \Real(\alpha_j) 
\leq
-\beta  + (p+1) \alpha
 < -\tilde \mu.
\]
If $q^s = 0$, and $j \in J_s$, then
\[
\Real(
\mu_j^q) = \sum_{k \in J_c} q_k \Real(\alpha_k) + \sum_{k \in J_u} q_k \Real(\alpha_k) - \Real(\alpha_j) \geq \beta - p\alpha > \tilde \mu.
\]
And if $q^u = 0$, and $j \in J_u$, then
\[
\Real(
\mu_j^q) = \sum_{k \in J_c} q_k \Real(\alpha_k) + \sum_{k \in J_s} q_k \Real(\alpha_k) - \Real(\alpha_j) \leq -\gamma + p\alpha < \tilde \mu.
\]
\end{proof}

\subsection{Update terms for $\xi_p$ and $F_p$} \label{sec update}

Suppose that $R_p$ has polynomial growth. Then the functions $a^q_j$ and their derivatives grow at most polynomially, uniformly in $q$ and $j$.

For every  $q \in \N_0^{\infty}$ with $|q|=p$ and $j \in \N$, consider the \ODE\ for \(\hat \xi^{q}_j\) and~\(\hat F^{q}_j\)
\beq{eq ODE j}
\hat F^{q}_j + (\hat \xi^{q}_j)'  + \mu^q_j \hat \xi^{q}_j = a^{q}_j.
\eeq
Define the maps
$
\hat \xi^{q}_j, \hat F^{q}_j \colon I \to \C
$
as follows. If $j \in J^q$, then
\beq{eq hat psi hat F Jq}
\hat \xi^{q}_j = 
\sgn(\Real(\mu_j^q)) e^{-\mu_j^q (\cdot)}\star a^q_j 
\quad\text{and}\quad
\hat F^{q}_j = 0.
\eeq
This definition makes sense because of \cref{lem Re muqj} and the growth behaviour of the functions $a^q_j$.
 If $j \not\in J^q$, then we set
\beq{eq hat psi hat F Jqc}
\hat \xi^{q}_j = 0 
\quad\text{and}\quad
\hat F^{q}_j = a^q_j.
\eeq
\begin{lemma} \label{lem ODE j}
With the above definitions, the \ODE\ \eqref{eq ODE j} is satisfied for all $q$ and $j$.
\end{lemma}
\begin{proof}
If $j \not\in J^q$, the statement is immediate from the definitions. If $j \in J^q$, it 
 follows from \cref{lem conv}.
\end{proof}

\begin{lemma} \label{lem conv hat psi hat F}
Suppose that $\beta - (p+1)\alpha > \tilde \mu$ and $\gamma - (p+1)\alpha > \tilde \mu$.
%, and 
%that  for all 
%$q \in \N_0^{\infty}$ such that $|q|=p$, and all $j \in \N$, the function $a^q_j$ grows at most polynomially.
%
% there is a $C>0$ such that for all 
%$q \in \N_0^{\infty}$ such that $|q|=p$ and all $t \in I$,
%\[
%|a^q(t)|\leq C e^{\tilde \mu |t|}.
%\]
Then the  sums
\beq{eq sums Fq psiq 2}
\hat F(t,v) =
\sum_{q \in \N_0^{\infty};\, |q| = p} \sum_{j \in \N} \hat F^{q}_j(t) e_j v^q 
\quad\text{and}\quad
\hat \xi(t,v) =
\sum_{q \in \N_0^{\infty};\, |q| = p}  \sum_{j \in \N} \hat \xi^{q}_j(t) e_j v^q
\eeq
converge in $\Pol^p(V_{\infty})$, differentiably in $t$.
\end{lemma}
\begin{proof}
The first sum in \eqref{eq sums Fq psiq 2} equals
\beq{eq sum Fq}
\sum_{q \in \N_0^{\infty};\, |q| = p} \sum_{j \in \N \setminus J^q} a^{q}_j(t) e_j v^q.
\eeq
Since $\{V_{k}\}_{k=1}^{\infty}$ is comparable to a nested sequence of separable Hilbert spaces in which the set $\{e_j\}_{j=1}^{\infty}$ is orthogonal, \cref{cor comp conv} implies that this series converges in $\Pol^p(V_{\infty})$.

Write
$
J^q_{\pm} := \{j \in J^q : \pm \Re(\mu^q_j) > \tilde \mu\}. 
$
\cref{lem Re muqj} states that $J^q = J^q_+ \cup J^q_-$.
So the second sum in \eqref{eq sums Fq psiq 2} equals
\begin{multline}
\sum_{q \in \N_0^{\infty};\, |q| = p}  \sum_{j \in J^q_+}
\int_{-\infty}^t e^{-\mu^q_j (t-\tau)} a^q_j(\tau) \, d\tau
\, e_j v^q 
\\
+ \sum_{q \in \N_0^{\infty};\, |q| = p}  \sum_{j \in J^q_-}
\int_{t}^{\infty} e^{-\mu^q_j (t-\tau)} a^q_j(\tau) \, d\tau
\, e_j v^q.
\label{eq sums Jpm}
\end{multline}
Tonelli's theorem implies that convergence of the first of these sums is equivalent to convergence of
\beq{eq int J+}
\int_{-\infty}^t
 e^{-\mu^q_j (t-\tau)}
 \Bigl(
\sum_{q \in \N_0^{\infty};\, |q| = p}  \sum_{j \in J^q_+}
 a^q_j(\tau) e_j v^q \Bigr) \, d\tau
\eeq
The sum inside the brackets converges in 
$\Pol^p(V_{\infty})$, uniformly in $\tau$, by convergence of \eqref{eq aqrs} and \cref{cor comp conv}. Since the functions $a^q_j$ grow at most polynomially, uniformly in $q$ and $j$,
the value of that sum grows at most polynomially as well, when viewed as a convergent series in $\Pol^p(V_k, V_l)$.
So the integral over $\tau$ converges in $\Pol^p(V_{\infty})$, by completeness of the spaces $\Pol^p(V_k, V_l)$.
By continuity of \eqref{eq int J+} in $t$, the convergence is uniform in $t$ on compact subsets of $I$. The derivatives of \eqref{eq int J+} with respect to $t$ are linear combinations of \eqref{eq int J+} and 
\[
\sum_{q \in \N_0^{\infty};\, |q| = p}  \sum_{j \in J^q_+}
 a^q_j(t) e_j v^q
\]
and therefore converge as well.

By an analogous argument, the second sum in \eqref{eq sums Jpm} converges as well, differentiably in $t$.
\end{proof}

\cref{prop csu} follows from \cref{lem ODE j,lem conv hat psi hat F}. 

\subsection{Proof of \cref{thm normal form}} \label{sec proof main}

\begin{lemma} \label{lem pol growth}
If $f, \xi_p, F_p, \Phi_p$ and $R_{p}$ have polynomial growth, then so do $\xi_{p+1}$,  $F_{p+1}$, $\Phi_{p+1}$ and~$R_{p+1}$.
\end{lemma}
\begin{proof}
If $R_p$ has polynomial growth, then the functions $a^q_j$ and all their derivatives grow at most polynomially, uniformly in $q$ and $j$. Hence, by \eqref{eq hat psi hat F Jq} and \eqref{eq hat psi hat F Jqc}, the map $\hat F$ has polynomial growth. By \cref{lem conv pol growth}, the same is true for $\hat \xi$. So $F_{p+1}$ and $\xi_{p+1}$ have polynomial growth.

Polynomial growth is preserved under composition and derivatives in the $I$ and $V_{\infty}$ directions. Hence the map $\hat \Phi$ as in \eqref{eq def hat Phi} has polynomial growth, and therefore so do $\Phi_{p+1}$ and $R_{p+1}$.
\end{proof}

Combining \cref{lem pol growth} with \cref{prop coord transf,prop csu}, we 
%can prove \cref{thm normal form}. We will in fact 
prove the following slightly more explicit version of \cref{thm normal form}.
\begin{theorem} \label{thm normal form 2}
Let $p \in \N$ be such that $p \geq 2$, 
$\beta - (p+1)\alpha > \tilde \mu$ and $\gamma - (p+1)\alpha > \tilde \mu$. Suppose that 
$f$ has polynomial growth.
%for every $v \in V_{\infty}$ and every $k \in \N$, there is a $C_{v, k} >0$ such that for all $t \in I$,
%\beq{eq f t}
%\|f(t,v)\|_{V_k} \leq C_{v, k} e^{\tilde \mu |t|}.
%\eeq
%and
%\[
%\|f(t,v)\|_{V_{\infty}} \leq C_v e^{\tilde \mu |t|},
%\]
%if $v \in V_{2^{p+1}}$.
%(\Todo: dependence of $C_v$ on $v$? Is ok because in the end things will be polynomial in $V$? Need to include other norms as well?)
%
%
%Suppose that for all $q \in \N_0^{\infty}$ with $|q| \leq p$, we have $a^q(t) = \cO(e^{\tilde \mu|t|})$ as $t \to \infty$ and $t \to -\infty$.
Then there are infinitely differentiable maps
\[
 F_p, \xi_p, R_p\colon I \times V_{\infty} \to V_{\infty},
\]
where $R_p = \cO(p)$,
%such that if $X$ and $x$ are as in \eqref{eq ODE XYZ p} and \eqref{eq def xyz p}, then
%\[
%\dot x(t) = Ax(t) + f(t, x(t)) + R_p(t, X(t))
%\]
%for all $t \in I$.
%
%
%
%Then there are polynomial maps
%\[
%F_p, \xi_p\colon I \times V_{\infty} \to V_{\infty}
%\]
%of degrees at most $p-1$, 
%
$F_p$ is a polynomial map that separates invariant subspaces,
 $\xi_p$ is a near-identity and $\xi_p - \id$ and $F_p$ are compact polynomials of orders at most $p-1$, such that if $X$ and $x$ are as in \eqref{eq ODE XYZ p} and \eqref{eq def xyz p}, then \eqref{eq ODE xyz} holds.
%
%uniformly in $t$ on compact subsets of $I$.
%
%Furthermore, we may choose $F_p$ such that is is a polyn
%
% its components in $V_c$, $V_s$ and $V_u$ are of the forms
%\[
%\begin{split}
%(F_p)_c (t,v)&= \sum_{
%\substack{
%q \in \N_0^{\infty};\, |q| \leq p \\
%\text{and $q^s = q^u=0$} \\
%\text{or   $q^s \not=0$ and $q^u \not= 0$}
%}
%}
%F^q(t) v^q; \\
%(F_p)_s (t,v)&= \sum_{
%\substack{q \in \N_0^{\infty};\, |q| \leq p \\ \text{ and $q^s \not=0$}}
%} F^q(t) v^q; \\
%(F_p)_u (t,v)&= \sum_{\substack{q \in \N_0^{\infty};\, |q| \leq p \\ \text{ and $q^u \not=0$}}} F^q(t) v^q,
%\end{split}
%\]
%for smooth maps $F^q \colon I \to V_{\infty}$, where the series converge in $\Pol(V_{\infty})$, diferentiably in $t$. 
Finally, there is a construction of the map $\xi_p$ in which all integrals over $I$ that occur are $\tilde \mu$-regular. 
\end{theorem}
\begin{proof}
We use induction on $p$ to prove that the claim holds for every $p$, including the auxiliary statement that 
$\xi_p, F_p, \Phi_p$ and $R_{p}$ have polynomial growth. 

%$R_{p'}$ is $p+1$ times compactly differentiable as a map from $I \times V_{2^{p'}}$ to $V_0$, and
% for every $v \in V_{\infty}$ and every $k \in \N$, there is a $C_{v, k} >0$ such that for all $t \in I$,
%\beq{eq Rp t}
%\|R_{p}(t,v)\|_{V_k} \leq C_{v,k} e^{\tilde \mu |t|}.
%\eeq
%We use the notation $p'$ for the induction parameter to distinguish it from the fixed degree of differentiability of $R_{p'}$ (namely $p+1$).

If $p=2$, then we may take $F_{p}(t,v) = 0$ and $\xi_{p}(t,v) = v$ for all $t \in I$ and $v \in V_0$. Then $R_2 = f$, so $\xi_p, F_p, \Phi_p$ and $R_{p}$ have polynomial growth because $f$ does. 
%The other required properties follow from the corresponding properties of $f$.

The induction step follows from \cref{lem pol growth,prop coord transf,prop csu}. 
\end{proof}

\section{Dynamics of the normal form equation} \label{sec dynamics}

It remains to prove \cref{lem Vj invar} and \cref{prop dynamics}, which we use to justify \cref{def centre mfd} based on \cref{thm normal form}. Throughout this section, we suppose that $F \colon I \times V_{\infty} \to V_{\infty}$ is a smooth map that separates invariant subspaces.

\begin{proof}[Proof of \cref{lem Vj invar}.]
First, suppose that $a=c$. For all $v \in V_c$ and all $q \in \N_0^{\infty}$ with $|q| \leq p$ and $q^s \not=0$ or $q^u \not=0$, we have $v^q = 0$. So the properties \eqref{eq Fcsu} of the map~$F$ imply that
$F(I \times V_c) \subset V_c$.
This, in turn, implies that for all maps $X\colon I \to V_{\infty}$ satisfying \eqref{eq ODE XYZ p}, if $X(t) \in V_c$ for a given~$t$ then $\dot X(t) \in V_c$. So $X(t) \in V_c$ for all $t \in I$. 

Next, suppose that $a=s$. If $v \in V_s$ and $q \in \N_0^{\infty}$, then $v^q=0$ if $q^u \not=0$. So $F_u(t,v)=0$ for all $t \in I$. And the components of $F_c(t,v)$ with $q^u \not=0$ are zero for the same reason, while its components with $q^s = 0$ are  zero since $v \in V_s$. Hence $F_c(t,v)=0$.
We conclude that $F(I \times V_s) \subset V_s$. As in the case $a=c$, this implies the claim for $a=s$. 

The argument for $a=u$ is entirely analogous to the case $a=s$.
\end{proof}

To prove \cref{prop dynamics}, we start with a general comparison estimate for solutions of  \ode{}s in Hilbert spaces.
\begin{lemma} \label{lem decay to Vc}
Let $V$ be a Hilbert space, $W \subset V$ a subspace, $I$ an open interval containing $0$, and $g$ a map from $I \times V$ into the space of linear operators from $W$ to $V$. Let $X \colon I \to W$ be a differentiable map (as a map into $V$), such that for all $t \in I$,
\[
\dot X(t) = g(t, X(t))X(t).
\]
%Then for all $t \in I$ with $t\geq0$,
%\[
%\|X(t)\|_V \leq \|X(0)\|_V. 
%\]
If $\zeta \in \R$ is such that $g(t,w)+\zeta$ is negative semidefinite for all $t \in I$ and $w \in W$, then for all $t \in I$ with $t\geq 0$, 
\[
\|X(t)\|_V \leq \|X(0)\|_V e^{-\zeta t} 
\]
%
%with $t>0$,
%\[
%\|X(t)\|_V \leq \|X(0)\|_V e^{-at}.
%\]
%And for all $t \in I$ with $t>0$,
%\[
%\|X(t)\|_V \geq \|X(0)\|_V e^{-at}.
%\]
\end{lemma}
\begin{proof}
First, suppose that $\zeta = 0$. Then
 for all $t\in I$,
\[
\frac{d}{dt} \|X(t)\|_V^2 = 2\Re(\dot X(t), X(t))_V = 2\Re (g(t, X(t))X(t), X(t))_V \leq 0.
\]
So $\|X\|_V^2$ is a nonnegative, non-increasing function on $I$, and the  claim for $\zeta = 0$ follows.

Next, let $\zeta \in \R$ be arbitrary. Then
\[
\frac{d}{dt}(X(t)e^{\zeta t}) = \bigl(g(t, X(t))+\zeta\bigr)X(t)e^{\zeta t}.
\]
Applying the  claim for $\zeta = 0$, with $X(t)$ replaced by $X(t)e^{\zeta t}$ and $g(t,w)$ by $g(t,w)+\zeta$, now yields the claim for $\zeta$.
%
%Hence, by the first claim, 
%\[
%\|X(t)e^{at}\|_V \leq \|X(0)\|_V,
%\]
%for all $t \in I$ with $t>0$.
\end{proof}

For any homogeneous polynomial map $p = p_{\lambda}$ between normed vector spaces $V$ and $W$, where $\lambda \in S\cB^n(V,W)$, 
define the map
$
\tilde p\colon V \to \cB(V,W)
$
by
\beq{eq def tilde p}
\tilde p(v_1)v_2 := \lambda(v_1, \ldots, v_1, v_2).
\eeq
Here $n-1$ copies of $v_1$ are inserted into $\lambda$ on the right hand side. 

For all $t \in I$, the map $F(t, \relbar)$  lies in $\Pol(V_{k}, V_1)$ for some $k$. The operator $A$ lies in $\cB(V_l, V_1)$ for some $l$. By replacing the smaller of $k$ or $l$ by the larger of these two numbers, we henceforth assume $k=l$. Applying the construction \cref{eq def tilde p} to each homogeneous term of $F(t, \relbar)$ and adding the resulting maps, we obtain a map $\tilde F\colon V_k \to \cB(V_k, V_1)$, such that for all $v \in V_k$,
\[
F(t,v) = \tilde F(t,v)v.
\]
For $a \in \{c,s,u\}$, we write $\tilde F_a$ for $\tilde F$ composed with orthogonal projection onto $V_a$.

For $j \in \N$,
let $q^{(j)} \in \N_0^{\infty}$ be defined by $q^{(j)}_m = 1$ if $m=j$, and $q^{(j)}_m = 0$ otherwise.
\begin{lemma} \label{lem tilFp}
Let $v \in V_k$. Write $v = v_c + v_s+v_u$, where $v_a \in V_a$ for $a  \in \{ c,s,u\}$. Then for all $t \in I$,
the components of $F(t,v)$ in $V_s$, $V_u$ and $V_c$ satisfy
\begin{align}
(F(t,v))_s &= \tilde F(t,v)v_s; \label{eq tilFp s}\\
(F(t,v))_u &= \tilde F(t,v)v_u; \label{eq tilFp u}\\
(F(t,v))_c &= \tilde F(t,v_c)v_c \quad \text{if $v_s = 0$ or $v_u = 0$.}\label{eq tilFp c}
\end{align}
\end{lemma}
\begin{proof}
Let $t \in I$ and $v \in V_k$.
To prove \eqref{eq tilFp s}, we use the fact that by \cref{eq sep invar s},
\[
F_s (t,v)= 
\sum_{j \in J_s} \,\, \sum_{q \in \N_0^{\infty}: |q| \leq p-1} F^{q+ q^{(j)}}(t) v^q v_j.
\]
So
\[
\tilde F_s(t,v) = \sum_{j \in J_s} \,\,  \sum_{q \in \N_0^{\infty}: |q| \leq p-1} F^{q+ q^{(j)}}(t) v^q e^j,
\]
where $\{e^j\}_{j \in \N}$ is the basis of $V_1^*$ dual to $\{e_j\}_{j \in \N}$. Hence
\[
\tilde F_s(t,v)v = \tilde F(t,v) v_s, 
\]
which implies \eqref{eq tilFp s}. The equality \eqref{eq tilFp u} can be proved anaogously.

To prove \eqref{eq tilFp c}, we note that by \cref{eq sep invar c},
\[
F_c (t,v)= F_{c,1}(t,v) + F_{c,2}(t,v),
\]
where
\begin{align}
F_{c,1}(t,v)&= \sum_{q \in \N_0^{\infty}: |q| \leq p, 
q^s = q^u=0} F^q(t) v^q; \label{eq Fpc1}\\
F_{c,2}(t,v)&= \sum_{q \in \N_0^{\infty}: |q| \leq p,
q^s \not= 0 \not= q^u} F^q(t) v^q. \label{eq Fpc2}
\end{align}
The right hand side of \eqref{eq Fpc1} only depends on $v_c$, and the right hand side of \eqref{eq Fpc2}
is zero if $v_s = 0$ or $v_u = 0$. So, under that condition,
\(
F_c (t,v) = F_{c}(t,v_c) = \tilde F(t,v_c)v_c.
\)
\end{proof}

\begin{proof}[Proof of \cref{prop dynamics}.]
Set
\beq{eq Dmu}
D_{\tilde \mu} := \bigl\{(t,v) \in I \times V_{\infty} : \|\tilde F(t,v)\|_{\cB(V_k, V_1)} < \tilde \mu\bigr\}.
\eeq
Because $F$ is a sum of polynomials of degrees at least two, we have $\tilde F(t,0) = 0$ for all $t$. So $D_{\tilde \mu}$ contains $I \times \{0\}$. It is open by continuity of $\tilde F$.

Let $X\colon I \to V_{\infty}$ be a solution of the constructed system~\eqref{eq ODE XYZ p}.
 As in the proof of \cref{lem Vj invar}, 
\beq{eq X dot dynamics}
\dot X_s(t) = AX(t) + F_s(t, X(t)) = \bigl(A + \tilde F(t, X(t) \bigr) X_s(t),
\eeq
where we used the first equality in \cref{lem tilFp} and the fact that $A$ preserves $V_s$.
For all $(t,v)\in D_{\tilde \mu} $, the operator 
\[
A + \tilde F_s(t, v) + \beta - \tilde \mu   \colon V_k \cap V_s \to V_1 \cap V_s
\]
is negative semidefinite. Hence the claim about $X_s$ follows from the second part of \cref{lem decay to Vc}. The claim about  $X_u$ can be proved similarly, via a version of \cref{lem decay to Vc} for positive-definite operators. 

Next, suppose that  $X_s(0) = 0$ or $X_u(0) = 0$. By \cref{lem Vj invar}, either $X_s(t) = 0$ for all $t \in I$ or $X_u(t) = 0$ for all $t \in I$. Similarly to \eqref{eq X dot dynamics}, the third equality in \cref{lem tilFp} implies that 
\[
\dot X_c(t) =  \bigl(A + \tilde F(t, X_c(t) \bigr) X_c(t),
\]
for all $t \in I$. And for all $(t,v)\in D_{\tilde \mu}$, the operator 
\[
A + \tilde F_c(t, v) - \alpha - \tilde \mu   \colon V_k \cap V_c \to V_1 \cap V_c
\]
is negative semidefinite. So by \cref{lem decay to Vc}, 
\[
\|X_c(t)\|_{V_1} \leq e^{(\alpha + \tilde \mu)t}\|X_c(0)\|_{V_1}
\]
for all $t\geq0$ in $I$. It similarly follows that for all $t\leq0$ in $I$,
\[
\|X_c(t)\|_{V_1} \leq e^{-(\alpha + \tilde \mu)t} \|X_c(0)\|_{V_1}.
\]
\end{proof}

\section{Example: a non-autonomous version of Burgers' equation} \label{sec ex}

Let $r \in \R$, and
consider the non-autononomous, nonlinear \PDE
%\ajr{Change space variable to \(\theta\) to avoid symbol conflict with dynamical \(x\).}
\beq{eq modified Burgers}
\partial_t u (t, \theta)= \partial_\theta^2 u(t,\theta) + ru(t,\theta) - \frac{t}{2}(\partial_\theta u (t,\theta))^2,
\eeq
with $2\pi$-periodic boundary conditions in $\theta$.
%\ajr{Think better to use BCs \(u_\theta=0\) at \(\theta=0,\pi\) so that \(e_j=\cos j\theta\) and eigenvalues are all multiplicity one.}
 Then \cref{thm special case} applies, where $\Omega$ is the circle. 
%We work out the centre manifold part of the coordinate transform from \cref{thm normal form} in \cref{sec ex thm}. We also compute this transform directly, in \cref{sec ex direct}, and conclude that the two approaches give the same result. The first construction has the advantage that it always leads to a 
%result, whereas the second construction is simpler in this case. This illustrates \cref{rem direct constr}. 

Using \cref{thm normal form}, we compute the centre manifold of the normal form system approximating \eqref{eq modified Burgers} up to residuals of order three, in \cref{sec ex thm}.
Via a direct approach,  we compute all invariant manifolds for  residuals of orders three and four, in  \cref{sec ex direct}. We find that the order three centre manifolds computed in the two ways agree. These computations illustrate \cref{rem direct constr}, that the construction from \cref{thm normal form} is guaranteed to give a result, while a direct computation may be more efficient in concrete situations.

\subsection{Centre manifold via \cref{thm normal form}} \label{sec ex thm}

In this setting,
\[
Au = u'' + ru
\quad\text{and}\quad
f(t,u) = - \frac{t}{2}(u')^2,
\]
where a prime denotes the derivative in the $\theta$-direction. The eigenfunctions of $A$ are $e_j$, for $j \in \Z$, given by $e_j(\theta) := e^{ij\theta}$. The eigenvalue corresponding to $e_j$ is $\alpha_j = r-j^2$ (which has multiplicity two when \(j\neq 0\)). 
Choose $\alpha, \beta, \gamma$ and $\tilde \mu$ such that $0 \leq \alpha < \tilde \mu < \beta = \gamma < 1$, and $\alpha < \frac{1}{2}$.
%We take $\alpha = 1/4$, $\tilde \mu = 1/2$, and $\beta = \gamma = 3/4$. 
%\ajr{Better to take $r$ to be near something definite like one, so \(n=1\).}
Suppose that $r$ lies within $\alpha$ of an integer of the form $n^2$, for a nonzero $n \in \Z$. Then  the eigenvalue $\alpha_n$ is central up to precision $\alpha$. 

We determine a corresponding centre manifold for a system that approximates~\eqref{eq modified Burgers} up to a third-order residual. This involves the coordinate transform $\xi_3$. 
%The complete expression for this map is very involved, but 
To compute this centre manifold,
we only need to apply $\xi_3$ to elements of $V_c = \Span\{ e_n,e_{-n}\}$. In other words, we only need to compute $\xi_3 (t, X_ne_n+X_{-n} e_{-n})$, for $t \in \R$ and $X_n,X_{-n} \in \C$. (We do not determine the domain $D_{\tilde \mu}$ here.)

%\ajr{I think $\xi_2\neq \id$?  Will work on it, also for \(r\approx 1\), and see what the rest of this example could become.  See next subsection as an alternative example.  May also be appropriate to discuss the centre-unstable manifold.}
For $p=2$, the map  $\xi_2$ is the identity map. 
So
\[
\xi_3 (t, X_ne_n+X_{-n}e_{-n}) = X_ne_n+X_{-n}e_{-n} + \hat \xi(X_ne_n + X_{-n}e_{-n}),
\]
where
\[
\hat \xi(X_ne_n+ X_{-n}e_{-n}) = \sum_{q \in \Z^{\infty};\, |q|=2} \sum_{j \in \Z} \hat \xi^q_j(t) e_j (X_ne_n + X_{-n}e_{-n})^q.
\]
For $j \in \Z$,
let $q^{(j)} \in \Z^{\infty}$ be defined by $q^{(j)}_m = 1$ if $m=j$, and $q^{(j)}_m = 0$ otherwise. Then, for $q \in \Z^{\infty}$ with $|q|=2$,
\[
(X_ne_n + X_{-n}e_{-n})^q = 
\begin{cases}
X_n^2 & \text{if }q = 2q^{(n)}; \\
X_{-n}^2 & \text{if }q = 2q^{(-n)}; \\
X_nX_{-n} & \text{if }q = q^{(n)} + q^{(-n)}\\
0 & \text{otherwise}.
\end{cases}
\]
So
\[
\hat \xi(X_ne_n+ X_{-n}e_{-n}) =  \sum_{j \in \Z}
\Bigl(X_n^2
 \hat \xi^{2q^{(n)}}_j(t)  + X_{-n}^2  \hat \xi^{2q^{(-n)}}_j(t) + X_nX_{-n} \hat \xi^{q^{(n)} +  q^{(-n)}}_j(t) 
 \Bigr)
 e_j. 
\]

The map $\hat \xi^{2q^{(n)}}_j$ is expressed in terms of the map $a^{2q^{(n)}}_j$ in 
\[
R_2(t,u) = -\sum_{j \in \Z} \sum_{q \in \Z^{\infty};\, |q|= 2} a^{q}_j(t) e_ju^q.
\]
(The order three term in \eqref{eq aqrs} now equals zero.) See \eqref{eq hat psi hat F Jq} and \eqref{eq hat psi hat F Jqc}. If $u = \sum_{l \in \Z} x_l e_l$, then
\[
R_2(t,u) = -\frac{t}{2}(u')^2 =- \frac{t}{2}\sum_{j \in \Z} \Bigl( \sum_{k \in \Z} k(j-k)x_k x_{j-k} \Bigr) e_j.
\]
The equality
$
x_k x_{j-k} = u^{2q^{(n)}} = x_n^2
$
holds precisely if $k = n$ and $j=2n$. Hence
\[
a^{2q^{(n)}}_{2n}(t) = \frac{t}{2} n^2
\]
and $a^{2q^{(n)}}_{j} = 0$ if $j \not= 2n$. An analogous argument shows that
\[
a^{2q^{(-n)}}_{-2n}(t) = \frac{t}{2} n^2
\]
and $a^{2q^{(-n)}}_{j} = 0$ if $j \not= -2n$.
The equality
$
x_k x_{j-k} = u^{q^{(n)} + q^{(-n)}} = x_n x_{-n}
$
holds precisely if $j=0$ and either $k = n$ or $k=-n$. Hence
\[
a^{q^{(n)} + q^{(-n)}}_{0}(t) = -t n^2
\]
and $a^{q^{(n)} + q^{(-n)}}_{j} = 0$ if $j \not= 0$.

The relevant numbers $\mu^q_j$  as in \cref{sec comp Fq} equal
\[
\begin{split}
\mu^{2q^{(n)}}_{2n} &= 2 \alpha_n - \alpha_{2n} = r+2n^2;\\
\mu^{2q^{(-n)}}_{-2n} &= 2 \alpha_{-n} - \alpha_{-2n} = r+2n^2; \\
\mu^{q^{(n)} + q^{(-n)}}_{0} &= \alpha_n + \alpha_{-n} - \alpha_0 = r-2n^2 
\end{split}
\]
(note that $\alpha_j =\alpha_{-j}$ for every $j$).
Because $n^2 \geq 1$ and $\alpha < \frac{1}{2}$,  the real parts of $\mu^{2q^{(n)}}_{2n}$ and $\mu^{2q^{(-n)}}_{2n} $ are 
 greater than $\alpha$, whereas the real part of $\mu^{q^{(n)} + q^{(-n)}}_{0}$ is smaller than $-\alpha$.

And with $J^q$ as in  \cref{sec comp Fq}, we have $2n \in J^{2q^{(n)}}$. Indeed,
\[
J_s = \{j \in \Z : |j| \geq n+1 \},
\]
so $2n \in J_s$ and $(2q^{(n)})^s = 0$. Similarly, $2n \in J^{2q^{(-n)}}$. And $0 \in J_u$ and $(q^{(n)} + q^{(-n)})^u = 0$, so $0 \in J^{q^{(n)} + q^{(-n)}}$.

Hence, by \eqref{eq hat psi hat F Jq},
\[
\begin{split}
\hat \xi^{2q^{(n)}}_{2n}(t) &= \int_{-\infty}^t e^{-(r+2n^2)(t-\tau)} \frac{\tau}{2} n^2 \, d\tau \\
&= \frac{n^2}{2(r+2n^2)}  \Bigl(t-  \frac{1}{r+2n^2}\Bigr). 
\end{split}
\]
The integral converges since $\Re(r+2n^2)>0$, and is $\tilde \mu$-regular.
Similarly,
\[
\hat \xi^{2q^{(-n)}}_{-2n}(t) = \frac{n^2}{2(r+2n^2)}  \Bigl(t-  \frac{1}{r+2n^2}\Bigr). 
\]
And because $\Re(\mu^{q^{(n)} + q^{(-n)}}_{0}) < -\alpha$,
\[
\begin{split}
\hat \xi^{q^{(n)}+ q^{(-n)}}_{0}(t) &= -\int_t^{\infty} e^{-(r-2n^2)(t-\tau)} ({-\tau} n^2) \, d\tau \\
&= \frac{-n^2}{r-2n^2}  \Bigl(t-  \frac{1}{r-2n^2}\Bigr). 
\end{split}
\]

We conclude that for all $t \in \R$ and $X_n, X_{-n} \in \C$,
\begin{multline} \label{eq xi3 Burgers centre}
\xi_3(t,X_ne_n + X_{-n}e_{-n}) = \\
X_ne_n +X_{-n}e_{-n} +  \frac{n^2}{2(r+2n^2)}  \Bigl(t-  \frac{1}{r+2n^2}\Bigr)(X_n^2e_{2n}+ X_{-n}^2e_{-2n}) \\
-
\frac{n^2}{r-2n^2}  \Bigl(t-  \frac{1}{r-2n^2}\Bigr) X_nX_{-n}.
\end{multline}
(The last term is a scalar multiple of the constant function $e_0$.)
If $r = n^2$, this simplifies to 
\begin{multline*}
\xi_3(t,X_ne_n + X_{-n}e_{-n}) = \\
X_ne_n+X_{-n}e_{-n} +  \frac{1}{6}  \Bigl(t-  \frac{1}{3n^2}\Bigr)(X_n^2 e_{2n}+X_{-n}^2e_{-2n}) + \Bigl(t+  \frac{1}{n^2}\Bigr) X_nX_{-n}.
\end{multline*}

\subsection{Invariant manifolds via direct computations} \label{sec ex direct}

For order of residual $p=2$, the map~$\xi_2$ is the identity map, \(x_j=X_j\)\,. 

Proceeding to order of residual \(p=3\) we construct quadratic corrections to the identity~\(\xi_2\) to form~\(\xi_3\).
In the eigenvector basis the field \(u(t,\theta)=\sum_{j} x_j(t)e^{i j\theta}\) (all sums in this section are over~\(\Z\)), and the \pde~\eqref{eq modified Burgers} becomes
\begin{equation}
\dot x_j=\alpha_jx_j +\frac t{2}\sum_{k} b_{jk}x_{j-k}x_k 
\quad\text{where }b_{jk}:=k(j-k).
\label{eqBurgPODEs}
\end{equation}
Writing
\[
x_j(t) = \xi_3(t, X(t))_j = X_j + \sum_{k,l \in \Z}g^{kl}_j(t)  X_k X_l,
\]
and solving\begin{journal}\footnote{The computer algebra code used for the computations in this section is available on \url{http://www.maths.adelaide.edu.au/anthony.roberts/pBurgers.txt}.}
\end{journal}
 for $g^{kl}_j$ such that $x_j$ satisfies \eqref{eqBurgPODEs} up to terms of order three if $\dot X_k = \alpha_k X_k$, we find that \(x(t)=\xi_3(t,X(t))\) is given by 
\begin{align}&
x_j=X_j+\frac12\sum_{k:|d_{jk}^{-1}|>\tilde\mu} b_{jk}[d_{jk}t-d_{jk}^2]X_{j-k}X_k\,,
\label{eqBurgP3}
\\&
\text{where }d_{jk}:=1/[-\alpha_j+\alpha_k+\alpha_{j-k}]=1/[r+2jk-k^2].
\nonumber
\end{align}
For \(r\approx n^2\) and odd~\(n\), the denominators in~\(d_{jk}\) are not small.
Then this map, combined with the linear \(\dot X_j=\alpha_j X_j\)\,, matches the \pde~\eqref{eq modified Burgers} to third-order errors.

However, for \(r\approx n^2\) and even~\(n>0\), some denominators are small, becoming zero when \(r=n^2\).  Then the divisor being zero becomes \(k(k-j)=n^2/2\) and hence has zeros for every pair of integer factors of~\(n^2/2\) (including negative pairs). 
Consequently these terms are excluded from the sum~\eqref{eqBurgP3}, and instead lead to nonlinearly modifying the evolution for some~\(j\) via
\begin{equation*}
\dot X_j=\alpha_j X_j+\frac t{2}\sum_{k:|d_{jk}^{-1}|<\tilde\mu} b_{jk}X_{j-k}X_k\,.
\end{equation*}

Often the centre manifold is of most interest, so in~\(\xi_3\) setting all \(X_j=0\) except~\(X_{\pm n}\), gives the quadratic approximate centre manifold to be \(x_j=X_j\) for all~\(j\) except
\begin{align*}&
x_0=X_0-n^2\left[\frac1{r-2n^2}t-\frac1{(r-2n^2)^2}\right]X_nX_{-n}\,,
\\&
x_{\pm 2n}=X_{\pm 2n}+\frac12n^2\left[\frac1{r+2n^2}t-\frac1{(r+2n^2)^2}\right]X_{\pm n}^2\,.
\end{align*}
This is the same result as \eqref{eq xi3 Burgers centre}.

Proceeding to order of residual \(p=4\) we may construct cubic corrections to~\(\xi_2\) to form~\(\xi_4\).
For simplicity, restrict attention to the cases of \(n\)~odd.
It is straightforward but tedious to construct that for~\(\xi_4\)
\begin{align}&
x_j=\xi_{3,j}+\sum_{k,l:|d_{jkl}^{-1}|>\tilde\mu} b_{jl}b_{lk}c_{jkl}(t)X_{j-l}X_{l-k}X_k\,,
\label{eqBurgP4}
\\\text{where }&
c_{jkl}:=\tfrac12d_{lk}d_{jkl}t^2
-(d_{lk}d_{jkl}^2+\tfrac12d_{lk}^2d_{jkl})t
+(d_{lk}d_{jkl}^3+\tfrac12d_{lk}^2d_{jkl}^2),
\nonumber
\\&
d_{jkl}:=1/[-\alpha_j+\alpha_k+\alpha_{l-k}+\alpha_{j-l}]
=1/[2r +2jl+2kl -2k^2-2l^2].
\nonumber
%\label{eqBurgd4}
\end{align}
The terms excluded from~\((k,l)\) in the sum~\eqref{eqBurgP4} must cause cubic terms in the evolution.
For example, when \(r=n^2=1\) then\footnote{The apparent pattern in these \ode{}s becomes more complicated---at \(\dot X_{\pm6}\) for example.}
\begin{subequations}\label{eqsBurgP4n1}%
\begin{align}&
\dot X_0=X_0,
\\&
\dot X_{\pm1}=(\tfrac19t-\tfrac13t^2)X_{-1}X_{\pm1}^2\,,
\\&
\dot X_{\pm2}=-3X_{\pm2}
+(\tfrac{104}{225}t-\tfrac{8}{15}t^2)X_{\mp1}X_{\pm1}X_{\pm2}\,,
\\&
\dot X_{\pm3}=-8X_{\pm2}
+(\tfrac{594}{1225}t-\tfrac{18}{35}t^2)X_{\mp1}X_{\pm1}X_{\pm3}\,,
\\&
\dot X_{\pm4}=-15X_{\pm3}
+(\tfrac{1952}{3969}t-\tfrac{32}{63}t^2)X_{\mp1}X_{\pm1}X_{\pm3}\,,
\\&\qquad\vdots\nonumber
\end{align}
\end{subequations}
By construction, in this case of \(r=1\), the coordinate transform~\eqref{eqBurgP4} together with the \ode{}s~\eqref{eqsBurgP4n1} creates a dynamical system in~\(u(t,\theta)=\sum_j x_je^{ij\theta}\) which is the same as the \pde~\eqref{eq modified Burgers} to a residual of order four.
In the combined system~\eqref{eqBurgP4,eqsBurgP4n1}, by definition \cref{def centre mfd} three invariant manifolds are: the 1D unstable manifold parametrised by~\(X_0\) with all other \(X_j=0\);  the 2D centre manifold parametrised by~\(X_{\pm 1}\) with all other \(X_j=0\); and the stable manifold with \(X_0=X_{\pm 1}=0\)\,.

\begin{ArXiv}
See \cref{sec code Burgers} for the computer algebra code used the for the computations in this subsection. It is also available on \url{http://www.maths.adelaide.edu.au/anthony.roberts/pBurgers.txt}.
\end{ArXiv}

\subsection*{Acknowledgement}
Part of this research was supported by the Australian Research Council grant DP150102385.

\appendix

\begin{ArXiv}

\section{Compact and finite-rank operators into Banach spaces}
\label{secCompFinite}

%\subsection{Approximating compact operators}

Let $V$  and $W$ be  Banach spaces, and suppose that $V^*$ has the approximation property (this is true for example if $V$ is a Hilbert space). Let $\{e^j\}_{j\in \N} \subset V^*$ and $\{f_k\}_{k \in \N} \subset W$ be countable subsets with dense spans. (So $V^*$ and $W$ are separable.)  
%%For each $j \in \N$, let $e^j \in V^*$ be given by taking inner producs with $e_j$. 

In the main text, we use the following, which is standard in the case where $V$ and $W$ are Hilbert spaces.
\begin{proposition}\label{prop approx cpt}
The space $\Span\{ e^j \otimes f_k : j,k \in \N\}$ is dense in $\cK(V,W)$.
\end{proposition}

Let $\cF(V,W)$ be the space of finite-rank linear operators from $V$ to $W$; that is, operators whose images are finite-dimensional.
%\begin{lemma} \label{lem FK}
%The space $\cF(V,W)$ is dense in $\cK(V,W)$. 
%\end{lemma}
%\begin{proof}
%By Proposition 4.12 in \cite{Ryan} (\Todo: check reference; see Proposition 4.3 in Jan Rozendaal's notes.), the claim holds if $V$ is any Banach space such that $V^*$ has the approximation property. But, since $V^* \cong V$ is a Hilbert space, this condition is satisfied.
%\end{proof}

\begin{lemma} \label{lem span F}
The space $\Span\{ e^j \otimes f_k : j,k \in \N\}$ is dense in $\cF(V,W)$.
\end{lemma}
\begin{proof}
Let $T  \in \cF(V,W)$. Since the image of $T$ is finite-dimensional, there are $v^1, \ldots, v^n \in V^*$ and $w_1, \ldots, w_n \in W$ such that
$
T = \sum_{l=1}^n v^l \otimes w_l\,.
$
%where $v_l^* \in V^*$ is given by taking inner products with $v_l$.

Let $\varepsilon > 0$.
 For every $l$, choose $r \in \N$ and $a_l^1, \ldots, a_l^r \in \C$ and $b_l^1, \ldots, b_l^r \in \C$ such that
 \[
 \Bigl\|v^l- \sum_{j=1}^r a_l^j e^j \Bigr\|_{V^*} \leq \sqrt{\varepsilon/n}
 \quad\text{and}\quad
  \Bigl\| w_l-\sum_{k=1}^r b_l^k f_k \Bigr\|_W \leq \sqrt{\varepsilon/n}.
 \]
Using the triangle and Cauchy--Schwartz inequalities, one finds that for all $v \in V$,
\begin{multline*}
 \bigl\| Tv - \Bigl(\sum_{j,k = 1}^r a_l^j b_l^k e^j \otimes f_k\Bigr) (v) \bigr\|_W 
=
 \Bigl\| \sum_{l=1}^n \bigl\langle v^l - \sum_{j=1}^r a^j_l e^j, v \bigr\rangle 
 \Bigl(w_l -  \sum_{k=1}^r b^j_l f_k\Bigr)
 \Bigr\|_W 
 \\
 \leq \|v\|_{V} \sum_{l=1}^n \Bigl(\bigl\|v^l - \sum_{j=1}^r a^j_l e^j\bigr\|_{V^*} \cdot  \bigl\| w_l -  \sum_{k=1}^r b^j_l f_k\bigr\|_W\Bigr) 
 \leq \varepsilon \|v\|_{V}.
\end{multline*}
\end{proof}

\begin{proof}[Proof of \cref{prop approx cpt}.]
Since $V^*$ has the approximation property, $\cF(V,W)$ is dense in $\cK(V,W)$. See for example Proposition 4.12(b) in the book by Ryan~\cite{Ryan}.
So the claim follows from \cref{lem span F}.
\end{proof}

\end{ArXiv}

\begin{ArXiv}

\section{Computer algebra code for Burgers example computation}\label{sec code Burgers}
\begin{Verbatim}
Comment periodic Burgers-like, infinite-D coord transform. 
Written in Reduce, see https://reduce-algebra.sourceforge.io/
AJR, 3 Jun 2019;
on div; off allfac; on revpri; 
clear sum;

operator xx; depend xx,t;
let df(xx(~i),t)=>sub(j=i,dxxjdt);
operator a; write "a(j) for eigenvalues alpha_j";
operator b; write "b(j,k)=k*(j-k)";
let b(j,j-~k)=>b(j,k);
factor o;

operator sum; linear sum;
let df(sum(~a,~b),~t)=>sum(df(a,t),b);

operator d;
write "d(j,k)=1/(-a(j)+a(k)+a(j-k))=1/(n^2+2*j*k-2k^2)";
write "d(j,k,l)=1/(-a(j)+a(k)+a(l-k)+a(j-l))
        =1/2/(n^2+j*l+k*l-k^2-l^2)";
let { d(j,k)*a(k) => 1-d(j,k)*(-a(j)+a(j-k))
    , d(j,k,l)*a(k) => 1-d(j,k)*(-a(j)+a(l-k)+a(j-l))
    };

write "Linear approx evolution";
xj:=xx(j);
dxxjdt:=a(j)*xx(j); 

write "Quadratic solution shape";
let o^2=>0;% order of error
xj:=xj+o/2*sum(b(j,k)*(d(j,k)*t-d(j,k)^2)*xx(k)*xx(j-k),k);
resj:=-df(xj,t)+a(j)*xj
    +o*t/2*sum(b(j,l)*sub(j=l,xj)*sub(j=j-l,xj),l);
resj:=sub(l=k,resj);

write "However, zero divisors occur in d(j,k)
=1/(-a(j)+a(k)+a(j-k)) =1/(n^2+2*j*k-2k^2).  But not for odd
n.  For even n>0, find that zero divisors come from all the
factors of n^2/2 since divisor=0 rewrites as k(k-j)=n^2/2. 
Both the positive and negative factors contribute.  For
example, n=2 gives n^2/2=2 with factors 2,1 and -2,-1 and
1,2 and -1,-2.  So zero divisor for k=2,j=1 and k=-2,j=-1
and k=1,j=-1 and k=-1,j=1.";

write "Cubic manifold shape---pv. zero divisors are avoided
in the quadratic terms by perhaps requiring n be odd.";
let o^3=>0;% order of error
resj:=-df(xj,t)+a(j)*xj
    +o*t/2*sum(b(j,l)*sub(j=l,xj)*sub(j=j-l,xj),l)$
depend k,kl; depend l,kl;
resj:=(resj where sum(~~a*sum(~b,~l),~k)=>sum(a*b,kl))$
resj:=(resj where sum(~b,l)=>sum(sub(l=k,b),k))$
resj:=(resj where sum(~a,kl)=>sum(sub(l=j-l,a),kl) 
    when df(a,xx(l)) neq 0)$
write "Extract the coefficient of X(j-l)X(l-k)X(k)";
rhsjkl:=(resj where sum(~a,kl) => coeffn( coeffn( coeffn(
    a,xx(k),1) ,xx(j-l),1) ,xx(l-k),1));
write "Iterative soln of ODE for coeff---zero divisors excluded";
cjkl:=0;
for it:=1:9 do begin
  write resc:=d(j,k,l)*(rhsjkl-df(cjkl,t))-cjkl;
  cjkl:=cjkl+resc;
  if resc=0 then write "success ",it:=10000+it;
end;
cjkl:=cjkl;
write "Update the xj";
xj:=xj+sum(cjkl*xx(j-l)*xx(l-k)*xx(k),kl);

write "But some RHS resonante and so must be in the
evolution instead. Need definite r so set ",n:=1, 
"Explore coefficients to a max wavenumber of ",maxj:=10,
"The negative j cases are the same with k, l, j-l and l-k
all also changed sign.  It looks like that to get dxx(j)
correct then we need maxj=j+3, but could be different for
n>1, guess perhaps maxj=j+3*n?";
r:=n^2; 
o:=1; % remove order symbol
factor xx;
% to get substitution have to change variables!
rhsjkl:=(rhsjkl where { b(l,k)=>kk*(ll-kk)
    , b(j,l)=>ll*(jj-ll)
    , d(l,k)=>1/(r+2*ll*kk-2*kk^2) })$
% Now sum over all contributions to dXj/dt
array dxx(maxj); 
for j:=0:maxj do write dxx(j):= (r-j^2)*xx(j)+
    for k:=-maxj:maxj sum for l:=-maxj:maxj sum
    if n^2+j*l+k*l-k^2-l^2=0
    then sub({jj=j,kk=k,ll=l},rhsjkl)*xx(j-l)*xx(l-k)*xx(k)
    else 0;

end;
\end{Verbatim}

\end{ArXiv}

%\begin{journal}
\bibliographystyle{plain}
%\end{journal}

\IfFileExists{ajr.sty}
    {\bibliography{bibexport,ajr,bib}}
    {\bibliography{bibexport,bib_peter,bibexport_2}}% append your bib files here

\end{document}